\documentclass[12pt]{amsart}

\usepackage{amssymb, amsmath, amsthm, mathtools,}
\usepackage{comment}
\usepackage{indentfirst}
\usepackage[all]{xy}
\usepackage{cite}
\usepackage{subfigure}

\theoremstyle{plain}
\newtheorem{prop} {Proposition} [section]
\newtheorem{thm}[prop] {Theorem} 
\newtheorem{dfn} [prop]{Definition}
\newtheorem{exam}[prop]{Example}
\newtheorem{cor}[prop]{Corollary}
\newtheorem{lem}[prop]{Lemma}

\theoremstyle{definition}
\newtheorem{prob*}{Problem}
\newtheorem{conj}{Conjecture}  
\newtheorem{rem}[prop]{Remark}

\newcounter{num}
\newcommand{\Rnum}[1]{\setcounter{num}{#1} \Roman{num}}

\title[SPHERICAL COMPLETION]%
{An application of the spherical completion to finite-dimensional normed spaces}
\author[ISHIZUKA]%
{KOSUKE ISHIZUKA}
\address{Mathematical Institute, Graduate School of Science, Tohoku University, 6-3 Aramakiaza, Aoba, Sendai, Miyagi 980-8578, Japan.}
\date{}
\email{kosuke.ishizuka.r7@dc.tohoku.ac.jp}

\begin{document}
\begin{abstract}
In this paper, we will establish a general method of studying finite-dimensional normed spaces, and apply this method to classifying $3$-dimensional and $4$-dimensional normed spaces over a non-spherically complete field. For this purpose, we will use the spherical completion. From the perspective of the spherical completion, each finite-dimensional normed space can be embedded into a simple space. In order to study simple spaces, the orthogonality is important. The orthogonality allows us to find a classification of finite-dimensional normed spaces. As an application of our study, we can get a characterization of strictly epicompact sets, which is an open problem.
\end{abstract}

\maketitle

\section{Introduction and preliminaries}

\subsection{Introduction}
In this paper, we will study an isometry classification of finite-dimensional normed spaces over a non-spherically complete valued field. If the coefficient field is spherically complete, then each finite-dimensional normed space is a direct sum of one-dimensional spaces. Thus, the classification problem is done (see Theorem \ref{.1thm1} and Remark \ref{rem1}). On the other hand, if the coefficient field is non-spherically complete, the problem is much more complicated. \par
In \cite{open}, a hole is used to classify $2$-dimensional normed spaces. Indeed, $2$-dimensional normed spaces can be classified by the equivalence relation on holes. On the other hand, holes are complicated and it is seen that they are not suitable for classifying finite-dimensional normed spaces of dimension $n \ge 3$. Therefore, we use the spherical completion instead of holes. \par
From the perspective of the spherical completion, each finite-dimensional normed space can be embedded into a direct sum of the spherical completion of the coefficient field (Theorem \ref{pthm1}). This observation leads us to study finite-dimensional normed spaces more simply. To study the spherical completion of the coefficient field, the orthogonality is important (Theorem \ref{thmh4}). When we determine whether the orthogonality is valid or not, we often use a well-known formula
\begin{align*}
  d(\pi(x), F / D) = d(x, D), \pi : E \to E / D,
\end{align*}
where $D \subseteq F \subseteq E$ is a sequence of normed spaces. This is very useful because a quotient normed space $E / D$ has a lower dimension than $E$. \par
In section \ref{str}, we recall a strictly epicompact set, which was introduced in \cite{equ}. In \cite{equ}, Schikhof suggested a problem with the characterization of strictly epicompact sets, which is still open. By Schikhof duality, studying strictly epicompact sets is equivalent to studying the condition (SE). For the condition (SE), see Proposition \ref{propst1}. This condition is a slightly modified condition appearing in \cite[Proposition 7.3]{equ}. In studying the condition (SE), it is important to understand the structure of a subspace of a normed space and its dual space (Corollary \ref{secor1} and Lemma \ref{selem1}). \par
In sections \ref{3d} and \ref{4d}, by using the orthogonality, we will classify $3$-dimensional normed spaces into five types, type $\mathrm{\Rnum{1}}_3 \sim \mathrm{\Rnum{5}}_3$, and classify $4$-dimensional normed spaces into seventeen types, type $\mathrm{\Rnum{1}}_4 \sim \mathrm{\Rnum{17}}_4$. As an appendix, in section \ref{appe}, we summarize properties of these types in tables. To show the validity of these types, we will use the results obtained in section \ref{sys}. In section \ref{sys}, we will study a systematic approach to not only $3$-dimensional and $4$-dimensional normed spaces but also all finite-dimensional normed spaces. Our key theorem is Theorem \ref{systhm1}, which is not difficult to prove but has many applications. \par
 According to the types defined in sections \ref{3d} and \ref{4d}, we can identify all subspaces and the dual space of a $3$-dimensional or $4$-dimensional normed space. As a result, we can characterize whether $3$-dimensional and $4$-dimensional normed spaces satisfy the condition (SE) or not. The characterization is related to the norm of an element of a normed space. Therefore, in this paper, the norm of an element of a normed space does not necessarily have a value in the valuation group of the coefficient field. \par

\subsection{Preliminaries}
In this paper, $K$ is a non-archimedean non-trivially valued field which is complete under the metric induced by the valuation $|\cdot| : K \to [0,\infty)$. The unit ball of $K$ is denoted by $B_K := \{ x \in K : |x| \leq 1 \}$. We denote by $V_K$ the valuation group of $K$, that is, $V_K := \{|x| : x \in K \setminus \{0\}\}$. \par
Throughout, $(E,\|\cdot\|_E)$ is a non-archimedean Banach space over $K$. Let $a \in E$ and $r > 0$. We write $B_E(a,r)$ for the closed ball with radius $r$ about $a$, that is, $B_E(a,r) := \{ x \in E : \|x-a\|_E \leq r \}$. If $a = 0$ and $r = 1$, we write $B_E := B_E(0,1)$\par
For a subset $X \subseteq E$, we denote by $[X]_K$ the $K$-vector space generated by $X$. Also, for an element $a \in E$, we set $d_E(a,X) := \inf_{x \in X} \|a - x\|_E$. If there is no confusion possible, we abbreviate $E,K$. \par
For a closed subspace $D \subseteq E$, we always consider $(E/D,\|\cdot\|)$ as the Banach space equipped with the quotient norm, $\|\pi(x)\| := d(x,D)$ where $\pi : E \to E/D$. A subset $X \subseteq E \setminus \{0\}$ is called orthogonal if for each finite subset $\{x_1, \cdots, x_n\} \subseteq X$, we have
\begin{align*}
  \|\lambda_1 x_1 + \cdots + \lambda_n x_n\| = \max_{1\le i \le n} \|\lambda_i x_i\|, \ (\lambda_1, \cdots, \lambda_n \in K).
\end{align*}
We denote by $(E',\|\cdot\|)$ the Banach space of all continuous $K$-linear functionals on $E$ with the usual operator norm. For a closed subspace $D \subseteq E$, we set $D^{\perp} := \{f \in E' : D \subseteq \mathrm{Ker}f\}$. If $E$ is finite-dimensional, then $D'$ is canonically isometrically isomorphic to $E'/D^{\perp}$ (cf. \cite{van}). For two Banach spaces $(E,\|\cdot\|),(F,\|\cdot\|)$, we write $E \cong F$ if there exists an isometric isomorphism from $E$ to $F$. \par
Let $(E_1,\|\cdot\|_{E_1})$, $(E_2,\|\cdot\|_{E_2})$ be two Banach spaces. We denote by $(E_1 \oplus E_2,\|\cdot\|_{E_1 \oplus E_2})$ the direct sum  of $E_1$ and $E_2$,
\begin{align*}
  \|(x_1,x_2)\|_{E_1 \oplus E_2} := \|x_1\|_{E_1} \lor \|x_2\|_{E_2} \ (= \max \{\|x_1\|_{E_1}, \|x_2\|_{E_2}\}).
\end{align*}
A Banach space $E$ is said to be decomposable if there exist two non-zero Banach spaces $E_1$ and $E_2$ such that $E \cong E_1 \oplus E_2$. If $E$ is not decomposable, then $E$ is said to be indecomposable. \par
We denote by $(E^{\lor},\|\cdot\|)$ the spherical completion of $(E,\|\cdot\|)$ (see \cite[Theorem 4.43]{van}), which is spherically complete and an immediate extension of $(E,\|\cdot\|)$. Here, for a closed subspace $D$ of $E$, we say that $E$ is an immediate extension of $D$ if for each $x \in E$, $x \neq 0$, there exists $d \in D$ for which $\|x - d\| < \|x\|$. The following lemma is fundamental, and we will use the lemma without reference in this paper.
\begin{lem}[{\cite[Lemma 4.42]{van}}]
  Let $E$ be an immediate extension of a closed linear subspace $D$, and $(F,\|\cdot\|)$ be a spherically complete Banach space. Then each linear isometry $T : D \to F$ can be extended to a linear isometry from $E$ to $F$.
\end{lem}
From now on, in this paper, $K$ is assumed to be \textbf{non-spherically complete}. Moreover, we fix a spherically complete field $(K^{\lor},|\cdot|)$ which is an immediate extension of $K$ (cf. \cite[Theorem 4.49]{van}). Then we have $V_K = V_{K^{\lor}}$. \par
For $t \in \mathbb{R}_{>0}$ and $x \in K^{\lor}$, we set $|x|_t := t|x|$. Also, we denote by $(K^{\lor},|\cdot|_t)$ the $K$-normed space $K^{\lor}$ equipped with the norm $|\cdot|_t$. Similarly, for $t_1, \cdots, t_n \in \mathbb{R}_{>0}$, we set
\begin{align*}
  ((K^{\lor})^n, |\cdot|_{t_1} \times \cdots \times |\cdot|_{t_n}) := (K^{\lor},|\cdot|_{t_1}) \oplus \cdots \oplus (K^{\lor},|\cdot|_{t_n}),
\end{align*}
the direct sum of the normed spaces $(K^{\lor},|\cdot|_{t_1}), \cdots, (K^{\lor},|\cdot|_{t_n})$. Moreover, for any $K$-subspace $D \subseteq (K^{\lor})^n$, we write
\begin{align*}
  (D,|\cdot|_{t_1} \times \cdots \times |\cdot|_{t_n})
\end{align*}
when the norm on $D$ is induced by $((K^{\lor})^n, |\cdot|_{t_1} \times \cdots \times |\cdot|_{t_n})$. \par

 In this section, suppose that the cardinality of a maximal orthogonal subset of $E$ is finite. Fix one structure of Banach space over $K^{\lor}$ on $E^{\lor}$ compatible with its structure as a $K$-Banach space (\cite[3.20]{gru}).

\begin{lem}
  Let $Y \subseteq E$ be an orthogonal subset. Then $Y \subseteq E^{\lor}$ is also an orthogonal subset with respect to the structure as a $K^{\lor}$-Banach space. 
\end{lem}

\begin{prop}
  Let $X \subseteq E$ be a maximal orthogonal subset of $E$. Then the $K^{\lor}$-subspace $[X]_{K^{\lor}} \subseteq E^{\lor}$ generated by $X$ is equal to $E^{\lor}$. Consequently, we have $\# X = \dim_{K^{\lor}}E^{\lor}$ and $\dim_{K^{\lor}}E^{\lor}$ is independent of the choice of the structure of a Banach space over $K^{\lor}$ on $E^{\lor}$.
\end{prop}
\begin{proof}
  Since $\# X < \infty$, $[X]_{K^{\lor}}$ is spherically complete. Therefore, since $E^{\lor}$ is an immediate extension of $[X]_{K^{\lor}}$, we have $[X]_{K^{\lor}} = E^{\lor}$. Now by the preceding lemma, we obtain $\# X = \dim_{K^{\lor}}E^{\lor}$.
\end{proof}

Due to the above proposition, we will denote by $\dim_{K^{\lor}}E^{\lor}$ the cardinality of a maximal orthogonal subset of $E$. The above proposition says that $E$ can be embedded into $(K^{\lor})^{\# X}$. More precisely, we have the following theorem.

\begin{thm}\label{pthm1}
  Let $X = \{a_1, \cdots, a_m\} \subseteq E$ be a maximal orthogonal subset of $E$, and put $t_1 = \|a_1\|, \cdots, t_m = \|a_m\|$. Then there exists a $K$-linear isometry 
  \begin{align*}
    T : E \to ((K^{\lor})^m, |\cdot|_{t_1} \times \cdots \times |\cdot|_{t_m})
  \end{align*}
  such that $T(a_1) = e_1, \cdots, T(a_m) = e_m$ where $e_1, \cdots, e_m \in (K^{\lor})^m$ are the canonical unit vectors.
\end{thm}

Thus, to study finite-dimensional normed spaces, we only have to study the finite-dimensional normed space of the form
\begin{align*}
  ([e_1, \cdots, e_m, f_1, \cdots, f_n], |\cdot|_{t_1} \times \cdots \times |\cdot|_{t_m})
\end{align*}
of dimension $m + n$, where $e_1, \cdots, e_m \in (K^{\lor})^m$ are the canonical unit vectors and each $f_j$ satisfies $f_j = \sum_i a_{i,j} e_i$ with $a_{i,j} \in K^{\lor} \setminus K$ or $a_{i,j} = 0$. \par

From now on, throughout this paper, all normed spaces $(E, \|\cdot\|)$ are assumed to be \textbf{finite-dimensional}.

\section{holes}

In \cite{open}, a maximal chain of balls of $K$ with an empty intersection is called a hole. Holes are useful to study $2$-dimensional normed spaces over $K$. On the other hand, an element $x \in K^{\lor} \setminus K$ corresponds to a hole. It is seen that an element $x \in K^{\lor} \setminus K$ is clearer than a hole and therefore we shall study finite-dimensional normed spaces via $K^{\lor}$. The definition below is introduced in \cite{fin}.

\begin{dfn}[{\cite[Proposition 2.7]{fin}}]
  Let $x,y \in K^{\lor} \setminus K$. We say that $x$ and $y$ are equivalent if there exist $\lambda, \mu \in K$ such that $\lambda x + \mu \in B(y, d(y, K))$. We write $x \sim y$ if $x$ and $y$ are equivalent.
\end{dfn}

\begin{rem}
  If $x \sim y$, then we can choose $\lambda, \mu \in K$ such that
  \begin{align*}
    |y - (\lambda x + \mu)| < d(y,K).
  \end{align*}
\end{rem}

\begin{prop} \label{proph1}
  If $z \in B(x,d(x,K))$ and $x \sim y$, then we have $z \sim y$. Therefore, the relation $\sim$ is an equivalence relation on $K^{\lor} \setminus K$.
\end{prop}
\begin{proof}
  Let $\lambda, \mu \in K$, $\lambda \neq 0$, satisfy $\lambda x + \mu \in B(y, d(y, K))$. Since $d(y,K)$ is not attained, we have 
  \begin{align*}
    d(y,K) = d(\lambda x + \mu, K) = |\lambda| \cdot d(x,K).
  \end{align*}
  Thus, we have the desired inequality
  \begin{align*}
    |\lambda z + \mu - (\lambda x + \mu)| \le |\lambda| \cdot |x-z| \le d(y, K).
  \end{align*}
   Now, the latter statement easily follows.
\end{proof}

From the proof of the above proposition, we have the following.

\begin{prop} \label{prph5}
  Let $x,y \in K^{\lor} \setminus K$ and $\lambda, \mu \in K$ with $\lambda x + \mu \in B(y, d(y, K))$. Then we have $|\lambda| = d(y,K)/d(x,K)$.
\end{prop}

Let $(E, \|\cdot\|)$ be a $2$-dimensional normed space with $\dim_{K^{\lor}} E^{\lor} = 1$, and let $e_1,e_2 \in E$ be its basis. Then there exists a linear isometry $T : E \to (K^{\lor},|\cdot|_{t})$ such that $T(e_1) = 1, T(e_2) \in K^{\lor} \setminus K$ where $t = \|e_1\|$. If a sequence $\lambda_1, \lambda_2, \cdots \in K$ satisfies $\|e_2 - \lambda_n e_1\| \to d_E(e_2,[e_1])$ as $n \to \infty$, then we easily see $|T(e_2) - \lambda_n| \to d(T(e_2),K)$ as $n \to \infty$. Conversely, we have the following.

\begin{prop} \label{proph9}
  Let $(E,\|\cdot\|)$ be a $2$-dimensional normed space with $\dim_{K^{\lor}} E^{\lor} = 1$, and let $e_1,e_2 \in E$ be its base. Choose a sequence $\lambda_1, \lambda_2, \cdots \in K$ such that  $\|e_2 - \lambda_n e_1\| \to d_E(e_2,[e_1])$ as $n \to \infty$. Then for each $x \in K^{\lor} \setminus K$ with $|x - \lambda_n| \to d(x,K)$ as $n \to \infty$ (e.g. $T(e_2)$ above), the map
  \begin{align*}
    (E,\|\cdot\|) &\to ([1, x], |\cdot|_{t}) \\
    e_1 &\mapsto 1 \\
    e_2 &\mapsto x
  \end{align*}
  where $t = \|e_1\|$ gives a linear isometry.
\end{prop}
\begin{proof}
 Let $T : E \to (K^{\lor},|\cdot|_{t})$ be a linear isometry such that $T(e_1) = 1$. Then we have $|T(e_2) - \lambda_n| \to d(T(e_2),K)$ as $n \to \infty$. Hence, we have
 \begin{align*}
   |T(e_2) - x| \le |T(e_2) - \lambda_n| \lor |x - \lambda_n| \to d(T(e_2),K) \lor d(x,K) (n \to \infty).
 \end{align*}
 and it is easy to see $d(T(e_2),K) = d(x,K)$. Moreover, since $d(T(e_2),K)$ is not attained, for each $\lambda \in K$, we have
 \begin{align*}
   |x - \lambda| = |x-T(e_2) - (T(e_2) - \lambda)| = |T(e_2) - \lambda|.
 \end{align*}
 Therefore, the map
 \begin{align*}
   ([1, T(e_2)], |\cdot|)  &\to ([1, x], |\cdot|) \\
    1 &\mapsto 1 \\
    T(e_2) &\mapsto x
 \end{align*}
 gives a linear isometry, which completes the proof.
\end{proof}

\begin{lem} \label{lemh2}
  Let $x \in K^{\lor} \setminus K$ and $\begin{pmatrix}
    a & b \\ c & d \\
  \end{pmatrix} \in \mathrm{GL}_2(K)$ where $\mathrm{GL}_2(K)$ is the general linear group over $K$ of degree $2$. Then we have
  \begin{align*}
    y := \frac{c + dx}{a + bx} \sim x.
  \end{align*}
\end{lem}
\begin{proof}
  By Proposition \ref{proph1}, we may assume $|x| = 1$ and $y = 1/x$. Put $r = d(x,K) < 1$ and choose $\lambda \in K$ such that $|x - \lambda| < r$. Then $|\lambda| = 1$ and we have 
  \begin{align*}
    |-\frac{1}{\lambda^2}x + \frac{2}{\lambda} - y| = |x - \lambda|^2 < |x - \lambda| < r.
  \end{align*}
  On the other hand, we have 
  \begin{align*}
    d(y,K) = \inf_{\lambda \in K} \left|\frac{1- \lambda x}{x} \right| = \inf_{\lambda \in K} |1 - \lambda x| = r
  \end{align*}
  and therefore the proof is complete.
\end{proof}

From the above lemma, we obtain the structure theorem for $2$-dimensional normed spaces.

\begin{thm}[cf. {\cite[Section 2]{open}}] \label{thmh3}
  Let $x,y \in K^{\lor} \setminus K$ and $t_1, t_2 \in \mathbb{R}_{>0}$. Then 
  \begin{align*}
    ([1, x], |\cdot|_{t_1}) \cong ([1, y], |\cdot|_{t_2})
  \end{align*}
  if and only if
  \begin{align*}
    x \sim y \ \mathrm{and} \ \frac{t_1}{t_2} \in V_K.
  \end{align*}
\end{thm}
\begin{proof}
  Suppose $x \sim y$ and $t_1 / t_2 \in V_K$. To prove $([1, x], |\cdot|_{t_1}) \cong ([1, y], |\cdot|_{t_2})$, we may assume $y \in B(x, d(x,K))$. Let $a \in K$ with $|a| = t_1 / t_2$. Then by Proposition \ref{proph9}, the map
  \begin{align*}
    ([1, x], |\cdot|_{t_1}) \to ([1, y], |\cdot|_{t_2}), \quad \lambda + \mu x \mapsto a(\lambda + \mu y)
  \end{align*}
  is an isometry. \par
  Conversely, suppose $([1, x], |\cdot|_{t_1}) \cong ([1, y], |\cdot|_{t_2})$. Then it is easy to see $t_1 / t_2 \in V_K$. Thus, there exists an isometry 
  \begin{align*}
    T : ([1, x], |\cdot|) \to ([1, y], |\cdot|).
  \end{align*}
  Put
  \begin{align*}
    T(1) = a + by, T(x) = c + dy \ \mathrm{and} \ x' = \frac{c + dy}{a + by}.
  \end{align*}
  Then for any $\lambda \in K$, we have the equality
  \begin{align*}
    |\lambda + x| = |T(\lambda + x)| &= |\lambda (a + by) + (c + dy)| \\
    &= |T(1)| \cdot |\lambda + x'| \\
    &= |\lambda + x'|.
  \end{align*}
  Let $\lambda_1, \lambda_2, \cdots \in K$ with $|\lambda_n + x| \to d(x,K)$ as $n \to \infty$. Then it follows from the above equality that we have 
  \begin{align*}
    |x - x'| = |x + \lambda_n - (x' + \lambda_n)| \le |x + \lambda_n| \to d(x,K) \ (n \to \infty).
  \end{align*}
  Therefore, we have $x \sim x' \sim y$ by Lemma \ref{lemh2}.
\end{proof}

\begin{thm}[cf. {\cite[Theorem 1.9]{note}}] \label{thmh10}
  Let $x \in K^{\lor} \setminus K$, $t \in \mathbb{R}_{>0}$ and put $r = d(x,K)$. Then we have $([1, x], |\cdot|_t)' \cong ([1, x], |\cdot|_{1/tr})$. More precisely, let $e_1, e_2$ be the dual basis of $x, 1$. Then the map
  \begin{align*}
    ([1, x], |\cdot|_t)' &\to ([1, x], |\cdot|_{\frac{1}{tr}}) \\
    e_1 &\mapsto 1 \\
    e_2 &\mapsto - x
  \end{align*}
  gives a linear isometry.
\end{thm}
\begin{proof}
  Let $e_1, e_2 \in E'$ be the dual basis of $x, 1$.
  Then we have
  \begin{align*}
    \|e_1\| &= \sup_{\lambda \in K} \frac{1}{|\lambda + x|_t} = \frac{1}{tr}, \\
    \|e_2\| &= \sup_{\lambda \in K} \frac{1}{|1 + \lambda x|_t} = \frac{1}{t \cdot \inf_{\lambda \in K} |1 + \lambda x|} = \frac{|x|}{tr}.
  \end{align*}
  Moreover, for any non-zero $\lambda \in K$, we have
  \begin{align*}
    \|e_1 + \lambda e_2\| &= \sup_{\mu \in K} \frac{|1 + \lambda \mu|}{|x + \mu|_t} \lor \frac{|\lambda|}{|1|_t} \\
    &= \frac{|\lambda|}{t} \cdot \left(\sup_{\mu \in K} \frac{|1/\lambda - x + (x + \mu)|}{|x + \mu|} \lor 1 \right) \\
    &=\frac{|\lambda|}{t} \cdot \frac{|1/\lambda - x|}{r} = \frac{1}{tr} \cdot \left|1 - \lambda x \right|.
  \end{align*}
  Hence, we can conclude 
  \begin{align*}
    ([1, x], |\cdot|_t)' \cong ([1, x], |\cdot|_{1/tr}).
  \end{align*} 
\end{proof}

Next, we have a characterization of the relation $\sim$ from the perspective of the orthogonality, which is important in this paper.

\begin{thm} \label{thmh4}
  Let $x,y \in K^{\lor} \setminus K$, and $\pi : K^{\lor} \to K^{\lor} / K$ be the canonical quotient map. Then $x \nsim y$ if and only if $\{ \pi(x), \pi(y) \}$ is an orthogonal set with respect to the quotient norm of $K^{\lor} / K$.
\end{thm}
\begin{proof}
 Suppose $x \nsim y$. Then for any $\lambda, \mu \in K$, we have $\lambda x + \mu \notin B_{K^{\lor}}(y,d(y,K))$. Thus, it follows that
 \begin{align*}
   d(\pi(y), [\pi(x)]) = d(y, [1,x]) \ge d(y,K) = \| \pi(y) \|.
 \end{align*}
  Hence, $\{ \pi(x), \pi(y) \}$ is an orthogonal set. \par
  Conversely, suppose that there exist $\lambda, \mu \in K$ such that $\lambda x + \mu \in B_{K^{\lor}}(y,d(y,K))$. We will prove that $\{ \pi(x), \pi(y) \}$ is not an orthogonal set.  Clearly, we may assume $\pi(x)$ and $\pi(y)$ are linearly independent. Then since $d(y - (\lambda x + \mu), K)$ is not attained, we have 
  \begin{align*}
    d(\pi(y), [\pi(x)]) \le d(y - (\lambda x + \mu), K) < d(y,K) = \|\pi(y)\|,
  \end{align*}
  which completes the proof.
\end{proof}

We shall give some applications of the above theorem. We recall the problem proposed in \cite{fin}. This problem can be interpreted as follows.

\begin{prob*}[{\cite[Problem 2.13]{fin}}]
  Let $n \ge 3$, and let $x_1, \cdots, x_n \in K^{\lor} \setminus K$ satisfy $x_i \nsim x_j$ whenever $i \neq j$. Then does a set $\{\pi(x_1), \cdots, \pi(x_n) \}$ form an orthogonal subset? Here, $\pi : K^{\lor} \to K^{\lor} / K$ is the canonical quotient map.
\end{prob*}

The answer to the problem above is no.
  
\begin{exam}
  Let $K$ be separable as a metric space and $r > 0$. Then by \cite[Corollary 2.14]{fin}, there exist $x,y \in K^{\lor} \setminus K$ such that $x \nsim y$ and $d(x,K) = d(y,K)$. Now, by Theorem \ref{thmh4}, $x \nsim x+y$ and $y \nsim x+y$. Therefore, $x,y$ and $x+y$ give the negative answer to the problem above. 
\end{exam}

\begin{thm} \label{thmh6}
  Let $K$ be an algebraically closed field of characteristic $p > 0$. Then there exists a sequence $x_1, x_2, \cdots \in K^{\lor}$ such that $\{\pi(x_1), \pi(x_2), \cdots \}$ is an orthogonal subset where $\pi : K^{\lor} \to K^{\lor}/K$ is the canonical quotient map.
\end{thm}
\begin{proof}
  Let $x \in K^{\lor} \setminus K$ with $|x| = 1$, and put $r = d(x,K)$. Choose a sequence $\lambda_1, \lambda_2, \cdots \in K$ such that $|x - \lambda_n| \to r$ as $n \to \infty$. Then for all $m \in \mathbb{N}$, we have
  \begin{align*}
   \|\pi(x^{p^m})\| = d(x^{p^m},K) \le |x - \lambda_n|^{p^m} \to r^{p^m} \ (n \to \infty).
  \end{align*}
  On the other hand, since $K$ is algebraically closed, we obtain
  \begin{align*}
   d(\pi(x^{p^m}),[\pi(x), \cdots, \pi(x^{p^{m-1}})]) = d(x^{p^m},[1,x,\cdots,x^{p^{m-1}}]) \ge (\inf_{\lambda \in K} |x - \lambda|)^{p^m} = r^{p^m}.
  \end{align*}
  These imply that $\{\pi(x), \pi(x^p), \cdots \}$ is an orthogonal set.
\end{proof}

\begin{thm} \label{thmh7}
  Let $k$ be the residue class field of $K$. Suppose that $K$ is an algebraically closed field of characteristic $0$ and $k$ is a field of characteristic $p > 0$. Then for each $m \in \mathbb{N}$, there exist $x_1, \cdots, x_m \in K^{\lor}$ such that $\{\pi(x_1), \cdots \pi(x_m)\}$ is an orthogonal subset where $\pi : K^{\lor} \to K^{\lor} / K$ is the canonical quotient map.
\end{thm}
\begin{proof}
  Fix $m \in \mathbb{N}$. Let $x \in K^{\lor} \setminus K$ with $|x| = 1$, and put $r = d(x,K)$. Choose $a \in K$ such that $|p|^{1/(p^m - p^{m-1})} < |a|r < 1$. Then there exists $\lambda \in K$ such that $y := ax - \lambda$ satisfies $|y| = 1$. Notice that $y$ also satisfies $s := d(y,K) = |a|r$. Let $k \in \mathbb{N}$ with $1 \le k \le m$, and $\zeta_{p^k} \in K$ be a primitive $p^k$-th root of unity. Choose $\lambda_1, \lambda_2, \cdots \in K$ such that $|y - \lambda_n| \to s$ as $n \to \infty$ and $|\lambda_n| = 1$ for all $n$. Then it follows from $|p|^{1/(p^m - p^{m-1})} < s$ that 
  \begin{align*}
    |y^{p^k} - \lambda_n^{p^k}| = \prod_{l=0}^{p^k -1} |y - \lambda_n \zeta_{p^k}^l| = |y - \lambda_n|^{p^k} \to s^{p^k} \ (n \to \infty).
  \end{align*}
  Hence, we obtain $\|\pi(y^{p^k})\| \le s^{p^k}$. On the other hand, by the same proof as that of Theorem \ref{thmh6}, we obtain the inequality $d(y^{p^k},[1,y,\cdots,y^{p^{k-1}}]) \ge s^{p^k}$. Therefore, $\{\pi(y), \cdots, \pi(y^{p^m})\}$ is an orthogonal set.
\end{proof}

Now, we recall \cite[Problem in Section 2]{open}.

\begin{prob*}[{\cite[Problem in Section 2]{open}}]
  Is the set of all $\sim$-equivalence classes is infinite for all non-spherically complete valued fields?
\end{prob*}

Finally, we have the following corollary, which is a partial answer to the problem above.

\begin{cor}
  Let $k$ be the residue class field of $K$. Suppose that $K$ is algebraically closed and $k$ is a field of characteristic $p > 0$. Then the set of all $\sim$-equivalence classes is an infinite set.
\end{cor}

\section{Strictly epicompact sets} \label{str}

\begin{dfn}[{\cite[Section 7]{equ}}]
  Let $E$ be a finite-dimensional normed space, and let $A \subseteq E$ be a bounded $B_K$-module. Then $A$ is said to be strictly epicompact if for each finite-dimensional vector space $F$ and each linear map $T : E \to F$, we have
  \begin{align*}
    T(A) = \bigcap_{|\lambda|>1} \lambda T(A).
  \end{align*}
\end{dfn}

The following proposition is proved in \cite{equ}, which is fundamental to studying strictly epicompact sets.

\begin{prop}[{\cite[Proposition 7.3]{equ}}]
  Let $(E,\|\cdot\|)$ be a finite-dimensional normed space. Then $B_{E'}$ is strictly epicompact if and only if for each subspace $D \subseteq E$ and each $f \in D'$ with $\|f\| \le 1$, there exists an extension $\tilde{f} \in E'$ of $f$ with $\|\tilde{f}\| \le 1$.
\end{prop}

Since $E$ is finite-dimensional, we can improve the above proposition.

\begin{prop} \label{propst1}
   Let $(E,\|\cdot\|)$ be a finite-dimensional normed space. Then $B_{E'}$ is strictly epicompact if and only if
  \begin{itemize}
    \item[$(\mathrm{SE})$] For each subspace $D \subseteq E$ and each $f \in D'$ with $\|f\| = 1$, there exists an extension $\tilde{f} \in E'$ of $f$ with $\|\tilde{f}\| = 1$.
  \end{itemize}
\end{prop}

Thus, to study strictly epicompact sets, it suffices to study whether $E$ satisfies (SE) or not. \par

\begin{prob*}[{\cite[Section 7]{equ}}]
  Which finite-dimensional normed spaces satisfy (SE)?
\end{prob*}

One goal of this paper is to study a necessary and sufficient condition for certain finite-dimensional normed spaces to satisfy (SE).

A simple observation leads to the following proposition, which is very useful.

\begin{prop} \label{seprop1}
  Let $(E,\|\cdot\|)$ be a finite-dimensional normed space, and let $D \subseteq E$ be a one-dimensional subspace. Then for each $f \in D'$ with $\|f\| = 1$, there exists an extension $\tilde{f} \in E'$ of $f$ with $\|\tilde{f}\| = 1$ if and only if
  \begin{align*}
    \|D\| \cap V_K = \emptyset \ \mathrm{or} \ E \cong D \oplus F \  \mathrm{for} \ \mathrm{some} \ F \subseteq E.
  \end{align*}
\end{prop}
\begin{proof}
  Indeed, if $\|D\| \cap V_K = \emptyset$, then there exists no $f \in D'$ such that $\|f\| = 1$. Thus, to prove the proposition, we may assume $\|D\| = |K|$. We will prove the sufficiency. Since $\|D\| = |K|$, there exists $f \in D'$ and $u \in D$ for which $\|f\| = \|u\| = |f(u)| = 1$. By assumption, there is an extension $\tilde{f} \in E'$ of $f$ with $\|\tilde{f}\| = 1$. Thus, we obtain
  \begin{align*}
    E \cong D \oplus \mathrm{Ker}f.
  \end{align*}
  The converse is obvious, and the proof is complete.
\end{proof}

\begin{cor}
  Let $(E, \|\cdot\|)$ be a finite-dimensional normed space such that $\|E\| = |K|$. Then $E$ satisfies $(\mathrm{SE})$ if and only if $E$ has an orthogonal base, that is, $\dim_{K^{\lor}} E^{\lor} = \dim_K E$.
\end{cor}

Due to the corollary above, it is not reasonable to restrict ourselves to studying only normed spaces $(E, \|\cdot\|)$ satisfying $\|E\| = |K|$. By Proposition \ref{seprop1}, we immediately get the following.

\begin{thm}
Let $x \in K^{\lor} \setminus K$ and $t \in \mathbb{R}_{>0}$. Then a $2$-dimensional normed space
\begin{align*}
  ([1, x], |\cdot|_{t})
\end{align*}
satisfies $(\mathrm{SE})$ if and only if $t \notin V_K$.
\end{thm}

\begin{cor} \label{secor1}
  Let $(E, \|\cdot\|)$ be a finite-dimensional normed space. Then $E$ satisfies $(\mathrm{SE})$ if and only if $E$ satisfies the following two conditions: \\
  $(1)$ $E / [u]$ satisfies $(\mathrm{SE})$ for each non-zero $u \in E$. \\
  $(2)$ each $2$-dimensional subspace $F$ of $E$ with $\dim_{K^{\lor}} F^{\lor} = 1$ satisfies $\| F \| \cap V_K = \emptyset$.
\end{cor}
\begin{proof}
  Suppose that the conditions $(1)$ and $(2)$ hold. We shall prove that $E$ satisfies (SE). Let $D$ be a subspace of $E$ and $f \in D'$ with $\|f\| = 1$. Then by the condition $(2)$ and the preceding theorem, we may assume $\dim_K D \ge 2$. Therefore, we have $\ker f \neq \{0\}$. Let $\pi : E \to E / \ker f$ be the canonical quotient map, and let $g \in (\pi(D))'$ satisfy $f = g \circ \pi$. Obviously, we have $\|g\| = 1$ and therefore by the condition $(1)$, there exists an extension $\tilde{g} \in (E / \ker f)'$ of $g$ with $\|\tilde{g}\| = 1.$ Now, $\tilde{g} \circ \pi \in E'$ is an extension of $f$ with $\|\tilde{g} \circ \pi\|$. The converse is easy to prove. Thus, the proof is complete.
\end{proof}

We give another necessary and sufficient condition for a normed space to satisfy (SE).

\begin{lem} \label{selem1}
  Let $(E, \|\cdot\|)$ be an $n$-dimensional normed space. Then $E$ satisfies $(\mathrm{SE})$ if and only if $E$ satisfies the following two conditions: \\
  $(1)$ each $(n -1)$-dimensional subspace of $E$ satisfies $(\mathrm{SE})$. \\
  $(2)$ each $2$-dimensional subspace $F$ of $E'$ satisfies either $\dim_{K^{\lor}} F^{\lor} = 2$ or $F \cong ([1, x], |\cdot|_t)$ with $x \in K^{\lor} \setminus K$ and $t \cdot d(x, K) \notin V_K$.
\end{lem}
\begin{proof}
  It suffices to prove that the condition $(2)$ is equivalent to the following condition $(2)'$:
  \begin{itemize}
    \item[$(2)'$] For each $(n-1)$-dimensional subspace $D$ of $E$ and $f \in D'$ with $\|f\| = 1$, there is an extension $\tilde{f} \in E'$ of $f$ with $\|\tilde{f}\| = 1$
  \end{itemize}
  First, suppose $(2)$. Let $D$ be an $(n-1)$-dimensional subspace of $E$. We note that the restriction map $E' \to D'$ can be identified with the quotient map $E' \to E'/D^{\perp}$. Now, let $f \in D'$ with $\|f\| = 1$. We can choose $g \in E'$ such that the restriction of $g$ to $D$ is equal to $f$. Consider a $2$-dimensional subspace $F = D^{\perp} + [g]$ of $E'$. Then by $(2)$ and $\|f\| = 1$, we must have $\dim_{K^{\lor}} F^{\lor} = 2$. Therefore, $F$ has an orthogonal base and there exists an extension $\tilde{f} \in F$ of $f$ satisfying $\|\tilde{f}\| = 1$. \par
  Conversely, let us suppose $(2)'$. Let $F$ be a $2$-dimensional subspace of $E'$ with $\dim_{K^{\lor}} F^{\lor} = 1$ and let $f, g \in F$ be such that $[f, g] = F$. Obviously, $\ker f$ is an $(n-1)$-dimensional subspace of $E$. Let $g_0 \in (\ker f)'$ be the restriction of g to $\ker f$. Since $\|g_0\| = d(g, [f])$ and $d(g, [f])$ is not attained, we must have $d(g, [f]) \notin V_K$ by $(2)'$. Thus, the proof is complete. 
\end{proof}

\section{Systematic approaches to finite-dimensional normed spaces} \label{sys}

In this section, we will give a general method, independent of dimension, of studying finite-dimensional normed spaces. Particularly important results are Corollary \ref{syscor1} and Theorem \ref{systhm2}. Corollary \ref{syscor1} says that if we classify $n$-dimensional normed spaces, then we can classify $(n + 1)$-dimensional normed spaces $E$ with $\dim_{K^{\lor}} E^{\lor} = 1$. Theorem \ref{systhm2} says that there exists exactly one "type" of finite-dimensional indecomposable normed space $E$ with $\dim_{K^{\lor}} E^{\lor} = \dim_K E - 1$.

\subsection{$\dim_{K^{\lor}} E^{\lor} = \dim_K E$} 

If an $n$-dimensional normed space $E$ satisfies $\dim_K E = \dim_{K^{\lor}} E^{\lor}$, then there exist $t_1, \cdots, t_n \in \mathbb{R}_{> 0}$ for which
\begin{align*}
  E \cong (K^n, |\cdot|_{t_1} \times \cdots \times |\cdot|_{t_n}).
\end{align*}

The following theorem is a structure theorem for finite-dimensional normed spaces $E$ with $\dim_{K^{\lor}} E^{\lor} = \dim_K E$.

\begin{thm} \label{.1thm1}
Let $t_1, \cdots, t_n, s_1, \cdots, s_n \in \mathbb{R}_{> 0}$. Then
  \begin{align*}
    F := (K^n, |\cdot|_{t_1} \times \cdots \times |\cdot|_{t_n}) \cong (K^n, |\cdot|_{s_1} \times \cdots \times |\cdot|_{s_n}) := G
  \end{align*}
  if and only if by a suitable permutation of $s_1, \cdots, s_n$, we have $t_i/s_i \in V_K$ for each $i=1, \cdots, n$.
\end{thm}
\begin{proof}
  We only prove the sufficiency. Suppose that there exists a linear isometry $T : F \to G$. It suffices to prove
  \begin{align*}
   m := \# \{1\le i \le n : t_i/t_1 \in K\} \le l:= \# \{1\le i \le n : s_i/t_1 \in K\}.
  \end{align*}
  To prove the above inequality, we may assume $\{1\le i \le n : t_i/t_1 \in K\} = \{t_1, \cdots, t_m\}$ and $\{1\le i \le n : s_i/t_1 \in K\} = \{s_1, \cdots, s_l\}$. Let $f_1, \cdots, f_n \in F$ and $g_1, \cdots, g_n \in G$ be the canonical unit vectors of $F$ and $G$. Then we have 
  \begin{align*}
    T([f_1, \cdots, f_m] \setminus \{0\}) \subseteq \{a_1g_1 + \cdots + a_n g_n : a_i \neq 0 \ \text{for some} \ 1 \le i \le l\}.
  \end{align*}
  Thus, for the canonical projection $P$ of $G$ onto $[g_1, \cdots, g_l]$, the restriction of $P \circ T$ to $[f_1, \cdots, f_m]$ is injective, and therefore, we have $m \le l$. 
\end{proof}

\begin{rem} \label{rem1}
  The above proof works even if $K$ is spherically complete.
\end{rem}

The following theorem is easy, and thus we omit the proof.

\begin{thm}
  If $\dim_{K^{\lor}} E^{\lor} = \dim_K E$, then $E$ satisfies $(\mathrm{SE})$.
\end{thm}

\subsection{$\dim_{K^{\lor}} E^{\lor} = 1$}

The following theorem is fundamental to studying a normed space $E$ with $\dim_{K^{\lor}} E^{\lor} = 1$.

\begin{thm} \label{systhm1}
  Let $(E, \|\cdot\|)$ be an $n$-dimensional normed space such that for each subspace $F \subseteq E$ with $\dim E/F = 2$, it satisfies $\dim_{K^{\lor}} (E/F)^{\lor} = 1$. Then all $(n-1)$-dimensional subspaces of $E$ are isometrically isomorphic to each other.
\end{thm}
\begin{proof}
  Let $G_1$ and $G_2$ be two distinct $(n-1)$-dimensional subspaces of $E$, and put $F := G_1 \cap G_2$. Obviously, we have $\dim E/F = 2$. Choose $x,y \in E$ such that $G_1 = F + K x$ and $G_2 = F + K y$, and let $\pi : E \to E/F$ be the canonical quotient map. Then by assumption, there exists $\lambda \in K$ such that $\|\pi(x) - \pi(\lambda y)\| < \|\pi(x)\|$. Hence, we can choose $u \in F$ such that
  \begin{align*}
    \|x - \lambda y - u\| < \|\pi(x)\|.
  \end{align*}
  Finally, we obtain
  \begin{align*}
    \|x + v\| = \|\lambda y + u + v\| \ \text{for all $v \in F$,}
  \end{align*} 
  which completes the proof.
\end{proof}

By the dual theory, if a finite-dimensional normed space $E$ satisfies $\dim_{K^{\lor}} E^{\lor} = 1$, then the dual space $E'$ of $E$ satisfies the assumption in the theorem above. Hence, we have the following corollary.

\begin{cor} \label{syscor1}
  Let $E$ be an $n$-dimensional normed space such that $\dim_{K^{\lor}} E^{\lor} = 1$. Then all $(n-1)$-dimensional normed spaces of the form $E/[u]$, $u \neq 0$, are isometrically isomorphic to each other.
\end{cor}

By the corollary above, we can refine Corollary \ref{secor1} and Lemma \ref{selem1} in the case $\dim_{K^{\lor}} E^{\lor} = 1$.

\begin{cor} \label{syscor3}
  Let $(E, \|\cdot\|)$ satisfy $\dim_{K^{\lor}} E^{\lor} = 1$ and $u \in E \setminus \{0\}$. \\
  $(1)$ $E$ satisfies $(\mathrm{SE})$ if and only if $\|E\| \cap V_K = \emptyset$ and $E / [u]$ satisfies $(\mathrm{SE})$. \\
  $(2)$ $E'$ satisfies $(\mathrm{SE})$ if and only if $\|E / [u] \| \cap V_K = \emptyset$ and $(E/[u])'$ satisfies $(\mathrm{SE})$.
\end{cor}

If an $(n + 1)$-dimensional normed space $E$ satisfies $\dim_{K^{\lor}} E^{\lor} = 1$, then there exist $x_1, \cdots, x_n \in K^{\lor} \setminus K$ and $t \in \mathbb{R}_{> 0}$ for which 
\begin{align*}
  E \cong ([1,x_1, \cdots, x_n], |\cdot|_t) \subseteq (K^{\lor}, |\cdot|_t).
\end{align*}
Therefore, we shall study a normed space of the form $([1,x_1, \cdots, x_n], |\cdot|_t)$.

\begin{lem} \label{syslem1}
  Let $n \in \mathbb{N}$ and let $x_1, \cdots, x_n, y_1, \cdots, y_n \in K^{\lor} \setminus K$. Let $\pi : K^{\lor} \to K^{\lor}/K$ be the canonical quotient map. Suppose that $\{\pi(x_1), \cdots, \pi(x_n)\}$ is an orthogonal set and $x_i \sim y_i$ for each $i = 1, \cdots, n$. Then we have
  \begin{align*}
    ([1, x_1, \cdots, x_n], |\cdot|) \cong ([1, y_1, \cdots, y_n], |\cdot|)
  \end{align*}
  and that $\{\pi(y_1), \cdots, \pi(y_n)\}$ is an orthogonal set. More precisely, if $|x_i - y_i| \le d(x_i,K)$ for each $i = 1, \cdots, n$, then the map
  \begin{align*}
    ([1, x_1, \cdots, x_n], |\cdot|) &\to ([1, y_1, \cdots, y_n], |\cdot|) \\
    1 &\mapsto 1 \\
    x_i &\mapsto y_i \ (1 \le i \le n)
  \end{align*}
  gives a linear isometry.
\end{lem}
\begin{proof}
  We may assume $x_2 = y_2, \cdots, x_n = y_n$ and $|x_1 - y_1| \le d(x_1, K)$. By assumption, we have $d(x_1,[1, x_2, \cdots, x_n]) = d(x_1,K)$. Since $d(x_1, [1,x_2, \cdots, x_n])$ is not attained, for each $\lambda, \lambda_2, \cdots, \lambda_n \in K$, we have
  \begin{align*}
    |\lambda + x_1 + \lambda_2 x_2 + \cdots + \lambda_n x_n| = |\lambda + y_1 + \lambda_2 x_2 + \cdots + \lambda_n x_n|.
  \end{align*}
  Thus, the map
  \begin{align*}
    ([1, x_1, \cdots, x_n], |\cdot|) &\to ([1, y_1, x_2, \cdots, x_n], |\cdot|) \\
    1 &\mapsto 1 \\
    x_1 &\mapsto y_1 \\
    x_i &\mapsto x_i \ (2 \le i \le n)
  \end{align*}
  gives a linear isometry. Now, it is easy to see that $\{\pi(y_1),\pi(x_2), \cdots, \pi(x_n)\}$ is an orthogonal set, which completes the proof.
\end{proof}

\begin{thm}
  Let $n \in \mathbb{N}$ and let $x_1, \cdots, x_n \in K^{\lor} \setminus K$. Let $\pi : K^{\lor} \to K^{\lor}/K$ be the canonical quotient map. Suppose that $\{\pi(x_1), \cdots, \pi(x_n)\}$ is an orthogonal set. Then for each $n$-dimensional subspace $F$ of $([1, x_1, \cdots, x_n], |\cdot|)$, there exist $y_1, \cdots, y_{n-1} \in [1, x_1, \cdots, x_n]$ such that 
  \begin{align*}
    F \cong ([1, y_1, \cdots, y_{n-1}], |\cdot|).
  \end{align*}
\end{thm}
\begin{proof}
  First, we shall prove that for each $x, y \in K^{\lor} \setminus K$ such that $[\pi(x), \pi(y)]$ is a $2$-dimensional normed space with an orthogonal base, there exists $z \in [1,x,y]$ such that $z \sim y/x$. Indeed, let $a \in K$ be such that $y + a x \nsim x$, and put $y' = y + a x$. It follows from $([1, x, y'], |\cdot|) \cong ([1, \tfrac{1}{x}, \tfrac{y'}{x}], |\cdot|)$ and Corollary \ref{syscor1} that $[\pi(1/x),\pi(y'/x)]$ has an orthogonal base. Hence, there exists $\lambda \in K$ such that 
  \begin{align*}
    \frac{1}{x} \nsim \frac{y' + \lambda}{x}.
  \end{align*}
  In particular, we have
  \begin{align*}
    d(\frac{y' + \lambda}{x}, K) = d(\frac{y' + \lambda}{x},[1, \frac{1}{x}]) = \frac{d(y',[1,x])}{|x|} = \frac{d(y',K)}{|x|}.
  \end{align*} 
  From this equality, we obtain
  \begin{align*}
    d(\frac{y' + \lambda}{x},[1,\frac{1}{x},y']) \le d(\frac{1}{x},K)d(y',K) = \frac{d(x,K)}{|x|} \frac{d(y',K)}{|x|} < d(\frac{y' + \lambda}{x}, K).
  \end{align*} 
  Therefore, there exists $w \in [1,1/x,y']$ such that 
  \begin{align*}
    w \sim \frac{y' + \lambda}{x}.
  \end{align*}
Then by Theorem \ref{thmh3} and Lemma \ref{syslem1}, we have $z' \sim y'/x$ for some $z' \in [1,1/x,w]$. Since $[1,1/x,w] \subseteq [1,1/x, y']$, we have $z \sim y'/x$ for some $z \in [1,x,y'] = [1,x,y]$, again by Theorem \ref{thmh3} and Lemma \ref{syslem1}. \par
Now, let $F$ be an $n$-dimensional subspace of $([1, x_1, \cdots, x_n], |\cdot|)$. To prove the theorem, we may assume $K \cap F = \{0\}$. Let $y' \in F$ be a non-zero element. Then by Corollary \ref{syscor1}, there exist $y_1', \cdots, y_{n-1}' \in F$ such that $\{\pi'(y_1'), \cdots, \pi'(y_{n-1}')\}$ is an orthogonal set where $\pi' : F \to F/[y']$ is the canonical quotient map. Obviously, $\{\pi(y_1'/y'), \cdots, \pi(y_{n-1}'/y')\}$ is also an orthogonal set. From what has been proved, for each $i = 1, \cdots, n-1$, there exists an element $y_i$ of $[y', y_i']$, a subspace of $[1, x_1, \cdots, x_n]$, such that $y_i \sim y_i'/y'$. Finally, by Theorem \ref{systhm1}, we have
\begin{align*}
  F \cong ([1, \frac{y_1'}{y'}, \cdots, \frac{y_{n-1}'}{y'}], |\cdot|) \cong ([1, y_1, \cdots, y_{n-1}], |\cdot|).
\end{align*}
\end{proof}

\begin{cor} \label{syscor2}
  Let $n \in \mathbb{N}$ and let $x_1, \cdots, x_n, y_1, \cdots, y_n \in K^{\lor} \setminus K$. Let $\pi : K^{\lor} \to K^{\lor}/K$ be the canonical quotient map. Suppose that $\{\pi(x_1), \cdots, \pi(x_n)\}$ is an orthogonal set and $x_i \sim y_i$ for each $i = 1, \cdots, n$. Then for each $x \in [1, x_1, \cdots, x_n] \setminus K$, there exists $y \in [1, y_1, \cdots, y_n]$ such that $x \sim y$.
\end{cor}

\subsection{$\dim_{K^{\lor}} E^{\lor} = \dim_K E - 1$}

Let $E$ be an $(n + 1)$-dimensional normed space that satisfies $\dim_{K^{\lor}} E^{\lor} = n$. Then there exist $x_1, \cdots, x_n \in K^{\lor} \setminus K$ and $t_1, \cdots, t_n \in \mathbb{R}_{>0}$ for which
\begin{align*}
    E \cong [(1,0,\cdots,0), (0,1,\cdots, 0), \cdots, (0,0,\cdots,1), (x_1, \cdots, x_n)]
\end{align*}
  where the last normed space is an $(n+1)$-dimensional subspace of $((K^{\lor})^n, |\cdot|_{t_1} \times \cdots \times |\cdot|_{t_n})$.

\begin{thm} \label{systhm2}
  Let $n \in \mathbb{N}$, $x_1, \cdots, x_n \in K^{\lor} \setminus K$, and $t_1, \cdots, t_n \in \mathbb{R}_{>0}$. For each $i = 1, \cdots, n$, put $r_i = d(x_i, K)$. Let $\pi : K^{\lor} \to K^{\lor} / K$ be the canonical quotient map. We set 
  \begin{align*}
    (E, \|\cdot\|) := [(1,0,\cdots,0), (0,1,\cdots, 0), \cdots, (0,0,\cdots,1), (x_1, \cdots, x_n)],  
  \end{align*}
  an $(n+1)$-dimensional subspace of $((K^{\lor})^n, |\cdot|_{t_1} \times \cdots \times |\cdot|_{t_n})$. Then we have the following two statements. \\
  $(1)$ If $E$ is indecomposable, then we have $r_1 t_1 = \cdots = r_n t_n$ and that $\{\pi(x_1), \cdots, \pi(x_n)\}$ is an orthogonal set. \\
  $(2)$ Suppose $r_1 t_1 = \cdots = r_n t_n =: t$ and that $\{\pi(x_1), \cdots, \pi(x_n)\}$ is an orthogonal set. Let $e_1, \cdots, e_n, e_{n+1} \in E'$ be the dual basis of $(1,0,\cdots,0), \cdots, (0,0,\cdots,1), (x_1, \cdots, x_n)$. Then the map
  \begin{align*}
    E' &\to ([1, x_1, \cdots, x_n], |\cdot|_{1/t}) \\
    e_i &\mapsto \left\{ 
      \begin{aligned}
      - x_i \ &; \ 1 \le i \le n \\
      1 \ &; \ i = n+1 
      \end{aligned}
    \right.
  \end{align*} 
  gives a linear isometry. In particular, $E$ is indecomposable and all $n$-dimensional subspaces of $E$ are isometrically isomorphic to an $n$-dimensional normed space 
  \begin{align*}
  (K^n, |\cdot|_{t_1} \times \cdots \times |\cdot|_{t_n}).
  \end{align*}
\end{thm}
\begin{proof}
  Put $f_1 := (1,0,\cdots,0), \cdots, f_n := (0,0,\cdots,1) \in E$. First, suppose that there exist $i, j$, $i \neq j$, such that $r_i t_i \neq r_j t_j$. Set 
  \begin{align*}
   I := \{1\le k \le n : r_k t_k = \max_{1 \le l \le n} r_l t_l\}.
  \end{align*}
  and let
  \begin{align*}
    I = \{i_1, \cdots, i_m\}, \ \{1,\cdots,n\} \setminus I = \{j_1, \cdots, j_{n-m}\}
  \end{align*}
  with $1 \le m < n$. Then for each $k = j_1, \cdots, j_{n-m}$, we can choose $a_{k} \in K$ such that $|x_k - a_k|_{t_k} < \max_{1 \le l \le n} r_l t_l$. Now, we put
  \begin{align*}
    y_k := \left\{ 
      \begin{aligned}
        x_k \ &; \ k \in I \\
        x_k - a_k \ &; \ k \notin I
      \end{aligned}
    \right.
  \end{align*}
  and
  \begin{align*}
    D_1 := [f_{i_1}, \cdots, f_{i_m}, (y_1, \cdots, y_n)], D_2 := [f_{j_1}, \cdots, f_{j_{m-n}}].
  \end{align*}
  Then it is easily seen that
  \begin{align*}
    E = D_1 \oplus D_2
  \end{align*}
  and therefore $E$ is decomposable. \par
  Secondly, suppose $r_1 t_1 = \cdots = r_n t_n$ and that $\{\pi(x_1), \cdots, \pi(x_n)\}$ is not an orthogonal set. We shall prove that $E$ is decomposable. Without loss of generality, we may assume that there exists $m$, $2 \le m \le n$, such that $\{\pi(x_1), \cdots, \pi(x_m)\}$ is not an orthogonal set and every subset of $\{\pi(x_1), \cdots, \pi(x_n)\}$ consisting of $m-1$ elements is an orthogonal set. Then there exist $b_1, \cdots, b_m \in K$ with $b_1 = 1$ such that
  \begin{align*}
    \|\sum_{j = 1}^{m} b_j \pi(x_j)\| < \|\pi(x_1)\|.
  \end{align*}
  Since $\{\pi(x_2), \cdots, \pi(x_m)\}$ is an orthogonal set, we obtain the equality
  \begin{align*}
    \|\pi(x_1)\| = \|\sum_{j = 2}^{m} b_j \pi(x_j)\| = \max_{2 \le j \le m} \|b_j \pi(x_j)\|.
  \end{align*}
  Let $i = 2, \cdots, m$. Then by the choice of $m$, we have 
  \begin{align*}
    \|\sum_{\substack{1 \le j \le m \\ j \neq i}} b_j \pi(x_j)\| = \max_{\substack{1 \le j \le m \\ j \neq i}} \|b_j \pi(x_j)\| \ge \|\pi(x_1)\| > \|\sum_{j = 1}^{m} b_j \pi(x_j)\|,
  \end{align*} 
  and therefore,
  \begin{align*}
    \|b_i \pi(x_i)\| = \|\sum_{\substack{1 \le j \le m \\ j \neq i}} b_j \pi(x_j)\| = \max_{\substack{1 \le j \le m \\ j \neq i}} \|b_j \pi(x_j)\|.
  \end{align*}
  Consequently, for each $i = 2, \cdots, m$, we have
  \begin{align*}
    |b_i| \cdot r_i = |b_1| \cdot r_1 = r_1,
  \end{align*}
  and hence
  \begin{align*}
    |b_i| \cdot t_1 = t_i.
  \end{align*}
  Now, let $b \in K$ be such that
  \begin{align*}
    |x_1 + \sum_{j = 2}^{m} b_j x_j + b| < r_1
  \end{align*}
  and set
  \begin{align*}
    d_j = f_j - b_j f_1 \ \text{for $j = 2, \cdots, m$}.
  \end{align*}
  We note that $\{f_1, d_2, \cdots, d_m, f_{m + 1}, \cdots, f_n\}$ is an orthogonal set from what has been proved. Moreover, for each $\lambda_1, \cdots, \lambda_n \in K$, we have
  \begin{align*}
    &\|(x_1 + b, x_2, \cdots, x_n) + \sum_{j = 2}^{m} \lambda_j d_j + \sum_{j = m + 1}^{n} \lambda_j f_j\| \\
    = \ &\|(-\sum_{j = 2}^{m} b_j x_j, x_2, \cdots, x_n) + \sum_{j = 2}^{m} \lambda_j d_j + \sum_{j = m + 1}^{n} \lambda_j f_j\| \\
    = \ &|\sum_{j = 2}^{m} b_j(x_j + \lambda_j)|_{t_1} \lor \max_{2 \le j \le n} |x_j + \lambda_j|_{t_j} \\ 
    = \ &\max_{2 \le j \le n} |x_j + \lambda_j|_{t_j}.
  \end{align*}
  Therefore, for each $\lambda \in K$, we have
  \begin{align*}
    &\|(x_1 + b, x_2, \cdots, x_n) + \sum_{j = 2}^{m} \lambda_j d_j + \sum_{j = m + 1}^{n} \lambda_j f_j + \lambda f_1\| \\
    \ge \ &\|(x_1 + b, x_2, \cdots, x_n) + \sum_{j = 2}^{m} \lambda_j d_j + \sum_{j = m + 1}^{n} \lambda_j f_j\|.
  \end{align*}
  Finally, we obtain
  \begin{align*}
    E = [f_1] \oplus [d_2, \cdots, d_m, f_{m + 1}, \cdots, f_{n}, (x_1 + b, x_2, \cdots, x_n)].
  \end{align*}
  Thus, $(1)$ is proved. \par
  Next, we shall prove $(2)$. Let us suppose $r_1 t_1 = \cdots = r_n t_n =: t$ and that $\{\pi(x_1), \cdots, \pi(x_n)\}$ is an orthogonal set. For each $i = 1, \cdots, n$, let
  \begin{align*}
    \pi_i : E \to E /[f_1, \cdots, f_{i-1},f_{i+1}, \cdots, f_n]
  \end{align*}
  be the canonical quotient map. Then for each $\lambda \in K$, we have
  \begin{align*}
    \|\lambda \pi_i(f_i) + \pi_i((x_1, \cdots, x_n))\| = |\lambda + x_i|_{t_i} \lor t = |\lambda + x_i|_{t_i}.
  \end{align*}
  Thus, the map
  \begin{align*}
    [e_{n+1},e_i] &\to ([1, x_i], |\cdot|_{1/t}) \\
    e_{n+1} &\mapsto 1 \\
    e_{i} &\mapsto - x_i
  \end{align*}
  gives a linear isometry by Theorem \ref{thmh10}. On the other hand, there exists a spherically complete normed space $(H,\|\cdot\|)$ and an isometry 
  \begin{align*}
    T : E' \to (K^{\lor},|\cdot|_{1/t}) \oplus (H,\|\cdot\|)
  \end{align*}
  such that $T(e_{n+1}) = (1,0)$. For each $i = 1, \cdots, n$, we put 
  \begin{align*}
    T(e_i) = (z_i, w_i) \ \text{with} \ z_i \in K^{\lor} \ \text{and} \ w_i \in H.
  \end{align*}
  Since $[e_{n+1},e_i]$ has no orthogonal base, we have
  \begin{align*}
    \|w_i\| \le \frac{d(z_i,K)}{t} \ \text{for each $i = 1, \cdots, n$}.
  \end{align*}
  Therefore, we have 
  \begin{align*}
  ([1, x_i], |\cdot|_{1/t}) \cong [e_{n+1},e_i] \cong ([1, z_i],|\cdot|_{1/t}) \ \text{and} \  |x_i + z_i| \le d(x_i, K)
  \end{align*}
  for each $i = 1, \cdots, n$. Hence by Lemma \ref{syslem1}, $\{\pi(z_1), \cdots, \pi(z_n)\}$ is an orthogonal set and the map
  \begin{align*}
    ([1, z_1, \cdots, z_n], |\cdot|_{1/t}) &\to ([1, x_1, \cdots, x_n], |\cdot|_{1/t}) \\
    1 &\mapsto 1 \\
    z_i &\mapsto - x_i \ (1 \le i \le n)
  \end{align*}
  gives a linear isometry. \par
  Now, we will prove
  \begin{align*}
    \|\lambda e_{n+1} + \lambda_1 e_1 + \cdots + \lambda_n e_n\| = |\lambda + \lambda_1 z_1 + \cdots + \lambda_n z_n|_{1/t}
  \end{align*}
  for each $\lambda, \lambda_1, \cdots, \lambda_n \in K$. Indeed, we have 
  \begin{align*}
    \|\lambda e_{n+1} + \lambda_1 e_1 + \cdots + \lambda_n e_n\| = |\lambda + \lambda_1 z_1 + \cdots + \lambda_n z_n|_{1/t} \lor \|\lambda_1 w_1 + \cdots + \lambda_n w_n\|.
  \end{align*}
  Moreover, for each $i = 1, \cdots, n$, we obtain
  \begin{align*}
    |\lambda + \lambda_1 z_1 + \cdots + \lambda_n z_n|_{1/t} &\ge \frac{|\lambda_i|}{t} d(\pi(z_i), [\pi(z_1), \cdots, \pi(z_{i-1}), \pi(z_{i+1}), \cdots, \pi(z_n)]) \\
    &= \frac{|\lambda_i|}{t} d(z_i,K) \\
    &\ge |\lambda_i| \|w_i\| = \|\lambda_i w_i\|.
  \end{align*}
  Finally, we have
  \begin{align*}
    \|\lambda_1 w_1 + \cdots + \lambda_n w_n\| \le \max_{1 \le i \le n} \|\lambda_i w_i\| \le |\lambda + \lambda_1 z_1 + \cdots + \lambda_n z_n|_{1/t}
  \end{align*}
  and 
  \begin{align*}
    \|\lambda e_{n+1} + \lambda_1 e_1 + \cdots + \lambda_n e_n\| = |\lambda + \lambda_1 z_1 + \cdots + \lambda_n z_n|_{1/t}.
  \end{align*}
  Therefore, the map
  \begin{align*}
    E' &\to ([1, x_1, \cdots, x_n], |\cdot|_{1/t}) \\
    e_i &\mapsto \left\{ 
      \begin{aligned}
      - x_i \ &; \ 1 \le i \le n \\
      1 \ &; \ i = n+1 
      \end{aligned}
    \right.
  \end{align*}
  gives a linear isometry. In particular, $E'$ is indecomposable, and hence $E$ is indecomposable. Moreover, by Theorem \ref{systhm1}, all $n$-dimensional subspaces of $E$ are isometrically isomorphic to the $n$-dimensional subspace $[f_1, \cdots, f_n]$, isometrically isomorphic to $(K^n, |\cdot|_{t_1} \times \cdots \times |\cdot|_{t_n})$. Thus, the proof is complete.
\end{proof}

\begin{cor}
  Let $E$ be an indecomposable $(n + 1)$-dimensional normed space such that $\dim_{K^{\lor}} E^{\lor} = n$. Then there exist $x_1, \cdots, x_n \in K^{\lor} \setminus K$ and $t_1, \cdots, t_n \in \mathbb{R}_{>0}$ such that $\{\pi(x_1), \cdots, \pi(x_n)\}$ is an orthogonal set and
  \begin{align*}
    E \cong [(1,0,\cdots, 0), \cdots, (0, \cdots, 0, 1), (x_1, \cdots, x_n)],
  \end{align*}
  a subspace of $((K^{\lor})^{n}, |\cdot|_{t_1} \times \cdots \times |\cdot|_{t_n})$.
\end{cor}

\begin{cor}
  Let $(E, \|\cdot\|)$ be an indecomposable normed space that satisfies $\dim_{K^{\lor}} E^{\lor} = \dim_K E - 1$. Then $E$ satisfies $(\mathrm{SE})$ if and only if $\| E \| \cap V_K = \emptyset$.
\end{cor}
\begin{proof}
  We can use Corollary \ref{syscor3} and Theorem \ref{systhm2}.
\end{proof}

By the reflexivity, we have the following.

\begin{cor} \label{syscor5}
  Let $x_1, \cdots, x_n \in K^{\lor} \setminus K$ and $t \in \mathbb{R}_{> 0}$, and let $\pi : K^{\lor} \to K^{\lor} / K$ be the canonical quotient map. Put $r_i = d(x_i, K)$ for each $i = 1, \cdots, n$. Set 
  \begin{align*}
     E := ([1, x_1, \cdots, x_n], |\cdot|_t)
  \end{align*}
  and 
  \begin{align*}
    F := [(1,0,\cdots,0), \cdots, (0, \cdots, 0, 1), (x_1, \cdots, x_n)] \subseteq ((K^{\lor})^n, |\cdot|_{\frac{1}{t r_1}} \times \cdots \times |\cdot|_{\frac{1}{t r_n}}).
  \end{align*}
  Let $e_1, \cdots, e_n, e_{n+1} \in E'$ be the dual basis of $x_1, \cdots, x_n, 1$. If $\{\pi(x_1), \cdots, \pi(x_n)\}$ is an orthogonal set, then the map $T : E' \to F$ defined by
  \begin{align*}
    T(e_i) = \left\{ 
      \begin{aligned}
      - f_i \ &; \ 1 \le i \le n \\
      (x_1, \cdots, x_n) \ &; \ i = n+1 
      \end{aligned}
    \right.
  \end{align*}
  is an isometry where $f_1, \cdots, f_n \in F$ are the canonical unit vectors.
\end{cor}

\begin{cor} \label{syscor4} 
  Let $x_1, \cdots, x_n \in K^{\lor} \setminus K$ and $t \in \mathbb{R}_{> 0}$, and let $\pi : K^{\lor} \to K^{\lor} / K$ be the canonical quotient map. Suppose that $\{\pi(x_1), \cdots, \pi(x_n)\}$ is an orthogonal set. Then the normed space
  \begin{align*}
     E := ([1, x_1, \cdots, x_n], |\cdot|_t)
  \end{align*}
  satisfies $(\mathrm{SE})$ if and only if $t \notin V_K$.
\end{cor}
\begin{proof}
  We can use the preceding corollary and Corollary \ref{syscor3}.
\end{proof}

As an application of Theorem \ref{systhm2}, we derive the following theorem which refines \cite[Proposition 2.12]{fin}.

\begin{thm}
  Let $x_1, \cdots, x_n \in K^{\lor} \setminus K$ and $t_1, \cdots, t_n \in \mathbb{R}_{> 0}$. For each $i = 1, \cdots, n$, put $r_i = d(x_i, K)$ and suppose $t_1 r_1 = \cdots = t_n r_n := t$. For each $i = 1, \cdots, n$, let $(c_{i,k})_{k \in \mathbb{N}}$ be a sequence in $K$ such that 
  \begin{align*}
    r_{i,k} := |c_{i,k} - x_i| \to r_i \  \text{as} \ k \to \infty \ \text{and $r_{i,k} \ge r_{i, k + 1}$ for all $k \in \mathbb{N}$}.  
  \end{align*}
  Set $(E, \|\cdot\|) := (K^n, |\cdot|_{t_1} \times \cdots \times |\cdot|_{t_n})$,
  \begin{align*}
    c_k = (c_{k,1}, \cdots, c_{k,n}) \in E \ \text{and} \ s_{k} := t_1 r_{1,k} \lor \cdots \lor t_n r_{n,k} \ (k \in \mathbb{N}).
  \end{align*}
  Then the following are equivalent: \\
  $(1)$ $\{\pi(x_1), \cdots, \pi(x_n)\}$ is an orthogonal set where $\pi : K^{\lor} \to K^{\lor} / K$ is the canonical quotient map. \\
  $(2)$ For each $a \in E$ and each proper linear subspace $L$ of $E$, there exists $k \in \mathbb{N}$ such that $(a + L) \cap B_E (c_k, s_k) = \emptyset.$ \\
  $(3)$ For each $i = 1, \cdots, n$ and $\lambda, \lambda_1, \cdots, \lambda_{i-1}, \lambda_{i + 1}, \cdots, \lambda_n \in K$, there exists $k_0 \in \mathbb{N}$ such that
  \begin{align*}
    |c_{i,k} - \sum_{\substack{1 \le j \le n \\ j \neq i}} \lambda_j c_{j,k} - \lambda| > r_{i,k_0} \ \text{for all $k > k_0$}.
  \end{align*}
\end{thm}
\begin{proof}
  First, we prove $(1) \Rightarrow (2)$.  Suppose $(1)$. Let $a = (a_1, \cdots, a_n), b_1, \cdots, b_{n-1} \in E$. By Theorem \ref{systhm2}, the subspace
  \begin{align*}
    [b_1, \cdots, b_{n-1}, (x_1 - a_1, \cdots, x_n - a_n)] \ \text{of} \ ((K^{\lor})^n, |\cdot|_{t_1} \times \cdots \times |\cdot|_{t_n})
  \end{align*}
  has an orthogonal base. Therefore, $d ((x_1 - a_1, \cdots, x_n - a_n), [b_1, \cdots, b_{n-1}])$ is attained. Hence, there exist $\lambda_1, \cdots, \lambda_{n-1}  \in K$ such that
  \begin{align*}
    \|(x_1 - a_1, \cdots, x_n - a_n) - \sum_{i = 1}^{n-1} \lambda_i b_i\| = d ((x_1 - a_1, \cdots, x_n - a_n), [b_1, \cdots, b_{n-1}]).
  \end{align*}
  In particular, we have $d ((x_1 - a_1, \cdots, x_n - a_n), [b_1, \cdots, b_{n-1}]) > t$. Now, let $k \in \mathbb{N}$ be such that
  \begin{align*}
    |c_{i,k} - x_i|_{t_i} < d ((x_1 - a_1, \cdots, x_n - a_n), [b_1, \cdots, b_{n-1}]) \ \text{for all $i = 1, \cdots, n$}.
  \end{align*}
  Then for each $\mu_1, \cdots, \mu_n \in K$, we have
  \begin{align*}
    \|c_k - a - \sum_{i = 1}^n \mu_i b_i\| &= \|(c_k - (x_1, \cdots, x_n)) + ((x_1 - a_1, \cdots, x_n - a_n) - \sum_{i = 1}^n \mu_i b_i)\| \\
    &= \|(x_1 - a_1, \cdots, x_n - a_n) - \sum_{i = 1}^n \mu_i b_i\| \\
    &\ge d ((x_1 - a_1, \cdots, x_n - a_n), [b_1, \cdots, b_{n-1}]) \\
    &> \max_{1 \le i \le n} |c_{i,k} - x_i|_{t_i} = s_k.
  \end{align*}
  Thus, $(a + [b_1, \cdots, b_{n-1}]) \cap B_E (c_k, s_k) = \emptyset$ and we obtain $(2)$. \par
  Secondly, we prove $(2) \Rightarrow (3)$. Suppose that $(3)$ is not satisfied. Then there exist $i, 1 \le i \le n$, and $\lambda, \lambda_1, \cdots, \lambda_{i-1}, \lambda_{i + 1}, \cdots \lambda_n \in K$ such that for all $k \in \mathbb{N}$, we have
  \begin{align*}
    |c_{i,n_k} - \sum_{\substack{1 \le j \le n \\ j \neq i}} \lambda_j c_{j,n_k} - \lambda| \le r_{i,k} \ \text{for some $n_k > k$}.
  \end{align*}
  Let $k \in \mathbb{N}$. Let $e_1, \cdots, e_n \in E$ be the canonical unit vectors, and let
  \begin{align*}
    L := [\{e_j + \lambda_j e_i : 1 \le j \le n, j \neq i\}] \ \text{and} \ x := \lambda e_i + \sum_{\substack{1 \le j \le n \\ j \neq i}} c_{j, n_k} (e_j + \lambda_j e_i) \in \lambda e_i + L.
  \end{align*}
  Then we have
  \begin{align*}
    \|x - c_{n_k}\|_E = |c_{i,n_k} - \sum_{\substack{1 \le j \le n \\ j \neq i}} \lambda_j c_{j, n_k} - \lambda|_{t_i} \le s_k,
  \end{align*}
  hence $\|x - c_k\| \le \|x - c_{n_k}\| \lor \|c_{n_k} - c_k\| \le s_k$. Therefore, we have 
  \begin{align*}
    (\lambda e_i + L) \cap B_E (c_k, s_k) \neq \emptyset \ \text{for all $k \in \mathbb{N}$},
  \end{align*}
  and $(2)$ is not satisfied. \par
  Finally, we prove $(3) \Rightarrow (1)$. Let us suppose that $(3)$ holds and that there exist $i$, $1 \le i \le n$, and $\lambda, \lambda_1, \cdots, \lambda_{i-1}, \lambda_{i + 1}, \cdots, \lambda_n \in K$ such that
  \begin{align*}
    |x_i - \sum_{\substack{1 \le j \le n \\ j \neq i}} \lambda_j x_j - \lambda| < r_i.
  \end{align*}
  Then we will derive a contradiction. By the above inequality, we have 
  \begin{align*}
    |c_{i,k} - \sum_{\substack{1 \le j \le n \\ j \neq i}} \lambda_j c_{j,k} - \lambda| &\le |c_{i,k} -  x_i - \sum_{\substack{1 \le j \le n \\ j \neq i}} \lambda_j \cdot (c_{j,k} - x_j) + (x_i - \sum_{\substack{1 \le j \le n \\ j \neq i}} \lambda_j x_j - \lambda)| \\ 
    &\le \max_{1 \le j \le n} |\lambda_j| \cdot r_{j,k} \lor r_i \ (\text{where we put $\lambda_i = 1$}) \\
    &= \max_{1 \le j \le n} |\lambda_j| \cdot r_{j,k}
  \end{align*}
  for all $k \in \mathbb{N}$. On the other hand, by $(3)$, there exists $k_0 \in \mathbb{N}$ such that 
  \begin{align*}
    |c_{i,k} - \sum_{\substack{1 \le j \le n \\ j \neq i}} \lambda_j c_{j,k} - \lambda| > \max_{1 \le j \le n} |\lambda_j| \cdot r_{j,k_0} \ \text{for all $k > k_0$},
  \end{align*}
  which is a contradiction.
\end{proof}

The following theorem is proved in \cite[Theorem 2.10]{fin}.

\begin{thm}[{\cite[Theorem 2.10]{fin}}] \label{systhm5}
  Suppose that $K$ is separable as a metric space. Then for each $r \in \mathbb{R}_{> 0}$ and each $n \in \mathbb{N}$, there exist $x_1, \cdots, x_n \in K^{\lor} \setminus K$ such that
  \begin{itemize}
    \item $d(x_i,K) = r$ for each $i = 1, \cdots, n$.
    \item $\{\pi(x_1), \cdots, \pi(x_n)\}$ is an orthogonal set where $\pi: K^{\lor} \to K^{\lor} / K$ is the canonical quotient map.
  \end{itemize}
\end{thm}

\subsection{Hyper-symmetric spaces} \label{hyper}

\begin{dfn}
  A hyper-symmetric space is recursively defined as follows: \\
  $(1)$ All zero-dimensional normed spaces are hyper-symmetric. \\
  $(2)$ Let $n \ge 1$. An $n$-dimensional normed space $(E,\|\cdot\|)$ is hyper-symmetric if $\dim_{K^{\lor}} E^{\lor} = 1$ and the $(n - 1)$-dimensional normed space $E / [u]$, $u \neq 0$, is hyper-symmetric. \par
  Notice that this definition is well-defined by Corollary \ref{syscor1}.
\end{dfn}

First, we show the heredity of a hyper-symmetric normed space.

\begin{thm} \label{systhm8}
  Let $E$ be a hyper-symmetric normed space. \\
  $(1)$ All quotient spaces of $E$ are hyper-symmetric. \\
  $(2)$ All subspaces of $E$ are hyper-symmetric. \\
  $(3)$ The dual space $E'$ of $E$ is hyper-symmetric.
\end{thm}
\begin{proof}
  By definition, $(1)$ follows. First, we shall prove $(2)$ by induction on $n = \dim_K E$. Let $F$ be a subspace of $E$ with $F \neq 0$ and $F \neq E$. Obviously, we have $\dim_{K^{\lor}} F^{\lor} = 1$. Let $u \in F$ be a non-zero element. Since $E/ [u]$ is an $(n-1)$-dimensional hyper-symmetric normed space, a subspace $F / [u]$ is hyper-symmetric by the induction hypothesis. Therefore, $F$ is hyper-symmetric as desired. \par
  Finally, we prove $(3)$ again by induction on $n = \dim_K E$. If $n = 1, 2$, we clearly have $(3)$. Suppose $n \ge 3$. For a subspace $F$ of $E'$, we set
  \begin{align*}
    F_{\perp} = \{x \in E : f(x) = 0 \ \text{for all $f \in F$}\}.
  \end{align*}
  Let $F$ be a $2$-dimensional subspace of $E'$. Then $F \cong (E / F_{\perp})'$ and therefore $F$ is hyper-symmetric by the induction hypothesis. In particular, $\dim_{K^{\lor}} F^{\lor} = 1$ for each $2$-dimensional subspace $F$. Thus, we obtain $\dim_{K^{\lor}} (E')^{\lor} = 1$. Now, let $f \in E'$ be a non-zero element. Then $E' / [f] \cong ([f]_{\perp})'$, and hence $E' / [f]$ is hyper-symmetric by the induction hypothesis.
\end{proof}

The following theorem says that a hyper-symmetric space surely has symmetry.

\begin{thm} \label{systhm7}
  Let $E$ be an $n$-dimensional hyper-symmetric space. Let $m \in \mathbb{N}$ with $1 \le m \le n-1$. \\
  $(1)$ All $m$-dimensional quotient spaces of $E$ are isometrically isomorphic to each other. \\
  $(2)$ All $m$-dimensional subspaces of $E$ are isometrically isomorphic to each other.
\end{thm}
\begin{proof}
  By Theorem \ref{systhm1}, it is easy to see that $(1)$ holds. On the other hand, $E'$ is hyper-symmetric by the preceding theorem. Hence, $(2)$ also follows.
\end{proof}

Now, by Corollary \ref{syscor3}, we obtain a necessary and sufficient condition for a hyper-symmetric space to satisfy (SE). 

\begin{thm} \label{hypthm1}
  Let $(E, \|\cdot\|)$ be an $n$-dimensional hyper-symmetric space and let $u_1, \cdots, u_{n-2}$ be linearly independent elements of $E$. Then $E$ satisfies $(\mathrm{SE})$ if and only if
  \begin{align*}
    \|E\| \cap V_K = \emptyset \ \text{and} \ \|E / [u_1, \cdots, u_k]\| \cap V_K = \emptyset \ \text{for all $k = 1, \cdots, n-2$}.
  \end{align*}
\end{thm}

\subsection{$E = (K,|\cdot|_t) \oplus (F, \|\cdot\|_F)$} \label{decomp} We shall study a decomposable normed space of the form $(K,|\cdot|_t) \oplus (F, \|\cdot\|_F)$. First, we will identify all subspaces. 

\begin{thm} \label{systhm3}
  Let $(F, \|\cdot\|_F)$ be an $n$-dimensional normed space and $t > 0$. Set $E := (K,|\cdot|_t) \oplus (F, \|\cdot\|_F)$. Then each $m$-dimensional subspace $E_1$ of $E$, $1 \le m \le n$,  satisfies either \\
  $(1)$ the restriction of the second projection to $E_1$ is an isometry \\
  or \\
  $(2)$ $E_1 \cong (K,|\cdot|_t) \oplus F_1$ for some $(m - 1)$-dimensional subspace $F_1$ of $F$.
\end{thm}
\begin{proof}
  We may assume $E_1 \nsubseteq \{0\} \oplus F$. Then we can choose a basis $x_1, \cdots, x_m$ of $E_1$ such that
  \begin{align*}
    x_1 = (1, u), x_2 = (0, u_1), \cdots, x_m = (0, u_{m-1}) \ \text{with} \ u, u_1, \cdots, u_{m-1} \in F.
  \end{align*}
  If $|1|_t \le d_F (u, [u_1, \cdots, u_{m-1}])$, then the restriction of the second projection to $E_1$ is an isometry. If $|1|_t > d_F (u, [u_1, \cdots, u_{m-1}])$, then we can choose $u$ such that $\|u\|_F < |1|_t$. Hence, we have $E_1 \cong (K,|\cdot|_t) \oplus [u_1, \cdots, u_{m-1}]$.
\end{proof}

Notice that the following theorem is not trivial because it is not known that a direct sum of normed spaces satisfying (SE) also satisfies (SE).

\begin{thm} \label{systhm4}
  Let $(F, \|\cdot\|_F)$ be a normed space and $t > 0$. Set $E := (K,|\cdot|_t) \oplus (F, \|\cdot\|_F)$. Then $E$ satisfies $(\mathrm{SE})$ if and only if $F$ satisfies $(\mathrm{SE})$.
\end{thm}
\begin{proof}
  Obviously, if $E$ satisfies $(\mathrm{SE})$, then $F$ satisfies $(\mathrm{SE})$. We prove the converse by induction on $n = \dim_K E$. If $n = 2$, the statement is clear. Suppose $n \ge 3$. Let $E_1$ be an $(n-1)$-dimensional subspace of $E$. Then by the preceding theorem and the induction hypothesis, $E_1$ satisfies $(\mathrm{SE})$. Moreover, by the preceding theorem, $E$ satisfies the condition $(2)$ in Lemma \ref{selem1}. Thus, $E$ satisfies (SE).
\end{proof}

By Theorem \ref{systhm3}, we obtain a structure theorem for finite-dimensional normed spaces of the form $(K,|\cdot|_t) \oplus (F, \|\cdot\|_F)$.

\begin{thm} \label{systhm6}
  Let $(F, \|\cdot\|_{F})$ and $(G, \|\cdot\|_{G})$ be two indecomposable normed spaces with $\dim_K F \ge 2$, and let $t_1, \cdots, t_m, s_1, \cdots, s_n \in \mathbb{R}_{> 0}$. Then
  \begin{align*}
    (K^m,|\cdot|_{t_1} \times \cdots \times |\cdot|_{t_m}) \oplus (F, \|\cdot\|_{F}) \cong (K^n,|\cdot|_{s_1} \times \cdots \times |\cdot|_{s_n}) \oplus (G, \|\cdot\|_{G})
  \end{align*}
  if and only if
  \begin{align*}
    (K^m,|\cdot|_{t_1} \times \cdots \times |\cdot|_{t_m}) \cong (K^n,|\cdot|_{s_1} \times \cdots \times |\cdot|_{s_n}) \ \text{and} \ (F, \|\cdot\|_{F}) \cong (G, \|\cdot\|_{G}).
  \end{align*}
\end{thm}
\begin{proof}
  Suppose that there exists an isometry $T$ from $(K^m,|\cdot|_{t_1} \times \cdots \times |\cdot|_{t_m}) \oplus (F, \|\cdot\|_{F})$ to $(K^n,|\cdot|_{s_1} \times \cdots \times |\cdot|_{s_n}) \oplus (G, \|\cdot\|_{G})$. Since $T(\{0\} \oplus F)$ is indecomposable with $\dim_K T(\{0\} \oplus F) \ge 2$, by Theorem \ref{systhm3} we have
  \begin{align*}
    \|u\|_{(K^n,|\cdot|_{s_1} \times \cdots \times |\cdot|_{s_n})} \le \|v\|_G \ \text{for all $(u,v) \in T(\{0\} \oplus F)$}.
  \end{align*}
  In particular, we have $\dim_K G \ge \dim_K F \ge 2$. Thus, considering $T^{-1}$, we obtain $\dim_K F = \dim_K G$ and $F \cong G$. Moreover, for each $x \in (K^n,|\cdot|_{s_1} \times \cdots \times |\cdot|_{s_n}) \oplus \{0\}$ and $y \in T(\{0\} \oplus F)$, $\{x,y\}$ is an orthogonal set. Therefore, since $T$ is an isometry, 
  \begin{align*}
    \|w\|_{(K^m,|\cdot|_{t_1} \times \cdots \times |\cdot|_{t_m})} \ge \|z\|_F \ \text{for all $(w,z) \in T^{-1}((K^n,|\cdot|_{s_1} \times \cdots \times |\cdot|_{s_n}) \oplus \{0\})$}.
  \end{align*}
  Thus, there exists an isometry from $(K^n,|\cdot|_{s_1} \times \cdots \times |\cdot|_{s_n})$ to $(K^m,|\cdot|_{t_1} \times \cdots \times |\cdot|_{t_m})$. Since $m = n$, this is an isometric isomorphism.
  \end{proof}

  Obviously, by combining Theorem \ref{.1thm1}, we can refine the theorem above.

\section{$3$-dimensional normed spaces} \label{3d}

We will apply the systematic approaches obtained in section \ref{sys}. \par
We recall that our strategy when studying finite-dimensional normed spaces is the following: \\
$(1)$ embed each finite-dimensional normed space into $((K^{\lor})^n, |\cdot|_{t_1} \times \cdots \times |\cdot|_{t_n})$. \\
$(2)$ study properties according to an embedding into $((K^{\lor})^n, |\cdot|_{t_1} \times \cdots \times |\cdot|_{t_n})$. \par
In this section, we study a $3$-dimensional normed space $E$. Based on $\dim_{K^{\lor}}E^{\lor}$, the cardinality of a maximal orthogonal subset, we will classify $3$-dimensional normed spaces into five types, type $\mathrm{\Rnum{1}}_3 \sim \mathrm{\Rnum{5}}_3$. 

\begin{dfn}
  A $3$-dimensional normed space $E$ is
  \begin{itemize}
  \setlength{\leftskip}{-20pt}
    \item of type $\mathrm{\Rnum{1}}_3$ if $\dim_{K^{\lor}}E^{\lor} = 3$.
    \item of type $\mathrm{\Rnum{2}}_3$ if $\dim_{K^{\lor}}E^{\lor} = 2$ and $E$ is decomposable.
    \item of type $\mathrm{\Rnum{3}}_3$ if $\dim_{K^{\lor}}E^{\lor} = 2$ and $E$ is indecomposable.
    \item of type $\mathrm{\Rnum{4}}_3$ if $\dim_{K^{\lor}}E^{\lor} = 1$ and $E/[u]$ has an orthogonal base for each non-zero element $u \in E$.
    \item of type $\mathrm{\Rnum{5}}_3$ if $\dim_{K^{\lor}}E^{\lor} = 1$ and $E/[u]$ has no orthogonal base for each non-zero element $u \in E$.
  \end{itemize}
\end{dfn}

By Corollary \ref{syscor1}, we have the following.

\begin{thm}
  Every $3$-dimensional normed space $E$ is exactly of one type of type $\mathrm{\Rnum{1}}_3 \sim \mathrm{\Rnum{5}}_3$.
\end{thm}

\subsection{Type $\mathrm{\protect \Rnum{1}}_3$}

\begin{thm}
  Let $E$ be a $3$-dimensional normed space of type $\mathrm{\Rnum{1}}_3$. \\
  $(1)$ $E$ satisfies $\mathrm{(SE)}$. \\
  $(2)$ There exist $t_1, t_2, t_3 \in \mathbb{R}_{>0}$ for which 
\begin{align*}
  E \cong (K^3, |\cdot|_{t_1} \times |\cdot|_{t_2} \times |\cdot|_{t_3}).
\end{align*}
Moreover, it follows that
\begin{align*}
  E' \cong (K^3, |\cdot|_{\frac{1}{t_1}} \times |\cdot|_{\frac{1}{t_2}} \times |\cdot|_{\frac{1}{t_3}})
\end{align*}
and $E'$ is of type $\mathrm{\Rnum{1}}_3$.
\end{thm}

A structure theorem for $3$-dimensional normed spaces of type $\mathrm{\Rnum{1}}_3$ is immediately obtained by Theorem \ref{.1thm1}. 

\begin{thm} 
Let $t_1, t_2, t_3, s_1, s_2, s_3 \in \mathbb{R}_{> 0}$. Then
  \begin{align*}
    (K^3, |\cdot|_{t_1} \times |\cdot|_{t_2} \times |\cdot|_{t_3}) \cong (K^3, |\cdot|_{s_1} \times |\cdot|_{s_2} \times |\cdot|_{s_3}) 
  \end{align*}
  if and only if by a suitable permutation of $s_1, s_2, s_3$, we have $t_i/s_i \in V_K$ for each $i=1, 2, 3$.
\end{thm}

\subsection{Type $\mathrm{\protect \Rnum{2}}_3$}

By definition, we have the following.

\begin{thm} \label{.4thm7}
  Let $E$ be a $3$-dimensional normed space of type $\mathrm{\Rnum{2}}_3$. Then there exist $x \in K^{\lor} \setminus K$ and $t_1,t_2 \in \mathbb{R}_{>0}$ for which
\begin{align*}
  E \cong (K,|\cdot|_{t_1}) \oplus ([1, x], |\cdot|_{t_2}). 
\end{align*}
Moreover, it follows that
\begin{align*}
  E' \cong (K,|\cdot|_{\frac{1}{t_1}}) \oplus ([1, x], |\cdot|_{\frac{1}{t_2 r}}) \ \text{where} \  r = d(x,K),
\end{align*}
and $E'$ is of type $\mathrm{\Rnum{2}}_3$.
\end{thm}
\begin{proof}
  The theorem follows from Theorem \ref{thmh10}.
\end{proof}

Thus, we can apply the results of section \ref{decomp} to a normed space of type $\mathrm{\Rnum{2}}_3$.

\begin{thm} \label{.4thm6} 
  Let $x \in K^{\lor} \setminus K$ and $t_1,t_2 \in \mathbb{R}_{>0}$, and put 
  \begin{align*}
    E := (K,|\cdot|_{t_1}) \oplus ([1, x], |\cdot|_{t_2}).
  \end{align*}
  Then for each $2$-dimensional subspace $F \subseteq E$, $F$ is isometrically isomorphic to either
  \begin{align*}
    ([1, x], |\cdot|_{t_2}) \ \mathrm{or}\ (K^2, |\cdot|_{t_1} \times |\cdot|_{t_2}).
  \end{align*}
\end{thm}
\begin{proof}
  We can use Theorem \ref{systhm3}.
\end{proof}

\begin{thm}
  Let $x,y \in K^{\lor} \setminus K$ and $t_1,t_2,s_1,s_2 \in \mathbb{R}_{>0}$. Then
  \begin{align*}
    (K,|\cdot|_{t_1}) \oplus ([1, x], |\cdot|_{t_2}) \cong (K,|\cdot|_{s_1}) \oplus ([1, y], |\cdot|_{s_2}) 
  \end{align*}
  if and only if
  \begin{align*}
    t_1/s_1, t_2/s_2 \in V_K \ \mathrm{and} \ x \sim y.
  \end{align*}
\end{thm}
\begin{proof}
  By Theorem \ref{systhm6}, the conclusion follows.
\end{proof}

Finally, by Theorem \ref{systhm4}, we have the following theorem.

\begin{thm} \label{.4thm4}
  Let $x \in K^{\lor} \setminus K$ and $t_1,t_2 \in \mathbb{R}_{>0}$, and put 
  \begin{align*}
    E := (K,|\cdot|_{t_1}) \oplus ([1, x], |\cdot|_{t_2}).
  \end{align*}
  Then $E$ satisfies $(\mathrm{SE})$ if and only if $t_2 \notin V_K$.
\end{thm}

At the end of this section, we prove some lemmas needed later.

\begin{lem} \label{lemh3}
  Let $x \in K^{\lor} \setminus K$ and $t_1,t_2 \in \mathbb{R}_{>0}$, and put 
  \begin{align*}
    E := (K,|\cdot|_{t_1}) \oplus ([1, x], |\cdot|_{t_2}).
  \end{align*}
  Then for any $\lambda \in K$, $[(1, \lambda), (0, x)] \cong ([1, x], |\cdot|_{t_2})$ if and only if $t_1 \le \tfrac{|\lambda|t_2 d(x,K)}{|x|}$.
\end{lem}
\begin{proof}
  By the proof of Theorem \ref{systhm3}, $[(1, \lambda), (0, x)] \cong ([1, x], |\cdot|_{t_2})$ if and only if
    \begin{align*}
    |1|_{t_1} \le d_{(K^{\lor}, |\cdot|_{t_2})} (\lambda, [x]).
  \end{align*}
  It is easy to see that the last condition is equivalent to $t_1 \le \tfrac{|\lambda|t_2 d(x,K)}{|x|}$.
\end{proof}

\begin{lem} \label{lemh4}
   Let $x \in K^{\lor} \setminus K$ and $t_1,t_2 \in \mathbb{R}_{>0}$, and put 
  \begin{align*}
    E := (K,|\cdot|_{t_1}) \oplus ([1, x], |\cdot|_{t_2}).
  \end{align*}
  Then for any $\lambda, \mu \in K$ with $|\mu|_{t_2} \le |\lambda|_{t_1}$, the map
  \begin{align*}
    E / [(\lambda, \mu)] &\to ([1, x], |\cdot|_{t_2}) \\
    \pi((0,1)) &\mapsto 1 \\
    \pi((0,x)) &\mapsto x
  \end{align*}
  gives a linear isometry where $\pi : E \to E / [(\lambda, \mu)]$ is the canonical quotient map.
\end{lem}
\begin{proof}
  Since $|\mu|_{t_2} \le |\lambda|_{t_1}$, for each non-zero $x' \in ([1, x], |\cdot|_{t_2})$, $\{(0, x'), (\lambda, \mu)\}$ is an orthogonal set. Now, the lemma easily follows. 
\end{proof}

\subsection{Type $\mathrm{\protect \Rnum{3}}_3$, $\mathrm{\protect \Rnum{4}}_3$} \label{type3,4}

If $E$ is a $3$-dimensional normed space of type $\mathrm{\Rnum{3}}_3$, then there exist $x,y \in K^{\lor} \setminus K$ and $t_1,t_2 \in \mathbb{R}_{>0}$ for which
\begin{align*}
  E \cong ([(1,0),(0,1),(x,y)]) \subseteq ((K^{\lor})^2,|\cdot|_{t_1} \times |\cdot|_{t_2}).
\end{align*}
Conversely, for each $x,y \in K^{\lor} \setminus K$ and $t_1,t_2 \in \mathbb{R}_{>0}$, a $3$-dimensional subspace
\begin{align*}
  ([(1,0),(0,1),(x,y)]) \subseteq ((K^{\lor})^2,|\cdot|_{t_1} \times |\cdot|_{t_2})
\end{align*}
is either of type $\mathrm{\Rnum{2}}_3$ or $\mathrm{\Rnum{3}}_3$. Then by Theorem \ref{systhm2}, we have the following.

\begin{thm} \label{.4thm2}
  Let $x,y \in K^{\lor} \setminus K$ and $t_1,t_2 \in \mathbb{R}_{>0}$, and put $r = d(x,K), s=d(y,K)$. We set
  \begin{align*}
    (E,\|\cdot\|) := [(1,0),(0,1),(x,y)],
  \end{align*}
  a subspace of $((K^{\lor})^2,|\cdot|_{t_1} \times |\cdot|_{t_2})$. \\
  $(1)$ $E$ is of type $\mathrm{\Rnum{3}}_3$ if and only if $ t_1 r = t_2 s$ and $x \nsim y$. \\
  $(2)$ Suppose that $E$ is of type $\mathrm{\Rnum{3}}_3$. Then $E$ satisfies $(\mathrm{SE})$ if and only if $t_1,t_2 \notin V_K$. \\
  $(3)$ Suppose that $E$ is of type $\mathrm{\Rnum{3}}_3$. Then each $2$-dimensional subspace of $E$ is isometrically isomorphic to
\begin{align*}
  (K^2,|\cdot|_{t_1} \times |\cdot|_{t_2}).
\end{align*} 
Moreover, let $e_1, e_2, e_3 \in E'$ be the dual basis of $(x,y),(1,0),(0,1)$. Then the map
  \begin{align*}
    E' &\to ([1, x, y], |\cdot|_{\frac{1}{t}}) \\
    e_1 &\mapsto 1 \\
    e_2 &\mapsto -x \\
    e_3 &\mapsto -y
  \end{align*}
  gives a linear isometry.
\end{thm}

For the following remark, see the proof of Theorem \ref{systhm2}.

\begin{rem} \label{sysrem1}
  In Theorem \ref{.4thm2}, if $t_1 r > t_2 s$ (resp. $t_1 r < t_2 s$), then $E$ is isometrically isomorphic to
  \begin{align*}
    ([1, x], |\cdot|_{t_1}) \oplus (K, |\cdot|_{t_2}) \ (\text{resp}. \ (K, |\cdot|_{t_1}) \oplus ([1, y], |\cdot|_{t_2})).
  \end{align*}
  If $x \sim y$ and $t_1 r = t_2 s$, then $E$ is isometrically isomorphic to
  \begin{align*}
    ([1, x], |\cdot|_{t_1}) \oplus (K, |\cdot|_{t_2}) \ (\cong (K, |\cdot|_{t_1}) \oplus ([1, y], |\cdot|_{t_2})).
  \end{align*}
\end{rem}

\begin{cor} \label{.4cor3}
   Let $x,y \in K^{\lor} \setminus K$ and $t_1,t_2 \in \mathbb{R}_{>0}$. We set
  \begin{align*}
    (E,\|\cdot\|) := ([(1,0),(0,1),(x,y)], |\cdot|_{t_1} \times |\cdot|_{t_2}),
  \end{align*}
  and suppose that $E$ is of type $\mathrm{\Rnum{3}}_3$. Then for each $\lambda, \mu \in K$, $\mu \neq 0$, the map
  \begin{align*}
    E / [(\lambda, \mu)] &\to ([1, x - \tfrac{\lambda}{\mu} y], |\cdot|_{\tfrac{t}{\beta}}) \\
    \pi((1,0)) &\mapsto 1 \\
    \pi((x,y)) &\mapsto x - \tfrac{\lambda}{\mu} y
  \end{align*}
  gives a linear isometry where $t := t_1 d(x, K) = t_2 d(y,K)$, $\beta = d(x - \tfrac{\lambda}{\mu}y, K)$ and $\pi : E \to E / [(\lambda, \mu)]$ is the canonical quotient map.
\end{cor}
\begin{proof}
  Let $e_1, e_2, e_3 \in E'$ be the dual basis of $(x,y), (1, 0), (0, 1)$. It is easy to see that $e_1, e_2 - \tfrac{\lambda}{\mu} e_3$ is the dual basis of $\pi((x, y)), \pi((1,0))$ via the natural isomorphism $[(\lambda, \mu)]^{\perp} \to (E / [(\lambda, \mu)])'$. On the other hand, by Theorem \ref{.4thm2}, the map
  \begin{align*}
    [(\lambda, \mu)]^{\perp} &\to ([1, x - \tfrac{\lambda}{\mu} y], |\cdot|_{\tfrac{1}{t}}) \\
    e_1 &\mapsto 1 \\
    e_2 - \tfrac{\lambda}{\mu} e_3 &\mapsto -x + \tfrac{\lambda}{\mu} y
  \end{align*}
  gives a linear isometry. Now, by Theorem \ref{thmh10} and the natural isomorphism from $E / [(\lambda, \mu)]$ to $(E / [(\lambda, \mu)])''$, the corollary holds.
\end{proof}

The following corollary immediately follows from Theorem \ref{.4thm2}.

\begin{cor}[cf. {\cite[Theorem 1.14]{note}}]  \label{.4cor2}
  Let $E$ be a $3$-dimensional normed space of type $\mathrm{\Rnum{3}}_3$. Then for each $u \in E$, $E/[u]$ has no orthogonal base.
\end{cor}

Next, we will study $3$-dimensional normed spaces $E$ with $\dim_{K^{\lor}} E^{\lor} = 1$. If a $3$-dimensional normed space $E$ satisfies $\dim_{K^{\lor}} E^{\lor} = 1$, then there exist $x,y \in K^{\lor} \setminus K$ and $t \in \mathbb{R}_{>0}$ such that 
\begin{align*}
  (E,\|\cdot\|) \cong ([1, x, y], |\cdot|_{t}).
\end{align*}
By Corollaries \ref{syscor2} and \ref{syscor4}, we have the following theorem.

\begin{thm} \label{.4thm3}
  Let $x,y \in K^{\lor} \setminus K$ and $t \in \mathbb{R}_{>0}$. We set $E := ([1, x, y], |\cdot|_{t})$. \\
  $(1)$ $E$ is of type $\mathrm{\Rnum{4}}_3$ if and only if there exist $z, w \in [1, x, y]$ such that $z \nsim w$. \\
  $(2)$ Suppose that $E$ is of type $\mathrm{\Rnum{4}}_3$. Then $E$ satisfies $\mathrm{(SE)}$ if and only if $t \notin V_K$. \\
  $(3)$ Suppose that $E$ is of type $\mathrm{\Rnum{4}}_3$. Then for each $2$-dimensional subspace F of $E$, there exists $z \in [1, x, y]$ such that
  \begin{align*}
    F \cong ([1,z], |\cdot|_t)
  \end{align*}
  Moreover, if $x \nsim y$, then we have
  \begin{align*}
    E' \cong [(1,0),(0,1),(x,y)] \subseteq ((K^{\lor})^2, |\cdot|_{\frac{1}{tr}} \times |\cdot|_{\frac{1}{ts}})
  \end{align*}
  where $r = d(x,K), s = d(y,K)$.
\end{thm}

\begin{cor} \label{.4cor1}
  If $E$ is a $3$-dimensional normed space of type $\mathrm{\Rnum{3}}_3$ (resp. $\mathrm{\Rnum{4}}_3$), then $E'$ is of type $\mathrm{\Rnum{4}}_3$ (resp. $\mathrm{\Rnum{3}}_3$). 
\end{cor}

Now, we obtain a structure theorem for normed spaces of type $\mathrm{\Rnum{4}}_3$.

\begin{thm}
  Let $x, y, z, w \in K^{\lor} \setminus K$ and $t, s \in \mathbb{R}_{>0}$. Suppose $x \nsim y$ and $z \nsim w$. Then
  \begin{align*}
    ([1, x, y], |\cdot|_t) \cong ([1, z, w], |\cdot|_s)
  \end{align*}
  if and only if $t / s \in V_K$ and there exist $x', y' \in [1, x, y]$ satisfying $x' \sim z$ and $y' \sim w$.
\end{thm}
\begin{proof}
  Suppose that there exists an isometry $T$ from $([1, x, y], |\cdot|_t)$ to $([1, z, w], |\cdot|_s)$. Obviously, we have $t / s \in V_K$. By Theorem \ref{.4thm3}, there exists $x' \in  [1, x, y]$ such that 
  \begin{align*}
    ([1, x'], |\cdot|_t) \cong T^{-1}([1, z]) \cong ([1, z], |\cdot|_t).
  \end{align*}
  Hence, we have $x' \sim z$. Similarly, there exists $y' \in [1, x, y]$ for which $y' \sim w$. The converse follows from Lemma \ref{syslem1}.
\end{proof}

\begin{rem} \label{rem3}
  By Theorem \ref{.4thm2}, we also obtain a structure theorem for normed spaces of type $\mathrm{\Rnum{3}}_3$.
\end{rem}

\subsection{Type $\mathrm{\protect \Rnum{5}}_3$} \label{type5}

Finally, we study $3$-dimensional normed spaces of type $\mathrm{\Rnum{5}}_3$. By definition, we have the following theorem.

\begin{thm}
  A $3$-dimensional normed space $E$ is of type $\mathrm{\Rnum{5}}_3$ if and only if it is hyper-symmetric.
\end{thm}

By the theorem above, we can apply the results of section \ref{hyper} to a normed space of type $\mathrm{\Rnum{5}}_3$.

\begin{prop} \label{.4prop3}
  If $E$ is a $3$-dimensional normed space of type $\mathrm{\Rnum{5}}_3$, then $E'$ is of type $\mathrm{\Rnum{5}}_3$. 
\end{prop}
\begin{proof}
  We can use Theorem \ref{systhm8}.
\end{proof}

\begin{thm} \label{.4thm1}
   Let $E$ be a $3$-dimensional normed space of type $\mathrm{\Rnum{5}}_3$. Then all $2$-dimensional subspaces of $E$ are isometrically isomorphic to each other. Moreover, all $2$-dimensional quotient spaces of $E$ are isometrically isomorphic to each other.
\end{thm}
\begin{proof}
  The theorem follows from Theorem \ref{systhm7}
\end{proof}

\begin{prop}\label{.4prop6}
  Let $x,y,z \in K^{\lor} \setminus K$ and $t_1,t_2 \in \mathbb{R}_{>0}$. Suppose that 
  \begin{align*}
    (E,\|\cdot\|) := ([1, x, y],|\cdot|_{t_1})
  \end{align*}
  is of type $\mathrm{\Rnum{5}}_3$ and
  \begin{align*}
    E/K \cong ([1, z],|\cdot|_{t_2}).
  \end{align*}
  Then we have $t_1 r / t_2 \in V_K$ where $r = d(x,K)$.
\end{prop}
\begin{proof}
  Since $r = d(x,K)$, we have $\|E/K\| = t_1 r |K|$. 
\end{proof}

By Theorem \ref{.4thm1} and Proposition \ref{.4prop6}, the following definition is well-defined.

\begin{dfn}
  Let $E$ be a $3$-dimensional normed space of type $\mathrm{\Rnum{5}}_3$. Let $x,y \in K^{\lor} \setminus K$ and $t \in \mathbb{R}_{>0}$. Then $E$ is said to be a $3$-dimensional normed space of subtype $(x,y,t)$ if $E$ satisfies
  \begin{itemize}
    \item $F \cong ([1, x], |\cdot|_{t})$ for all $2$-dimensional subspaces $F \subseteq E$,
    \item $E/[u] \cong ([1, y], |\cdot|_{tr})$ for all non-zero elements $u \in E$
  \end{itemize}
  where $r = d(x,K)$.
\end{dfn}

By definition and Theorem \ref{hypthm1}, we have the following.

\begin{thm} \label{.4thm5}
  Let $E$ be of subtype $(x,y,t)$. Then $E$ satisfies $(\mathrm{SE})$ if and only if
  \begin{align*}
    t \notin V_K \ \mathrm{and} \ tr \notin V_K \ \mathrm{where} \ r = d(x,K).
  \end{align*}
\end{thm}

\begin{prop} \label{.4prop1}
  Let $x,y \in K^{\lor} \setminus K$ and $ t \in \mathbb{R}_{>0}$. Then there exists a $3$-dimensional normed space of type $\mathrm{\Rnum{5}}_3$ such that $E$ is of subtype $(x,y,t)$.
\end{prop}
\begin{proof}
  Let $r = d(x,K)$ and $\pi : (K^{\lor},|\cdot|) \to K^{\lor}/K$ be the canonical quotient map. Then since $K^{\lor}/K$ is spherically complete, there exists a linear isometry 
  \begin{align*}
    T : (K^{\lor},|\cdot|_{r}) \to K^{\lor}/K
  \end{align*} 
  such that $T(1) = \pi(x)$. Let $z \in K^{\lor}$ with $\pi(z) = T(y)$. Then it is easy to see that $([1, x, z], |\cdot|_{t})$ is of subtype $(x,y,t)$.
\end{proof}

\begin{prop}
   Let $(E, \|\cdot\|)$ be a $3$-dimensional normed space of subtype $(x,y,t)$. Then $E'$ is of subtype $(y,x,1/trs)$ where $r = d(x,K), s= d(y,K)$. 
\end{prop}
\begin{proof}
Let $u \in E$ be a non-zero element. Then we have 
\begin{align*}
  \|E'\| = \|(E/[u])'\| = \|([1, y],|\cdot|_{tr})'\| = \frac{1}{trs} |K|
\end{align*} 
by Theorem \ref{thmh10}. Therefore, we see that $E'$ is of subtype $(y,x,1/trs)$.
\end{proof}

 Let $E_1$ be of subtype $(x,y,t)$ and $E_2$ be of subtype $(x',y',t')$. If $E_1 \cong E_2$, then it is clear that we have $x \sim x', y \sim y'$ and $t/t' \in V_K$. However, the converse is unknown.

\begin{prob*} \label{prob1}
  Let $E_1$ and $E_2$ be two $3$-dimensional normed spaces of type $\mathrm{\Rnum{5}}_3$. Suppose that $E_1$ and $E_2$ are of the same subtype. Then $E_1 \cong E_2$ ?
\end{prob*}

Here, we shall give some observations to attack the above problem.

Notice that if $E$ is of subtype $(x,y,t)$, then there exists $z \in K^{\lor} \setminus K$ such that
\begin{align*}
  E \cong ([1, x, z], |\cdot|_{t}).
\end{align*}

\begin{prop} \label{.4prop7}
  Let $x,y,y',z,w \in K^{\lor}$, and suppose that $([1, x, z], |\cdot|_t)$ is of subtype $(x,y,t)$ and $([1, x, w], |\cdot|_t)$ is of subtype $(x,y',t)$. \par Suppose $y \sim y'$. If $[\pi(x),\pi(z),\pi(w)]$ is an immediate extension of $[\pi(x),\pi(z)]$ where $\pi: (K^{\lor},|\cdot|) \to K^{\lor}/K$ is the canonical quotient map, then we have
  \begin{align*}
    ([1, x, z], |\cdot|_t) \cong ([1, x, w], |\cdot|_t).
  \end{align*}
\end{prop}
\begin{proof}
  Since $([1, x, z], |\cdot|_t)$ is of type $\mathrm{\Rnum{5}}$, we have $\dim_{K^{\lor}} ([\pi(x),\pi(z)])^{\lor} = 1$. Therefore, by assumption, we also get $\dim_{K^{\lor}} ([\pi(x),\pi(z),\pi(w)])^{\lor} = 1$. Thus, $[\pi(x),\pi(z),\pi(w)]$ is isometrically isomorphic to a subspace of $(K^{\lor},|\cdot|_r)$. Hence, it follows from
  \begin{align*}
    [\pi(x),\pi(z)] \cong ([1, y],|\cdot|_r) \cong ([1, y'],|\cdot|_r) \cong [\pi(x),\pi(w)]
  \end{align*}
  that there exist $\lambda, \mu \in K$, $\lambda \neq 0$, such that
  \begin{align*}
    \|\lambda \pi(w) + \mu \pi(x) - \pi(z)\| < d(\pi(z),[\pi(x)]).
  \end{align*}
  Then we can choose $\eta \in K$ such that
  \begin{align*}
    |\lambda w + \mu x + \eta - z| < d(\pi(z),[\pi(x)]) = d(z,[1,x])
  \end{align*}
  Let us put $w' = \lambda w + \mu x + \eta$. Then for each $\eta_1, \eta_2 \in K$, we obtain the equality
  \begin{align*}
    |\eta_1 +\eta_2 x + w'| = |\eta_1 +\eta_2 x + z + (w' - z)| = |\eta_1 +\eta_2 x + z|.
  \end{align*}
  Therefore, the map
  \begin{align*}
    ([1, x, z], |\cdot|_t) &\to ([1, x, w], |\cdot|_t) \\
    1 &\mapsto 1 \\
    x &\mapsto x \\
    z & \mapsto w'
  \end{align*}
  satisfies the requirement.
\end{proof}

\begin{cor}
  Suppose that the set of all equivalence classes of the relation $\sim$ is a singleton, that is, we have $x \sim y$ for each $x,y \in K^{\lor} \setminus K$. Then Problem \ref{prob1} is true.
\end{cor}
\begin{proof}
  We can use the preceding theorem and Theorem \ref{thmh4}.
\end{proof}

The example below shows that Problem \ref{prob1} is non-trivial.

\begin{exam}
  Suppose that there exist $x, y, \alpha \in K^{\lor} \setminus K$ such that $x \nsim y$ and $d(y,K) = d(x,K) \cdot d(\alpha,K)$. Then there exist $z, w \in K^{\lor} \setminus K$ such that 
  \begin{align*}
    ([1, x, z], |\cdot|) \ \mathrm{and} \ ([1, x, w], |\cdot|)
  \end{align*}
  are of the same subtype $(x, \alpha, 1)$, and there exists no linear isometry
  \begin{align*}
    T : ([1, x, z], |\cdot|) \to ([1, x, w], |\cdot|)
  \end{align*}
  for which $T(1) = 1$ and $T(x) = x$.
\end{exam}
\begin{proof}
  Let $\pi : (K^{\lor},|\cdot|) \to K^{\lor}/K$ be the canonical quotient map. By assumption, it follows from Theorem \ref{thmh4} that $[\pi(x),\pi(y)] \cong (K^2, |\cdot|_s \times |\cdot|_r)$ where $r = d(x,K)$ and $s = d(y,K)$. Moreover, since $K^{\lor}/K$ is spherically complete, there exists a linear isometry $R : ((K^{\lor})^2, |\cdot|_s \times |\cdot|_r) \to K^{\lor}/K$ such that $R((1,0)) = \pi(y)$ and $R((0,1)) = \pi(x)$. \par 
  Let us choose $z, w \in K^{\lor} \setminus K$ such that 
  \begin{align*}
    \pi(z) = R((1,\alpha)) \ \mathrm{and} \  \pi(w) = R((0,\alpha)).
  \end{align*}
  We shall prove $z$ and $w$ satisfy the requirement. Suppose that there exists a linear isometry
  \begin{align*}
    T : ([1, x, z], |\cdot|) \to ([1, x, w], |\cdot|)
  \end{align*}
  for which $T(1) = 1$ and $T(x) = x$. Put $w' = T(z)$. Then we have 
  \begin{align*}
    |w' - z| \le |z - (\lambda x + \mu)| \ \text{for each} \  \lambda, \mu \in K.
  \end{align*}
    Hence, we have
  \begin{align*}
    |w' - z| \le d(z,[1,x]) = d(\pi(z),[\pi(x)]) = d_{(K^{\lor},|\cdot|_r)}(\alpha,K) = r \cdot d(\alpha, K).
  \end{align*}
  Therefore, since $w' - z \notin K$, we have 
  \begin{align*}
   s \le \|\pi(w') - \pi(z) \| < |w' - z| \le r \cdot d(\alpha, K),
  \end{align*}
  which is a contradiction.
\end{proof}

\section{$4$-dimensional normed spaces} \label{4d}

In this section, we study $4$-dimensional normed spaces. In comparison with $3$-dimensional normed spaces, studying $4$-dimensional normed spaces is much more difficult. Indeed, in studying $3$-dimensional normed spaces, almost all results follow systematically. On the other hand, in studying $4$-dimensional normed spaces, it is even difficult to determine whether a normed space is decomposable or not. According to the types of $3$-dimensional normed spaces, type $\mathrm{\Rnum{1}}_3 \sim \mathrm{\Rnum{5}}_3$, we will classify $4$-dimensional normed spaces into seventeen types, type $\mathrm{\Rnum{1}}_4 \sim \mathrm{\Rnum{17}}_4$. 

\begin{dfn}
  A $4$-dimensional normed space $(E, \|\cdot\|)$ is
  \begin{itemize}
  \setlength{\leftskip}{-20pt}
    \item of type $\mathrm{\Rnum{1}}_4$ if $\dim_{K^{\lor}}E^{\lor} = 4$.
    \item of type $\mathrm{\Rnum{2}}_4$ if $\dim_{K^{\lor}}E^{\lor} = 3$ and $E$ is decomposable, and if $E$ contains a $3$-dimensional subspace of type $\mathrm{\Rnum{2}}_3$.
    \item of type $\mathrm{\Rnum{3}}_4$ if $\dim_{K^{\lor}}E^{\lor} = 3$ and $E$ is decomposable, and if $E$ contains a $3$-dimensional subspace of type $\mathrm{\Rnum{3}}_3$.
    \item of type $\mathrm{\Rnum{4}}_4$ if $\dim_{K^{\lor}}E^{\lor} = 2$ and $E$ is decomposable, and if $E$ contains a $3$-dimensional subspace of type $\mathrm{\Rnum{4}}_3$.
    \item of type $\mathrm{\Rnum{5}}_4$ if $\dim_{K^{\lor}}E^{\lor} = 2$ and $E$ is decomposable, and if $E$ contains a $3$-dimensional subspace of type $\mathrm{\Rnum{5}}_3$.
    \item of type $\mathrm{\Rnum{6}}_4$ if $\dim_{K^{\lor}}E^{\lor} = 2$ and $E$ is decomposable, and if $E$ contains a $3$-dimensional subspace of type $\mathrm{\Rnum{3}}_3$.
    \item of type $\mathrm{\Rnum{7}}_4$ if $\dim_{K^{\lor}}E^{\lor} = 2$ and $E$ is decomposable, and if all $3$-dimensional subspaces of $E$ are of type $\mathrm{\Rnum{2}}_3$.
    \item of type $\mathrm{\Rnum{8}}_4$ if $\dim_{K^{\lor}}E^{\lor} = 1$ and $E / [u]$ is a $3$-dimensional normed space of type $\mathrm{\Rnum{1}}_3$ for each non-zero element $u \in E$.
    \item of type $\mathrm{\Rnum{9}}_4$ if $\dim_{K^{\lor}}E^{\lor} = 1$ and $E / [u]$ is a $3$-dimensional normed space of type $\mathrm{\Rnum{2}}_3$ for each non-zero element $u \in E$.
    \item of type $\mathrm{\Rnum{10}}_4$ if $\dim_{K^{\lor}}E^{\lor} = 1$ and $E / [u]$ is a $3$-dimensional normed space of type $\mathrm{\Rnum{3}}_3$ for each non-zero element $u \in E$.
    \item of type $\mathrm{\Rnum{11}}_4$ if $\dim_{K^{\lor}}E^{\lor} = 1$ and $E / [u]$ is a $3$-dimensional normed space of type $\mathrm{\Rnum{4}}_3$ for each non-zero element $u \in E$.
    \item of type $\mathrm{\Rnum{12}}_4$ if $\dim_{K^{\lor}}E^{\lor} = 1$ and $E / [u]$ is a $3$-dimensional normed space of type $\mathrm{\Rnum{5}}_3$ for each non-zero element $u \in E$.
    \item of type $\mathrm{\Rnum{13}}_4$ if $\dim_{K^{\lor}}E^{\lor} = 3$ and $E$ is indecomposable.
    \item of type $\mathrm{\Rnum{14}}_4$ if $\dim_{K^{\lor}}E^{\lor} = 2$ and $E$ is indecomposable, and if all $3$-dimensional subspaces of $E$ are of type $\mathrm{\Rnum{2}}_3$.
    \item of type $\mathrm{\Rnum{15}}_4$ if $\dim_{K^{\lor}}E^{\lor} = 2$ and $E$ is indecomposable, and if $E$ contains both $3$-dimensional normed spaces of type $\mathrm{\Rnum{2}}_3$ and $\mathrm{\Rnum{3}}_3$.
    \item of type $\mathrm{\Rnum{16}}_4$ if $\dim_{K^{\lor}}E^{\lor} = 2$ and $E$ is indecomposable, and if all $3$-dimensional subspaces of $E$ are of type $\mathrm{\Rnum{3}}_3$ and $\dim_{K^{\lor}} (E / F)^{\lor} = 1$ for each $2$-dimensional subspace of $E$.
    \item of type $\mathrm{\Rnum{17}}_4$ if $\dim_{K^{\lor}}E^{\lor} = 2$ and $E$ is indecomposable, and if all $3$-dimensional subspaces of $E$ are of type $\mathrm{\Rnum{3}}_3$ and $\dim_{K^{\lor}} (E / F)^{\lor} = 2$ for each $2$-dimensional subspace of $E$.
  \end{itemize}
\end{dfn}

In Theorem \ref{thm19}, we prove that these types classify $4$-dimensional normed spaces. Moreover, the type of the dual space is determined by the type of the original normed space (see Theorem \ref{thm17}). Finally, in Proposition \ref{prop2} and Theorem \ref{thm21}, we find that the types are not redundant.

\subsection{Type $\mathrm{\protect \Rnum{1}}_4$}

\begin{thm}
  Let $(E,\|\cdot\|)$ be a $4$-dimensional normed space of type $\mathrm{\Rnum{1}}_4$. \\
  $(1)$ $E$ satisfies $\mathrm{(SE)}$. \\
  $(2)$ There exist $t_1, t_2, t_3, t_4 \in \mathbb{R}_{>0}$ for which 
\begin{align*}
  (E,\|\cdot\|) \cong (K^4, |\cdot|_{t_1} \times |\cdot|_{t_2} \times |\cdot|_{t_3} \times |\cdot|_{t_4}).
\end{align*}
Moreover, it follows that
\begin{align*}
  E' \cong (K^4, |\cdot|_{\frac{1}{t_1}} \times |\cdot|_{\frac{1}{t_2}} \times |\cdot|_{\frac{1}{t_3}} \times |\cdot|_{\frac{1}{t_4}})
\end{align*}
and $E'$ is of type $\mathrm{\Rnum{1}}_4$.
\end{thm}

\subsection{Type $\mathrm{\protect \Rnum{2}}_4$, $\mathrm{\protect \Rnum{3}}_4$} \label{type2,3}

First, we will prepare the lemmas needed later.

\begin{lem} \label{lem1}
  Let $E$ be a decomposable $4$-dimensional normed space with $\dim_{K^{\lor}}E^{\lor} = 3$. Then $E$ is isometrically isomorphic to either
  \begin{align*}
    (K^2, |\cdot|_{t_1} \times |\cdot|_{t_2}) \oplus ([1, x],|\cdot|_{t_3}) \ \text{with $x \in K^{\lor} \setminus K$ and $t_1, t_2, t_3 \in \mathbb{R}_{> 0}$}
  \end{align*}
  or
  \begin{align*}
    (K,|\cdot|_{t_1}) \oplus ([(1,0),(0,1),(x,y)], |\cdot|_{t_2} \times |\cdot|_{t_3}) 
  \end{align*}
  where $x, y \in K^{\lor} \setminus K$, $x \nsim y$, and $t_2 d(x,K) = t_3 d(y,K)$.
\end{lem}
\begin{proof}
  Let $E_1, E_2$ be two subspaces of $E$ such that $E \cong E_1 \oplus E_2$. Without loss of generality, we may assume $\dim_{K^{\lor}}E_1^{\lor} = 1$ and $\dim_{K^{\lor}}E_2^{\lor} = 2$. If $\dim_K E_1 = 2$, then $\dim_K E_2 = 2$. Hence, $E_2$ has an orthogonal base. Therefore, 
  \begin{align*}
    E \cong (K^2, |\cdot|_{t_1} \times |\cdot|_{t_2}) \oplus ([1, x],|\cdot|_{t_3}) \ \text{with $x \in K^{\lor} \setminus K$ and $t_1, t_2, t_3 \in \mathbb{R}_{> 0}$}.
  \end{align*}
  If $\dim_K E_1 = 1$, then $E_2$ is either of type $\mathrm{\Rnum{2}}_3$ or $\mathrm{\Rnum{3}}_3$. If $E_2$ is of type $\mathrm{\Rnum{2}}_3$, then $E$ is also isometrically isomorphic to a normed space of the form as above. If $E_2$ is of type $\mathrm{\Rnum{3}}_3$, then
    \begin{align*}
    E \cong (K,|\cdot|_{t_1}) \oplus ([(1,0),(0,1),(x,y)], |\cdot|_{t_2} \times |\cdot|_{t_3}) 
  \end{align*}
  where $x, y \in K^{\lor} \setminus K$, $x \nsim y$, and $t_2 d(x,K) = t_3 d(y,K)$. Thus, the proof is complete.
\end{proof}

\begin{lem} \label{lem2}
  Let $x \in K^{\lor} \setminus K$ and $t_1, t_2, t_3 \in \mathbb{R}_{> 0}$. Then each $3$-dimensional subspace of $(K^2, |\cdot|_{t_1} \times |\cdot|_{t_2}) \oplus ([1, x],|\cdot|_{t_3})$ is isometrically isomorphic to one of the following normed spaces:
  \begin{itemize}
    \item $(K^3,|\cdot|_{t_1} \times |\cdot|_{t_2} \times |\cdot|_{t_3})$.
    \item $(K,|\cdot|_{t_1}) \oplus  ([1, x],|\cdot|_{t_3})$.
    \item $(K,|\cdot|_{t_2}) \oplus  ([1, x],|\cdot|_{t_3})$.
  \end{itemize}
  In particular, it contains no subspace of type $\mathrm{\Rnum{3}}_3$.
\end{lem}
\begin{proof}
  We can use Theorem \ref{systhm3}.
\end{proof}

\begin{lem} \label{lem3}
  Let $x, y \in K^{\lor} \setminus K$, $x \nsim y$, and let $t_1, t_2 \in \mathbb{R}_{> 0}$. Suppose $t_2 d(x,K) = t_3 d(y,K)$. Then each $3$-dimensional subspace of $(K,|\cdot|_{t_1}) \oplus ([(1,0),(0,1),(x,y)], |\cdot|_{t_2} \times |\cdot|_{t_3})$ is isometrically isomorphic to either 
  \begin{align*}
    (K^3,|\cdot|_{t_1} \times |\cdot|_{t_2} \times |\cdot|_{t_3}) \ \text{or} \ ([(1,0),(0,1),(x,y)], |\cdot|_{t_2} \times |\cdot|_{t_3}).
  \end{align*}
  In particular, it contains no subspace of type $\mathrm{\Rnum{2}}_3$.
\end{lem}
\begin{proof}
  We can use Theorems \ref{systhm3} and \ref{.4thm2}.
\end{proof}

Now, we obtain theorems that determine the type of decomposable $4$-dimensional normed spaces $E$ with $\dim_{K^{\lor}} E^{\lor} = 3$.

\begin{thm}
  Let $E$ be a $4$-dimensional normed space. \\
  $(1)$ $E$ is of type $\mathrm{\Rnum{2}}_4$ if and only if there exist $x \in K^{\lor} \setminus K$ and $t_1, t_2, t_3 \in \mathbb{R}_{> 0}$ for which
  \begin{align*}
    E \cong (K^2, |\cdot|_{t_1} \times |\cdot|_{t_2}) \oplus ([1, x],|\cdot|_{t_3}).
  \end{align*}
  \noindent $(2)$ Suppose that $E$ is a $4$-dimensional normed space of type $\mathrm{\Rnum{2}}_4$ of the form as in $(1)$. Then $E$ satisfies $\mathrm{(SE)}$ if and only if $t_3 \notin V_K$.
\end{thm}
\begin{proof}
  By Lemmas \ref{lem1} and \ref{lem3}, we obtain $(1)$. Moreover, by Theorems \ref{systhm4} and \ref{.4thm4}, $(2)$ follows. 
\end{proof}

\begin{thm}
  Let $E$ be a $4$-dimensional normed space. \\
  $(1)$ $E$ is of type $\mathrm{\Rnum{3}}_4$ if and only if there exist $x, y \in K^{\lor} \setminus K$, $x \nsim y$, and $t_1, t_2 \in \mathbb{R}_{> 0}$ with $t_2 d(x,K) = t_3 d(y,K)$ for which
  \begin{align*}
    E \cong (K,|\cdot|_{t_1}) \oplus ([(1,0),(0,1),(x,y)], |\cdot|_{t_2} \times |\cdot|_{t_3}).
  \end{align*}
  \noindent $(2)$ Suppose that $E$ is a $4$-dimensional normed space of type $\mathrm{\Rnum{3}}_4$ of the form as in $(1)$. Then $E$ satisfies $\mathrm{(SE)}$ if and only if $t_2, t_3 \notin V_K$.
\end{thm}
\begin{proof}
  By Lemmas \ref{lem1} and \ref{lem2}, we obtain $(1)$. Moreover, by Theorems \ref{systhm4} and \ref{.4thm3}, we have $(2)$. 
\end{proof}

\begin{thm} \label{thm1} 
 A decomposable $4$-dimensional normed space $E$ with $\dim_{K^{\lor}}E^{\lor} = 3$ is either of type $\mathrm{\Rnum{2}}_4$ or type $\mathrm{\Rnum{3}}_4$.
\end{thm}
\begin{proof}
  We can use Lemmas \ref{lem1}, \ref{lem2} and \ref{lem3}.
\end{proof}

\subsection{Type $\mathrm{\protect \Rnum{4}}_4 \sim \mathrm{\protect \Rnum{7}}_4$}

As in section \ref{type2,3}, we first identify all subspaces of decomposable $4$-dimensional normed spaces $E$ with $\dim_{K^{\lor}}E^{\lor} = 2$.

\begin{lem} \label{lem4}
  Let $E$ be a decomposable $4$-dimensional normed space with $\dim_{K^{\lor}}E^{\lor} = 2$. Then $E$ is isometrically isomorphic to either
  \begin{align*}
  (K, |\cdot|_{t_1}) \oplus ([1, x, y], |\cdot|_{t_2}) \ \text{where $x, y \in K^{\lor} \setminus K$ and $t_1, t_2 \in \mathbb{R}_{> 0}$}
  \end{align*}
  or
  \begin{align*}
    ([1, x], |\cdot|_{t_1}) \oplus ([1, y], |\cdot|_{t_2}) \ \text{where $x, y \in K^{\lor} \setminus K$ and $t_1, t_2 \in \mathbb{R}_{> 0}$}.
  \end{align*}
\end{lem}
\begin{proof}
  Let $E_1, E_2$ be two subspaces of $E$ such that $E \cong E_1 \oplus E_2$. Obviously, we have $\dim_{K^{\lor}}E_1^{\lor} = \dim_{K^{\lor}}E_2^{\lor} = 1$. If $\dim_K E_1 = 1$, then $\dim_K E_2 = 3$ and 
  \begin{align*}
    E \cong (K, |\cdot|_{t_1}) \oplus ([1, x, y], |\cdot|_{t_2}) \ \text{where $x, y \in K^{\lor} \setminus K$ and $t_1, t_2 \in \mathbb{R}_{> 0}$}.
  \end{align*}
  If $\dim_K E_1 = 2$, then $\dim_K E_2 = 2$ and
  \begin{align*}
    E \cong ([1, x], |\cdot|_{t_1}) \oplus ([1, y], |\cdot|_{t_2}) \ \text{where $x, y \in K^{\lor} \setminus K$ and $t_1, t_2 \in \mathbb{R}_{> 0}$}.
  \end{align*}
\end{proof}

\begin{lem} \label{lem5}
   Let $x, y \in K^{\lor} \setminus K$ and $t_1, t_2 \in \mathbb{R}_{> 0}$. Set
  \begin{align*}
    E := (K, |\cdot|_{t_1}) \oplus ([1, x, y], |\cdot|_{t_2}).
  \end{align*}
  \noindent $(1)$ If $([1, x, y], |\cdot|_{t_2})$ is of type $\mathrm{\Rnum{4}}_3$, then each $3$-dimensional subspace $F$ of $E$ is isometrically isomorphic to either 
  \begin{align*}
    ([1, x, y], |\cdot|_{t_2}) \ \text{or} \ (K, |\cdot|_{t_1}) \oplus ([1, z], |\cdot|_{t_2}) \ \text{where $z \in [1, x, y] \setminus K$}.
  \end{align*}
  \noindent $(2)$ If $([1, x, y], |\cdot|_{t_2})$ is of type $\mathrm{\Rnum{5}}_3$, then each $3$-dimensional subspace $F$ of $E$ is isometrically isomorphic to either 
  \begin{align*}
    ([1, x, y], |\cdot|_{t_2}) \ \text{or} \ (K, |\cdot|_{t_1}) \oplus ([1, x], |\cdot|_{t_2}).
  \end{align*}
\end{lem}
\begin{proof}
  For $(1)$, use Theorems \ref{systhm3} and \ref{.4thm3}. Moreover, by Theorems \ref{systhm3} and \ref{.4thm1}, we obtain $(2)$. 
\end{proof}

\begin{lem} \label{lem6}
  Let $x, y \in K^{\lor} \setminus K$ and $t_1, t_2 \in \mathbb{R}_{> 0}$. Set
  \begin{align*}
    E := ([1, x], |\cdot|_{t_1}) \oplus ([1, y], |\cdot|_{t_2}).
  \end{align*}
  \noindent $(1)$ If $x \nsim y$ and $\tfrac{t_1 d(x, K)}{t_2 d(y, K)} \in V_K$, then each $3$-dimensional subspace $F$ of $E$ is isometrically isomorphic to one of the following:
  \begin{itemize}
    \item $(K, |\cdot|_{t_1}) \oplus ([1, y], |\cdot|_{t_2})$.
    \item $([1, x], |\cdot|_{t_1}) \oplus (K, |\cdot|_{t_2})$. 
    \item $([(1,0), (0,1), (x, \lambda y)], |\cdot|_{t_1} \times |\cdot|_{t_2})$ where $\lambda \in K$ with $|\lambda| = \tfrac{t_1 d(x, K)}{t_2 d(y, K)}$, a $3$-dimensional normed of type $\mathrm{\Rnum{3}}_3$.
  \end{itemize}
  $(2)$ If $x \sim y$ or $\tfrac{t_1 d(x, K)}{t_2 d(y, K)} \notin V_K$, then each $3$-dimensional subspace $F$ of $E$ is isometrically isomorphic to either
  \begin{align*}
    (K, |\cdot|_{t_1}) \oplus ([1, y], |\cdot|_{t_2}) \ \text{or} \ ([1, x], |\cdot|_{t_1}) \oplus (K, |\cdot|_{t_2}).
  \end{align*}
\end{lem}
\begin{proof}
  Let $u, v, w \in E$ be such that $F = [u, v, w]$. It suffices to consider the following four cases:
  \begin{itemize}
   \item[$(a)$] $u = (0, y), v = (1, 0), w = (0, 1)$.
   \item[$(b)$] $u = (x, \lambda y), v = (1, 0), w = (0, 1)$ ($\lambda \in K$).
   \item[$(c)$] $u = (x + \lambda, 0), v = (\mu , y), w = (0, 1)$ ($\lambda, \mu \in K$).
   \item[$(d)$] $u = (x, \lambda), v = (0, y + \mu), w = (1, \eta)$ ($\lambda, \mu, \eta \in K$).
  \end{itemize}
  In the case $(a)$, $F$ is always isometrically isomorphic to $(K, |\cdot|_{t_1}) \oplus ([1, y], |\cdot|_{t_2})$. Next, consider the case $(b)$. If $x \sim y$ or  $|\lambda| \neq \tfrac{t_1 d(x, K)}{t_2 d(y, K)}$, then by Remark \ref{sysrem1}, $F$ is isometrically isomorphic to either
  \begin{align*}
    (K, |\cdot|_{t_1}) \oplus ([1, y], |\cdot|_{t_2}) \ \text{or} \ ([1, x], |\cdot|_{t_1}) \oplus (K, |\cdot|_{t_2}).
  \end{align*}
  If $x \nsim y$ and  $|\lambda| = \tfrac{t_1 d(x, K)}{t_2 d(y, K)}$, then $F$ is isometrically isomorphic to a $3$-dimensional normed $([(1,0), (0,1), (x, \lambda y)], |\cdot|_{t_1} \times |\cdot|_{t_2})$ of type $\mathrm{\Rnum{3}}_3$. \par
  Finally, we shall show that both the cases $(c)$ and $(d)$ can be reduced to the case $(b)$. In the case $(c)$, we have
  \begin{align*}
    F \cong ([(1, 0), (0, 1), (\frac{\mu}{x + \lambda}, y)], |\cdot|_{t_1 \cdot |x + \lambda|} \times |\cdot|_{t_2}).
  \end{align*}
  If $\mu = 0$, then we are done. Suppose $\mu \neq 0$. Then we can choose $\lambda_0 \in K$ such that $d(\tfrac{\mu}{x + \lambda} - \lambda_0 x, K) < d(\tfrac{\mu}{x + \lambda}, K)$. Now, we can conclude
  \begin{align*}
    F &\cong ([(1, 0), (0, 1), (\frac{\mu}{x + \lambda}, y)], |\cdot|_{t_1 \cdot |x + \lambda|} \times |\cdot|_{t_2}) \\
    &\cong  ([(1, 0), (0, 1), (\lambda_0 x, y)], |\cdot|_{t_1 \cdot |x + \lambda|} \times |\cdot|_{t_2}) \\
    &\cong ([(1, 0), (0, 1), (x, \lambda_0^{-1} y)], |\cdot|_{t_1} \times |\cdot|_{\frac{t_2}{|x + \lambda|}}) \\
    &\cong ([(1, 0), (0, 1), (x, \lambda_1 \lambda_0^{-1} y)], |\cdot|_{t_1} \times |\cdot|_{t_2})
  \end{align*}
  where $\lambda_1 \in K$ with $|\lambda_1| = 1 / |x + \lambda|$. Thus, the case $(c)$ can be reduced to the case $(a)$. Now, consider the case $(d)$. Then it is easy to see
  \begin{align*}
    F \cong \left\{ 
    \begin{aligned}
      ([(1,0), (0,1), (x, \frac{\lambda}{y + \mu})], |\cdot|_{t_1} \times |\cdot|_{t_2 \cdot |y + \mu|}) \ &; \eta = 0 \\
      ([(1,0), (0,1), (\frac{1}{x - \lambda / \eta}, \frac{\eta}{y + \mu})], |\cdot|_{t_1 \cdot |x - \lambda / \eta|} \times |\cdot|_{t_2 \cdot |y + \mu|}) \  &; \eta \neq 0.
    \end{aligned} \right.
  \end{align*}
  Therefore, as in the case $(c)$, the case $(d)$ can be reduced to the case $(b)$. Thus, the proof is complete.
\end{proof}

Finally, we can determine the type of a decomposable $4$-dimensional normed space $E$ with $\dim_{K^{\lor}} E^{\lor} = 2$.

\begin{thm}
  Let $E$ be a $4$-dimensional normed space. \\
  $(1)$ $E$ is of type $\mathrm{\Rnum{4}}_4$ if and only if there exist $x, y \in K^{\lor} \setminus K$ and $t_1, t_2 \in \mathbb{R}_{> 0}$ for which
  \begin{align*}
    E \cong (K, |\cdot|_{t_1}) \oplus ([1, x, y],|\cdot|_{t_2})
  \end{align*}
  and $([1, x, y],|\cdot|_{t_2})$ is of type $\mathrm{\Rnum{4}}_3$. \\
  $(2)$ Suppose that $E$ is a normed space of type $\mathrm{\Rnum{4}}_4$ of the form as in $(1)$. Then $E$ satisfies $\mathrm{(SE)}$ if and only if $t_2 \notin V_K$.
\end{thm}
\begin{proof}
  By Lemmas \ref{lem4}, \ref{lem5}, and \ref{lem6}, we have $(1)$. Moreover, by Theorems \ref{systhm4} and \ref{.4thm3}, we obtain $(2)$.
\end{proof}

\begin{thm}
  Let $E$ be a $4$-dimensional normed space. \\
  $(1)$ $E$ is of type $\mathrm{\Rnum{5}}_4$ if and only if there exist $x, y \in K^{\lor} \setminus K$ and $t_1, t_2 \in \mathbb{R}_{> 0}$ for which
  \begin{align*}
    E \cong (K, |\cdot|_{t_1}) \oplus ([1, x, y],|\cdot|_{t_2})
  \end{align*}
  and $([1, x, y],|\cdot|_{t_2})$ is of type $\mathrm{\Rnum{5}}_3$. \\
  $(2)$ Suppose that $E$ is a normed space of type $\mathrm{\Rnum{5}}_4$ of the form as in $(1)$. Then $E$ satisfies $\mathrm{(SE)}$ if and only if $t_2, t_2 d(x, K) \notin V_K$.
\end{thm}
\begin{proof}
  For $(1)$, use Lemmas \ref{lem4}, \ref{lem5}, and \ref{lem6}. By Theorems \ref{systhm4} and \ref{.4thm5}. we obtain $(2)$.
\end{proof}

\begin{thm} \label{thm2} 
  Let $E$ be a $4$-dimensional normed space. \\
  $(1)$ $E$ is of type $\mathrm{\Rnum{6}}_4$ if and only if there exist $x, y \in K^{\lor} \setminus K$ and $t_1, t_2 \in \mathbb{R}_{> 0}$ with $x \nsim y$ and $\tfrac{t_1 d(x, K)}{t_2 d(y, K)} \in V_K$ for which
  \begin{align*}
    E \cong ([1, x], |\cdot|_{t_1}) \oplus ([1, y],|\cdot|_{t_2})
  \end{align*}
  \noindent $(2)$ Suppose that $E$ is a normed space of type $\mathrm{\Rnum{6}}_4$ of the form as in $(1)$. Then $E$ satisfies $\mathrm{(SE)}$ if and only if $t_1, t_2 \notin V_K$.
\end{thm}
\begin{proof}
  By Lemmas \ref{lem4}, \ref{lem5}, and \ref{lem6}, we have $(1)$. We shall prove $(2)$. By Lemmas \ref{lem6} and Theorems \ref{.4thm4} and \ref{.4thm2}, each $3$-dimensional subspace satisfies (SE) if and only if $t_1, t_2 \notin V_K$. Moreover, since
  \begin{align*}
    E' \cong ([1, x], |\cdot|_{1 / t_1 d(x, K)}) \oplus ([1, y],|\cdot|_{1 / t_2 d(y, K)}),
  \end{align*}
  by Lemma \ref{lem6} and Theorems \ref{.4thm6} and \ref{.4thm2}, each $2$-dimensional subspace of $E'$ is isometrically isomorphic to one of the following:
  \begin{itemize}
    \item $([1, x], |\cdot|_{1 / t_1 d(x, K)})$.
    \item $([1, y],|\cdot|_{1 / t_2 d(y, K)})$.
    \item $(K^2, |\cdot|_{1 / t_1 d(x, K)} \times |\cdot|_{1 / t_2 d(y, K)})$.
  \end{itemize}
  Therefore, $E$ satisfies the condition $(2)$ in Lemma \ref{selem1} if and only if $t_1, t_2 \notin V_K$. Finally, by Lemma \ref{selem1}, $E$ satisfies (SE) if and only if $t_1, t_2 \notin V_K$. Thus, the proof is complete.
\end{proof}

\begin{thm}
  Let $E$ be a $4$-dimensional normed space. \\
  $(1)$ $E$ is of type $\mathrm{\Rnum{7}}_4$ if and only if there exist $x, y \in K^{\lor} \setminus K$ and $t_1, t_2 \in \mathbb{R}_{> 0}$ with $x \sim y$ or $\tfrac{t_1 d(x, K)}{t_2 d(y, K)} \notin V_K$ for which
  \begin{align*}
    E \cong ([1, x], |\cdot|_{t_1}) \oplus ([1, y],|\cdot|_{t_2})
  \end{align*}
  \noindent $(2)$ Suppose that $E$ is a normed space of type $\mathrm{\Rnum{7}}_4$ of the form as in $(1)$. Then $E$ satisfies $\mathrm{(SE)}$ if and only if $t_1, t_2 \notin V_K$.
\end{thm}
\begin{proof}
  We can apply a proof similar to that of Theorem \ref{thm2}. 
\end{proof}

\begin{thm}
  A decomposable $4$-dimensional normed space $E$ with $\dim_{K^{\lor}}E^{\lor} = 2$ is exactly of one type of type $\mathrm{\Rnum{4}}_4 \sim \mathrm{\Rnum{7}}_4$.
\end{thm}
\begin{proof}
  We can use Lemmas \ref{lem4}, \ref{lem5} and \ref{lem6}. 
\end{proof}

By Theorem \ref{thm1} and the theorem above, we have the following.

\begin{thm} \label{thm15}
  Every decomposable $4$-dimensional normed space $(E, \|\cdot\|)$ is exactly of one type of type $\mathrm{\Rnum{1}}_4 \sim \mathrm{\Rnum{7}}_4$.
\end{thm}

Now, by Corollary \ref{.4cor1} and Proposition \ref{.4prop3}, we can identify the dual spaces of decomposable $4$-dimensional normed spaces.

\begin{thm}
  If $E$ is a $4$-dimensional normed space of type $\mathrm{\Rnum{1}}_4$ (resp. $\mathrm{\Rnum{2}}_4$, $\mathrm{\Rnum{3}}_4$, $\mathrm{\Rnum{4}}_4$, $\mathrm{\Rnum{5}}_4$, $\mathrm{\Rnum{6}}_4$, $\mathrm{\Rnum{7}}_4$), then $E'$ is of type $\mathrm{\Rnum{1}}_4$ (resp. $\mathrm{\Rnum{2}}_4$, $\mathrm{\Rnum{4}}_4$, $\mathrm{\Rnum{3}}_4$, $\mathrm{\Rnum{5}}_4$, $\mathrm{\Rnum{6}}_4$, $\mathrm{\Rnum{7}}_4$).
\end{thm}

\subsection{Type $\mathrm{\protect \Rnum{8}}_4, \mathrm{\protect \Rnum{13}}_4$}

As in section \ref{type3,4}, we obtain the following results. We omit the proofs of the results.

\begin{thm}   \label{thm12}
  Let $x, y, z \in K^{\lor} \setminus K$ and $t_1, t_2, t_3 \in \mathbb{R}_{>0}$. Let $\pi : K^{\lor} \to K^{\lor} / K$ be the canonical quotient map. We set
  \begin{align*}
    (E,\|\cdot\|) := [(1,0,0),(0,1,0),(0,0,1),(x,y,z)],
  \end{align*}
  a subspace of $((K^{\lor})^3,|\cdot|_{t_1} \times |\cdot|_{t_2} \times |\cdot|_{t_3})$. \\
  $(1)$ $E$ is of type $\mathrm{\Rnum{13}}_4$ if and only if $ t_1 d(x,K) = t_2 d(y,K) = t_3 d(z,K)$ and $\{\pi(x), \pi(y), \pi(z)\}$ is an orthogonal set. \\
  $(2)$ Suppose that $E$ is of type $\mathrm{\Rnum{13}}_4$. Then $E$ satisfies $(\mathrm{SE})$ if and only if $t_1,t_2,t_3 \notin V_K$. \\
  $(3)$ Suppose that $E$ is of type $\mathrm{\Rnum{13}}_4$. Then each $3$-dimensional subspace of $E$ is isometrically isomorphic to
\begin{align*}
  (K^3,|\cdot|_{t_1} \times |\cdot|_{t_2} \times |\cdot|_{t_3}).
\end{align*} 
Moreover, let $e_1, e_2, e_3, e_4 \in E'$ be the dual basis of $(x,y,z),(1,0,0),(0,1,0),(0,0,1)$. Then the map
  \begin{align*}
    E' &\to ([1, x, y], |\cdot|_{\frac{1}{t}}) \\
    e_1 &\mapsto 1 \\
    e_2 &\mapsto -x \\
    e_3 &\mapsto -y \\
    e_4 &\mapsto -z
  \end{align*}
  gives a linear isometry.
\end{thm}

\begin{thm} 
  Let $x,y,z \in K^{\lor} \setminus K$ and $t \in \mathbb{R}_{>0}$.  Let $\pi : K^{\lor} \to K^{\lor} / K$ be the canonical quotient map. We set $E := ([1, x, y, z], |\cdot|_{t})$. \\
  $(1)$ $E$ is of type $\mathrm{\Rnum{8}}_4$ if and only if $\{\pi(x'), \pi(y'), \pi(z')\}$ is an orthogonal set for some $x', y', z' \in E$. \\
  $(2)$ Suppose that $E$ is of type $\mathrm{\Rnum{8}}_4$. Then $E$ satisfies $\mathrm{(SE)}$ if and only if $t \notin V_K$. \\
  $(3)$ Suppose that $E$ is of type $\mathrm{\Rnum{8}}_4$. Then for each $3$-dimensional subspace F of $E$, there exists $x', y' \in [1, x, y, z]$ such that
  \begin{align*}
    F \cong ([1,x',y'], |\cdot|_t)
  \end{align*}
  Moreover, we have
  \begin{align*}
    E' \cong ([(1,0, 0),(0,1,0),(0,0,1)(x,y,z)], |\cdot|_{\frac{1}{tt_1}} \times |\cdot|_{\frac{1}{tt_2}} \times |\cdot|_{\frac{1}{tt_3}})
  \end{align*}
  where $t_1 = d(x,K), t_2 = d(y,K), t_3 = d(z,K)$.
\end{thm}

\begin{cor} \label{cor3}
  If $E$ is a $3$-dimensional normed space of type $\mathrm{\Rnum{8}}_4$ (resp. $\mathrm{\Rnum{13}}_4$), then $E'$ is of type $\mathrm{\Rnum{13}}_4$ (resp. $\mathrm{\Rnum{8}}_4$). 
\end{cor}

\begin{thm}
  Let $x_1, y_1, z_1, x_2, y_2, z_2 \in K^{\lor} \setminus K$ and $t, s \in \mathbb{R}_{>0}$. Let $\pi : K^{\lor} \to K^{\lor} / K$ be the canonical quotient map. Suppose that both $\{\pi(x_1), \pi(y_1), \pi(z_1)\}$ and $\{\pi(x_2), \pi(y_2), \pi(z_2)\}$ are orthogonal sets. Then
  \begin{align*}
    ([1, x_1, y_1, z_1], |\cdot|_t) \cong ([1, x_2, y_2, z_2], |\cdot|_s)
  \end{align*}
  if and only if $t / s \in V_K$ and there exist $x, y, z \in [1, x_1, y_1, z_1]$ satisfying $x \sim x_2$, $y \sim y_2$ and $z \sim z_2$.
\end{thm}

\subsection{Type $\mathrm{\protect \Rnum{12}}_4$}

Since a $3$-dimensional normed space is of type $\mathrm{\Rnum{5}}_3$ if and only if it is hyper-symmetric, we have the following theorem.

\begin{thm}
  A $4$-dimensional normed space is of type $\mathrm{\Rnum{12}}_4$ if and only if it is hyper-symmetric.
\end{thm}

Thus, as in section \ref{type5}, we can apply the results of section \ref{hyper} to a normed space of type $\mathrm{\Rnum{12}}_4$. We omit the proofs of the following results.

\begin{prop} \label{prop1}
  If $E$ is a $4$-dimensional normed space of type $\mathrm{\Rnum{12}}_4$, then $E'$ is of type $\mathrm{\Rnum{12}}_4$. 
\end{prop}

\begin{thm} 
   Let $E$ be a $4$-dimensional normed space of type $\mathrm{\Rnum{12}}_4$. Then all $3$-dimensional subspaces of $E$ are isometrically isomorphic to each other. Moreover, all $3$-dimensional quotient spaces of $E$ are isometrically isomorphic to each other.
\end{thm}

\begin{thm}
  Let $x, y, z \in K^{\lor} \setminus K$ and $t \in \mathbb{R}_{> 0}$. Let $\pi : K^{\lor} \to K^{\lor} / K$ be the canonical quotient map. We set $E := ([1, x, y, z], |\cdot|_{t})$. \\
  $(1)$ $E$ is of type $\mathrm{\Rnum{12}}_4$ if and only if $[\pi(x), \pi(y), \pi(z)]$ is of type $\mathrm{\Rnum{5}}_3$. \\
  $(2)$ Suppose that $E$ is of type $\mathrm{\Rnum{12}}_4$. Then $E$ satisfies $\mathrm{(SE)}$ if and only if
  \begin{align*}
    t, t d(x,K), t d(y,[1,x]) \notin V_K.
  \end{align*}
\end{thm}

\subsection{Type $\mathrm{\protect \Rnum{9}}_4, \mathrm{\protect \Rnum{10}}_4, \mathrm{\protect \Rnum{11}}_4$}

By Corollary \ref{syscor1}, we have the following.
\begin{thm} \label{thm10} 
  Let $E$ be a $4$-dimensional normed space with $\dim_{K^{\lor}} E^{\lor} = 1$. Then $E$ is exactly of one type of type $\mathrm{\Rnum{8}}_4 \sim \mathrm{\Rnum{12}}_4$. Moreover, if $E$ is of type $\mathrm{\Rnum{8}}_4$ (resp. $\mathrm{\Rnum{9}}_4, \mathrm{\Rnum{10}}_4, \mathrm{\Rnum{11}}_4, \mathrm{\Rnum{12}}_4$), then each $3$-dimensional subspace of $E'$ is of type $\mathrm{\Rnum{1}}_3$ (resp. $\mathrm{\Rnum{2}}_3, \mathrm{\Rnum{4}}_3, \mathrm{\Rnum{3}}_3, \mathrm{\Rnum{5}}_3$).
\end{thm}

Since the dual space of an indecomposable normed space is also indecomposable, we have the following.

\begin{thm}
  If $E$ is a $4$-dimensional normed space of type $\mathrm{\Rnum{9}}_4$ (resp. $\mathrm{\Rnum{10}}_4, \mathrm{\Rnum{11}}_4$), then $E'$ is of type $\mathrm{\Rnum{14}}_4$ (resp. $\mathrm{\Rnum{10}}_4, \mathrm{\Rnum{16}}_4$).
\end{thm}
\begin{proof}
  If $E$ is of type $\mathrm{\Rnum{9}}_4$ (resp. $\mathrm{\Rnum{10}}_4, \mathrm{\Rnum{11}}_4$), then by the preceding theorem, $E'$ is an indecomposable normed space whose $3$-dimensional subspaces are all of type $\mathrm{\Rnum{2}}_3$ (resp. $\mathrm{\Rnum{4}}_3, \mathrm{\Rnum{3}}_3$). \par
  First, suppose that $E$ is of type $\mathrm{\Rnum{9}}_4$. Then we have $\dim_{K^{\lor}} (E')^{\lor} = 2$ and $E'$ is of type $\mathrm{\Rnum{14}}_4$. \par
  Secondly, suppose that $E$ is of type $\mathrm{\Rnum{11}}_4$. Then we have $\dim_{K^{\lor}} (E')^{\lor} = 2$. Let $F$ be a $2$-dimensional subspace of $E'$. Since $\dim_{K^{\lor}} (E'')^{\lor} = \dim_{K^{\lor}} E^{\lor} =  1$, we obtain $\dim_{K^{\lor}} (E'/F)^{\lor} = 1$. Thus, $E'$ is of type $\mathrm{\Rnum{16}}_4$. \par
  Finally, suppose that $E$ is of type $\mathrm{\Rnum{10}}_4$. Then $\dim_{K^{\lor}} (E')^{\lor} = 1$. Then by the preceding theorem, $E'$ must be either of type $\mathrm{\Rnum{10}}_4$ or $\mathrm{\Rnum{12}}_4$. Moreover, by Proposition \ref{prop1}, $E'$ must be of type $\mathrm{\Rnum{10}}_4$, which completes the proof.
\end{proof}

\begin{thm}
  Let $x, y, z \in K^{\lor} \setminus K$ and $t \in \mathbb{R}_{> 0}$. Set $(E, \|\cdot\|) := ([1, x, y, z], |\cdot|_t)$ and let $\pi : E \to E / K$ be the canonical quotient map. Then $E$ is of type $\mathrm{\Rnum{9}}_4$ (resp. $\mathrm{\Rnum{10}}_4, \mathrm{\Rnum{11}}_4$) if and only if $[\pi(x), \pi(y), \pi(z)]$ is of type $\mathrm{\Rnum{2}}_3$ (resp. $\mathrm{\Rnum{3}}_3, \mathrm{\Rnum{4}}_3$). Moreover, we have the following: \\
  $(1)$ Suppose that $E$ is of type $\mathrm{\Rnum{9}}_4$. Then there exists $x', y' \in [1, x, y, z] \setminus K$ such that $[1,x',y']$ is of type $\mathrm{\Rnum{5}}_3$. Furthermore, $E$ satisfies $\mathrm{(SE)}$ if and only if $t, t d(x', K) \notin V_K$. \\
  $(2)$ Suppose that $E$ is of type $\mathrm{\Rnum{10}}_4$. Then there exists $x', y' \in [1, x, y] \setminus K$ for which $x' \nsim y'$. Furthermore, $E$ satisfies $\mathrm{(SE)}$ if and only if $t, td(x',K), td(y',K) \notin V_K$. \\
  $(3)$ Suppose that $E$ is of type $\mathrm{\Rnum{11}}_4$. Then $E$ satisfies $\mathrm{(SE)}$ if and only if $t, td(x,K) \notin V_K$.
\end{thm}
\begin{proof}
  By Corollary \ref{syscor1}, $E$ is of type $\mathrm{\Rnum{9}}_4$ (resp. $\mathrm{\Rnum{10}}_4, \mathrm{\Rnum{11}}_4$) if and only if $[\pi(x), \pi(y), \pi(z)]$ is of type $\mathrm{\Rnum{2}}_3$ (resp. $\mathrm{\Rnum{3}}_3, \mathrm{\Rnum{4}}_3$). Now, we shall prove $(1), (2)$ and $(3)$. By Corollary \ref{syscor3}, $E$ satisfies (SE) if and only if $t \notin V_K$ and a $3$-dimensional quotient space $[\pi(x),\pi(y), \pi(z)]$ of $E$ satisfies (SE). Hence by Theorem \ref{.4thm3}, $(3)$ immediately follows. It remains to prove $(1)$ and $(2)$. \par
  First, suppose that $E$ is of type $\mathrm{\Rnum{10}}_4$. Then by Theorem \ref{.4thm2}, $[\pi(x), \pi(y)]$ has an orthogonal base. Therefore, there exist $x', y' \in [1, x, y] \setminus K$ for which $x' \nsim y'$. Since $\|[\pi(x),\pi(y), \pi(z)]\| = t d(x', K) \cup t d(y', K)$, by Theorem \ref{.4thm2}, $[\pi(x),\pi(y), \pi(z)]$ satisfies (SE) if and only if $t, td(x',K), td(y',K) \notin V_K$. Thus, $(2)$ is proved. \\
  Finally, suppose that $E$ is of type $\mathrm{\Rnum{9}}_4$. Since $E/K$ contains a $2$-dimensional subspace without an orthogonal base, there exists $x', y' \in [1, x, y, z] \setminus K$ such that $[1,x',y']$ is of type $\mathrm{\Rnum{5}}_3$. Now, we fix $x', y' \in [1, x, y, z] \setminus K$ satisfying that $[1,x',y']$ is of type $\mathrm{\Rnum{5}}_3$. Then there exists an isometry $T$ from $[\pi(x), \pi(y), \pi(z)]$ to a $3$-dimensional normed space
  \begin{align*}
    (K,|\cdot|_{s}) \oplus ([1,w],|\cdot|_{td(x',K)}) \ \text{where $w \in K^{\lor} \setminus K$ and $s \in \mathbb{R}_{> 0}$}.
  \end{align*}
 for which $T(\pi(x')) = (0,1)$ and $T(\pi(y')) = (0,w)$. By Theorem \ref{.4thm4}, $[\pi(x), \pi(y), \pi(z)]$ satisfies (SE) if and only if $td(x',K) \notin V_K$. Thus, we obtain $(1)$.
\end{proof}

\subsection{Type $\mathrm{\protect \Rnum{14}}_4, \mathrm{\protect \Rnum{15}}_4$} 

Let $E$ be a $4$-dimensional normed space with $\dim_{K^{\lor}} E^{\lor} = 2$ that contains a $3$-dimensional normed space, isometrically isomorphic to 
\begin{align*}
  ([1,x],|\cdot|_t) \oplus (K,|\cdot|_s) \ \text{where $x \in K^{\lor} \setminus K$ and $t, s \in \mathbb{R}_{> 0}$}.
\end{align*}
Then there exist $y, z \in K^{\lor} \setminus K$ for which
\begin{align*}
  E \cong ([(1,0), (0,1), (x,0), (y,z)],|\cdot|_{t} \times |\cdot|_{s}).
\end{align*}
We will study $4$-dimensional normed spaces of the above form.

\begin{lem} \label{lem7}
  Let $x, y, z \in K^{\lor} \setminus K$ and $t, s \in \mathbb{R}_{> 0}$. We set 
  \begin{align*}
    (E, \|\cdot\|) := ([(1,0), (0,1), (x,0), (y,z)],|\cdot|_{t} \times |\cdot|_{s}),
  \end{align*}
  and put $\gamma = d(z,K)$ and $\delta = d(y, [1,x])$. \\
  $(1)$ If $t \delta < s \gamma$, then $E$ is isometrically isomorphic to
  \begin{align*}
    ([1,x], |\cdot|_t) \oplus ([1,z], |\cdot|_s)
  \end{align*}
  either of type $\mathrm{\Rnum{6}}_4$ or $\mathrm{\Rnum{7}}_4$. \\
  $(2)$ If $t \delta > s \gamma$, then $E$ is isometrically isomorphic to
  \begin{align*}
    ([1, x, y], |\cdot|_t) \oplus (K, |\cdot|_s),
  \end{align*}
  either of type $\mathrm{\Rnum{4}}_4$ or $\mathrm{\Rnum{5}}_4$. 
\end{lem}
\begin{proof}
  First, suppose $t \delta < s \gamma$. Then there exist $a, b \in K$ for which $|y - ax - b|_t < s \gamma$. Then for each $\lambda, \mu \in K$, we have
  \begin{align*}
    d_E (\lambda (y - ax - b,z) + \mu (0,1), [(1,0),(x,0)]) &= |\lambda z + \mu|_s \\
    &= \|\lambda (y - ax - b,z) + \mu (0,1)\|.
  \end{align*}
  Therefore, we obtain $E \cong ([1,x], |\cdot|_t) \oplus ([1,z], |\cdot|_s)$. \par
  Secondly, suppose $t \delta > s \gamma$. Then there exists $c \in K$ for which $|z - c|_s < t \delta$. Then for each $\lambda, \mu, \eta \in K$, we have
  \begin{align*}
    &d_E (\lambda (y, z -c) + \mu (x,0) + \eta (1,0), [(0,1)]) \\
     = \quad &|\lambda y + \mu x + \eta|_t = \|\lambda (y, z-a) + \mu (x,0) + \eta (1,0)\|.
  \end{align*}
  Therefore, we have $E \cong ([1, x, y], |\cdot|_t) \oplus (K, |\cdot|_s)$, which completes the proof.
\end{proof}

From the lemma above, if $E$ is indecomposable, then $t \delta = s \gamma$. However, the converse is not true.

\begin{lem} \label{lem8}
  Let us keep the notation in Lemma \ref{lem7}. Suppose $t \delta = s \gamma$ and that $([1, x, y], |\cdot|_t)$ is of type $\mathrm{\Rnum{5}}_3$ and of subtype $(x,w,t)$. If $w \sim z$, then $E$ is isometrically isomorphic to
  \begin{align*}
    ([1, x, y], |\cdot|_t) \oplus (K,|\cdot|_s). 
  \end{align*}
\end{lem}
\begin{proof}
  Let $T$ be a linear isometry from $[(1,0), (x,0), (y,0)] / [(1,0)]$ to $([1, w], |\cdot|_{t \alpha})$ such that $T(\pi((x,0))) = 1$ where $\alpha = d (x, K)$ and $\pi : E \to E / [(1,0)]$ is the canonical quotient map. Put $w' = T(\pi((y,0)))$. Since $T$ is an isometry, we have
  \begin{align*}
    t \alpha \cdot d(w',K) = d(\pi((y,0)), [\pi((x,0))]) = t d(y,[1,x]) = t \delta.
  \end{align*}
  Hence, $d(w', K) = \delta / \alpha$. Moreover, by assumption, we have $z \sim w'$. Thus, there exist $a, b \in K$ with $|a| = \gamma \alpha / \delta$ such that $|z - a w' - b| < d(z,K)$. Then it is easy to see
  \begin{align*}
    E &\cong ([(1,0), (0,1), (x,0), (y,a w')],|\cdot|_{t} \times |\cdot|_{s}) \\
    &\cong ([(1,0), (0,1), (x,0), (y,w')],|\cdot|_{t} \times |\cdot|_{\frac{s \gamma \alpha}{\delta}}) \\
    &= ([(1,0), (0,1), (x,0), (y,w')],|\cdot|_{t} \times |\cdot|_{t \alpha}) 
  \end{align*}
  where the last equality follows from $t \delta = s \gamma$. Now, for each $\lambda, \mu, \eta \in K$, we have
  \begin{align*}
    |\lambda y + \mu x + \eta|_t \ge \|\lambda \pi((y,0)) + \mu \pi((x,0))\| = |\lambda w' + \mu|_{t \alpha},
  \end{align*}
  and hence,
  \begin{align*}
   |\lambda y + \mu x + \eta|_t &= \|\lambda (y, w') + \mu (x,1) + \eta (1,0)\| \\ 
   &\ge d(\lambda (y, w') + \mu (x,1) + \eta (1,0),[(0,1)]).
  \end{align*}
  Therefore, we obtain
  \begin{align*}
    E \cong ([1, x, y],|\cdot|_t) \oplus (K, |\cdot|_{t \alpha}).
  \end{align*}
  Finally, since $z \sim w'$, we have $d(z,K)/d(w',K) = t \alpha / s \in V_K$, which completes the proof.
\end{proof}

\begin{lem} \label{lem11}
  Let us keep the notation in Lemma \ref{lem7}. Suppose $t \delta = s \gamma$ and that $([1, x, y], |\cdot|_t)$ is of type $\mathrm{\Rnum{4}}_3$. If there exists $w \in [1, x, y] \setminus K$ such that $w \nsim x$ and $w \sim z$. Then $E$ is isometrically isomorphic to
  \begin{align*}
    ([1,x],|\cdot|_t) \oplus ([1, w], |\cdot|_t). 
  \end{align*}
\end{lem}
\begin{proof}
  Let $a \in K$ be such that $y + a x \nsim x$. Obviously, we have $d(y,[1,x]) = d(y + ax,[1,x])$. Therefore, without loss of generality, we may assume $x \nsim y$ and $w$ is of the form $w = y + \lambda x$ where $\lambda \in K$. Since $w \nsim x$, we have 
  \begin{align*}
    d(w,K) = d(w,[1,x]) = d(y,K) = d(y,[1,x]) = \delta.
  \end{align*} 
  Hence by assumption $w \sim z$, there exists $\mu, \eta \in K$ with $|\mu| = \gamma / \delta$ for which $|z - \mu w - \eta| < d(z,K)$. Therefore, we have
  \begin{align*}
    E &\cong ([(1,0), (0,1), (x,0), (y,\mu w)],|\cdot|_{t} \times |\cdot|_{s}) \\
    &\cong ([(1,0), (0,1), (x,0), (y,w)],|\cdot|_{t} \times |\cdot|_{s \gamma / \delta}) \\
    &= ([(1,0), (0,1), (x,0), (y,w)],|\cdot|_{t} \times |\cdot|_{t}) 
  \end{align*}
  where the last equality follows from $t \delta = s \gamma$. Now, it is easy to see
  \begin{align*}
    ([(1,0), (0,1), (x,0), (y,w)],|\cdot|_{t} \times |\cdot|_{t}) \cong [(1,0),(x,0)] \oplus [(1,1),(w,w)],
  \end{align*}
  which completes the proof.
\end{proof}

\begin{cor} \label{cor1}
  Let us keep the notation in Lemma \ref{lem7}. Suppose that $E$ is indecomposable and $([1, x, y], |\cdot|_t)$ is of type $\mathrm{\Rnum{4}}_3$. Then $E$ contains a $3$-dimensional subspace of $\mathrm{\Rnum{3}}_3$.
\end{cor}
\begin{proof}
  Without loss of generality, we may assume $x \nsim y$. Then by the preceding lemma, we have $y \nsim z$. Since $t d(y,K) = t \delta = s \gamma$, $[(1,0),(0,1),(y,z)]$ is of $\mathrm{\Rnum{3}}_3$.
\end{proof}

The following lemma is easy to prove, and thus we omit the proof.

\begin{lem} \label{lem10}
  Let us keep the notation in Lemma \ref{lem7} and suppose $t \delta = s \gamma$. Let $\pi : E \to E / [(0,1)]$ be the canonical quotient map. Then the map
  \begin{align*}
     E / [(0,1)] &\to ([1,x,y],|\cdot|_t) \\
    \pi((1,0)) &\mapsto 1 \\
    \pi((x,0)) &\mapsto x \\
    \pi((y,z)) &\mapsto y
  \end{align*}
  gives a linear isometry.
\end{lem}

\begin{lem} \label{lem9}
  Let us keep the notation in Lemma \ref{lem7} and suppose $t \delta = s \gamma$. Put $\alpha = d(x,K)$. Consider $E$ as a subspace of $E_1 := E + [(y,0)]$. Let $\pi : E_1 \to E_1 / [(1,0)]$ be the canonical quotient map.\\
  $(1)$ Suppose $x \nsim y$. Then the map
  \begin{align*}
    E / [(1,0)] &\to (K,|\cdot|_{t \alpha}) \oplus ([1, z], |\cdot|_s) \\
    \pi((x,0)) &\mapsto (1,0) \\
    \pi((0,1)) &\mapsto (0,1) \\
    \pi((y,z)) &\mapsto (0,z) 
  \end{align*}
  gives a linear isometry. \\
  $(2)$ Suppose that $([1, x, y], |\cdot|_t)$ is of type $\mathrm{\Rnum{5}}_3$ and of subtype $(x,w,t)$. Let $T$ be a linear isometry from $[(1,0), (x,0), (y,0)] / [(1,0)]$ to $([1, w], |\cdot|_{t \alpha})$ such that $T(\pi((x,0))) = 1$. Put $w' = T(\pi((y,0)))$. Then $d(w', K) = \delta / \alpha$ in $(K^{\lor},|\cdot|)$ and the map
  \begin{align*}
    E / [(1,0)] &\to ([(1,0), (0,1), (w', z)],|\cdot|_{t \alpha} \times |\cdot|_{s}) \\
    \pi((x,0)) &\mapsto (1,0) \\
    \pi((0,1)) &\mapsto (0,1) \\
    \pi((y,z)) &\mapsto (w',z) \\
  \end{align*}
  gives a linear isometry.
\end{lem}
\begin{proof}
  First, suppose $x \nsim y$. Then we have $\delta = d(y, K)$ and $\{\pi((x,0)), \pi((y,0))\}$ is an orthogonal set. Therefore, for each $\lambda, \mu, \eta \in K$, we obtain
  \begin{align*}
    \|\lambda \pi((x,0)) + \mu \pi((y,z)) + \eta \pi((0,1))\| &= |\lambda|_{t \alpha} \lor t \delta |\mu| \lor s |\mu z + \eta| \\
    &= |\lambda|_{t \alpha} \lor s \gamma |\mu| \lor s |\mu z + \eta| \\
    &= |\lambda|_{t \alpha} \lor |\mu z + \eta|_s
  \end{align*}
  Thus, $(1)$ is proved. \\
  Secondly, we shall prove $(2)$. By the same proof as that of Lemma \ref{lem8}, we have $d(w', K) = \delta / \alpha$. Moreover, for $\lambda, \mu, \eta \in K$, we have
  \begin{align*}
    \|\lambda \pi((x,0)) + \mu \pi((0,1)) + \eta \pi((y,z))\| = |\lambda + \eta w'|_{t \alpha} \lor |\mu + \eta z|_s.
  \end{align*}
  Thus, the proof is complete.
\end{proof}

\begin{rem} \label{rem2}
  Let us keep the notation in Lemma \ref{lem7}. Then there exists a linear isometry
  \begin{align*}
    E &\to ([(1,0), (0,1), (\frac{1}{x}, 1), (\frac{y}{x}, z)],|\cdot|_{t |x|} \times |\cdot|_s) \\
    (x,0) &\mapsto (1,0) \\
    (0,1) &\mapsto (0,1) \\
    (1,0) &\mapsto (\frac{1}{x},0) \\
    (y,z) &\mapsto (\frac{y}{x},z).
  \end{align*}
  Moreover, we see $d(\tfrac{y}{x},[1, \tfrac{1}{x}]) = d(y,[1,x]) / |x|$. Hence, if $t \delta = s \gamma$, then we have
  \begin{align*}
    (t |x|) \cdot d(\frac{y}{x},[1, \frac{1}{x}]) = s \gamma.
  \end{align*}
\end{rem}

Now, we obtain the following theorem that characterizes $4$-dimensional normed spaces of type $\mathrm{\Rnum{14}}_4$.

 \begin{thm} \label{thm11}
  Let $x, y, z \in K^{\lor} \setminus K$ and $t, s \in \mathbb{R}_{> 0}$. We set 
  \begin{align*}
    (E, \|\cdot\|) := ([(1,0), (0,1), (x,0), (y,z)],|\cdot|_{t} \times |\cdot|_{s}),
  \end{align*}
  and put $\gamma = d(z,K)$ and $\delta = d(y, [1,x])$. \\
  $(1)$ $E$ is of type $\mathrm{\Rnum{14}}_4$ if and only if $t \delta = s \gamma$, and $([1,x,y],|\cdot|_t)$ is of type $\mathrm{\Rnum{5}}_3$ and of subtype $(x,w,t)$ with $w \nsim z$. \\
  $(2)$ Suppose that $E$ is of type $\mathrm{\Rnum{14}}_4$. Then each $3$-dimensional subspace of $E$ is isometrically isomorphic to $([1,x],|\cdot|_t) \oplus (K,|\cdot|_s)$. \\
  $(3)$ Suppose that $E$ is of type $\mathrm{\Rnum{14}}_4$. Then $E$ satisfies $\mathrm{(SE)}$ if and only if
  \begin{align*}
    t, s, td(x,K) \notin V_K.
  \end{align*}
\end{thm}
\begin{proof}
  First, we shall prove $(1)$. Suppose $E$ is of type $\mathrm{\Rnum{14}}_4$. Then by Lemma \ref{lem7}, we have $t \delta = s \gamma$. Since $E$ contains no $3$-dimensional subspace of type $\mathrm{\Rnum{3}}_3$, by Corollary \ref{cor1}, $([1,x,y],|\cdot|_t)$ is of type $\mathrm{\Rnum{5}}_3$. Now, let $w \in K^{\lor} \setminus K$ be such that $E$ is of subtype $(x,w,t)$. Then by Lemma \ref{lem8}, we obtain $w \nsim z$. \par
  Conversely, suppose $t \delta = s \gamma$, and that $([1,x,y],|\cdot|_t)$ is of type $\mathrm{\Rnum{5}}_3$ and of subtype $(x,w,t)$ with $w \nsim z$. We shall prove $\dim_{K^{\lor}} (E')^{\lor} = 1$, that is $\dim_{K^{\lor}} (E/F)^{\lor} = 1$ for each $2$-dimensional subspace $F$ of $E$. Let $F$ be a $2$-dimensional subspace of $E$ and $u, v \in F$ be such that $F = [u,v]$. It suffices to consider the following six cases:
  \begin{itemize}
    \item[$(A)$] $u = (0,1)$, $v = (1,0)$. 
    \item[$(B)$] $u = (x,0) + \lambda (0,1)$, $v = (1,0)$ $(\lambda \in K)$. 
    \item[$(C)$] $u = (y,z) + \lambda (x,0) + \mu (0,1)$, $v = (1,0)$ $(\lambda, \mu, \eta \in K)$. 
    \item[$(D)$] $u = (x,0) + \mu (1,0)$, $v = (0,1) + a(1,0)$ $(\mu,  a \in K)$.
    \item[$(E)$] $u = (y,z) + \lambda (x,0) + \eta (1,0)$, $v = (0,1) + a(1,0)$ $(\lambda, \mu, \eta, a \in K)$. 
    \item[$(F)$] $u = (y,z) + \mu (0,1) + \eta (1,0)$, $v = (x,0) + a(0,1) + b(1,0)$ $(\lambda, \mu, \eta, a, b\in K)$.
  \end{itemize}
  By Lemma \ref{lem9} and $w \nsim z$, $E / [(1,0)]$ is of type $\mathrm{\Rnum{3}}_3$. Therefore, in cases $(A), (B)$ and $(C)$, we are done from Corollary \ref{.4cor2}. Secondly, consider the case $(D)$. Since $d(y,[1,x]) = d(y,[1,x + \mu])$, we may assume $\mu = 0$. Then by Remark \ref{rem2}, the case $(D)$ can be reduced to the case $(A)$ or $(B)$. \par
  Thirdly, consider the case $(F)$. Considering
  \begin{align*}
    x' = x + b, y' = y + \eta \ \text{and} \ z' = z + \mu,
  \end{align*}
  without loss of generality, we may assume $u = (y,z), v = (x,a)$. Then by Remark \ref{rem2}, the case $(F)$ can be reduced to the case $(C)$ or $(E)$. \par
  Now, it remains to consider the case $(E)$. Considering $y' = y + \lambda x + \eta$, without loss of generality, we may assume $u = (y,z), v = (a,1)$. Let $\pi : E \to E / [v]$ be the canonical quotient map. First, suppose $|a|_t \le |1|_s$. We shall prove $\dim_{K^{\lor}} (E/[v])^{\lor} = 1$. Since $|a|_t \le |1|_s$, the map
  \begin{align*}
    [\pi((1,0)), \pi((x,0))] &\to ([1,x],|\cdot|_t) \\
    \pi((1,0)) &\mapsto 1 \\
    \pi((x,0)) &\mapsto x
  \end{align*}
  gives a linear isometry. In particular, $\dim_{K^{\lor}} ( [\pi((1,0)), \pi((x,0))])^{\lor} = 1$. It remains to prove that $E / [v]$ is an immediate extension of $[\pi((1,0)), \pi((x,0))]$. Let $\lambda, \mu \in K$. Then
  \begin{align*}
    \|\pi((y,z)) + \lambda \pi((x,0)) + \mu \pi((1,0))\| &= \inf_{\eta \in K} \|(y + \lambda x + \mu, z) + \eta (a,1)\| \\
    &\ge t d(y + \lambda x + \mu, K) \\
    &> t d(y, [1,x])
  \end{align*}
  where the last inequality follows because $([1,x,y],|\cdot|_t)$ is of type $\mathrm{\Rnum{5}}_3$. Now, since $d(\pi((y,z)), [\pi((1,0)), \pi((x,0))]) = t d(y,[1,x])$, $E / [v]$ is an immediate extension of $[\pi((1,0)), \pi((x,0))]$. Hence, $\dim_{K^{\lor}} (E/[v])^{\lor} = 1$ and $E / [v]$ is either of  type $\mathrm{\Rnum{4}}_3$ or $\mathrm{\Rnum{5}}_3$. On the other hand, by Lemma \ref{lem10}, we have
  \begin{align*}
    (E/[v]) \big/ [\pi((1,0))] \cong E/ [(1,0),(0,1)] \cong ([1,x,y],|\cdot|_t) / K.
  \end{align*}
  Therefore, since $([1,x,y],|\cdot|_t)$ is of type $\mathrm{\Rnum{5}}_3$, $E / [v]$ must be of type $\mathrm{\Rnum{5}}_3$. Thus, we obtain $\dim_{K^{\lor}} (E/F)^{\lor} = 1$. \par
  Finally suppose $|a|_t > |1|_s$. Let $\lambda \in K$ be such that $|x +\lambda| < |a|_t \alpha/ |1|_s$ where $\alpha = d(x,K)$. Put $x' = x + \lambda$. By Lemma \ref{lem10}, we have
  \begin{align*}
    (E/[v]) \big/ [\pi((1,0))] \cong E/ [(1,0),(0,1)] \cong ([1,x,y],|\cdot|_t) / K \cong ([1,w],|\cdot|_{t \alpha}).
  \end{align*}
  On the other hand, by Lemma \ref{lem9} and Remark \ref{rem2}, we have
  \begin{align*}
    (E/[v]) \big/ [\pi((x',0))] &\cong E /[(a,1),(x',0)] \\
    &\cong ([(1,0)(0,1)(1/x',0)(y/x',z)],|\cdot|_{t |x'|} \times |\cdot|_s) \big/ [(a/x', 1), (1,0)] \\
    &\cong ([(1,0),(0,1),(w',z)],|\cdot|_{t \alpha/ |x'|} \times |\cdot|_s) \big/ [(a,1)]
  \end{align*}
  where $w' \sim w$ and $d(w',K) = (\delta/|x'|) \big/ d(\tfrac{1}{x'},K) = |x'| \delta / \alpha$. Hence, we have
  \begin{align*}
    w' \nsim z \ \text{and} \ \frac{t \alpha d(w', K)}{|x'|} = s d(z, K).
  \end{align*}
  Now, by Corollary \ref{.4cor3}, we obtain
  \begin{align*}
    ([(1,0),(0,1),(w',z)],|\cdot|_{t \alpha/ |x'|} \times |\cdot|_s) \big/ [(a,1)] \cong ([1,w' - az],|\cdot|_{\frac{s \gamma}{\beta}})
  \end{align*}
  where $\beta = d(w' - az, K)$. Moreover, by the choice of $a$, we have
  \begin{align*}
    d(w',K) = \frac{|x'| \delta}{\alpha} < \frac{|a|t \delta}{s} = |a| \gamma = d(az,K).
  \end{align*}
  Thus, we obtain $az - w' \sim z$ and $([1,az - w'],|\cdot|_{\frac{s \gamma}{\beta}}) \cong ([1,z],|\cdot|_{\frac{s \gamma}{\beta}})$. From what has been proved, we see that both
  \begin{align*}
    ([1,w],|\cdot|_{t \alpha}) \ \text{and} \ ([1,z],|\cdot|_{\frac{s \gamma}{\beta}})
  \end{align*}
  are quotient spaces of the $3$-dimensional normed space $E / [v]$. Since $w \nsim z$, $E / [v]$ must be of type $\mathrm{\Rnum{3}}_3$. In particular, by Corollary \ref{.4cor2}, $\dim_{K^{\lor}} (E/F)^{\lor} = 1$. \par
  Thus, we have proved $\dim_{K^{\lor}} (E')^{\lor} = 1$. In particular, $E'$ is indecomposable, and so is $E$. Since $E$ contains a $3$-dimensional normed space of type $\mathrm{\Rnum{2}}_3$, isometrically isomorphic to
  \begin{align*}
    ([1,x],|\cdot|_t) \oplus (K,|\cdot|_s),
  \end{align*}
  it follows from Theorem \ref{thm10} that $E'$ is of type $\mathrm{\Rnum{9}}_4$ and $E$ is of type $\mathrm{\Rnum{14}}_4$. Thus, $(1)$ is proved. Moreover, by Corollary \ref{syscor1}, we obtain $(2)$. \par
  Finally, we prove $(3)$. Since $\dim_{K^{\lor}} (E')^{\lor} = 1$, we can apply Corollary \ref{syscor3}. Hence, $E$ satisfies (SE) if and only if $\| \big(([1,x],|\cdot|_t) \oplus (K,|\cdot|_s)\big)' \| \cap V_K = \emptyset$ and $([1,x],|\cdot|_t) \oplus (K,|\cdot|_s)$ satisfies (SE). Then by Theorems \ref{.4thm7} and \ref{.4thm4}, we complete the proof.
\end{proof}

\begin{cor} \label{cor2}
  Let $E$ be a $4$-dimensional normed space of type $\mathrm{\Rnum{14}}_4$. Then $E'$ is of type $\mathrm{\Rnum{9}}_4$ and all $3$-dimensional subspaces of $E$ are isometrically isomorphic to each other.
\end{cor}
\begin{proof}
  By definition, $\dim_{K^{\lor}} E^{\lor} = 2$ and $E$ contains a $3$-dimensional normed space of type $\mathrm{\Rnum{2}}_3$. Therefore, there exist $x, y, z \in K^{\lor} \setminus K$ and $t, s \in \mathbb{R}_{> 0}$ such that
  \begin{align*}
    E \cong ([(1,0), (0,1), (x,0), (y,z)],|\cdot|_{t} \times |\cdot|_{s}).
  \end{align*}
  Thus, by the proof of the preceding theorem, the proof is complete.
\end{proof}

We also characterize $4$-dimensional normed spaces of type $\mathrm{\Rnum{15}}_4$.

\begin{thm} \label{thm13}
  Let $E$ be a $4$-dimensional normed space of type $\mathrm{\Rnum{15}}_4$. Then $E'$ is of type $\mathrm{\Rnum{15}}_4$.
\end{thm}
\begin{proof}
  By the same proof as that of the preceding corollary, there exist $x, y, z \in K^{\lor} \setminus K$ and $t, s \in \mathbb{R}_{> 0}$ such that
  \begin{align*}
    E \cong ([(1,0), (0,1), (x,0), (y,z)],|\cdot|_{t} \times |\cdot|_{s}).
  \end{align*}
  Since $E$ is indecomposable, we have $t \delta = s \gamma$ by Lemma \ref{lem7}. Moreover, by Lemma \ref{lem8} and Theorem \ref{thm11}, $([1,x,y],|\cdot|_t)$ is of type $\mathrm{\Rnum{4}}_3$. Hence by Lemma \ref{lem10} (resp. Lemma \ref{lem9}), $E'$ contains a $3$-dimensional subspace of type $\mathrm{\Rnum{3}}_3$ (resp. $\mathrm{\Rnum{2}}_3$). In particular, $\dim_{K^{\lor}} (E')^{\lor} \ge 2$. If $\dim_{K^{\lor}} (E')^{\lor} = 3$, then $E'$ is of type $\mathrm{\Rnum{13}}_4$ and $E$ is of type $\mathrm{\Rnum{8}}_3$ by Corollary \ref{cor3}, which is a contradiction. Hence, we have $\dim_{K^{\lor}} (E')^{\lor} = 2$ and $E'$ is of type $\mathrm{\Rnum{15}}_4$.
\end{proof}

\begin{thm} \label{thm14}
  Let $x, y, z \in K^{\lor} \setminus K$ and $t, s \in \mathbb{R}_{> 0}$. We set 
  \begin{align*}
    (E, \|\cdot\|) := ([(1,0), (0,1), (x,0), (y,z)],|\cdot|_{t} \times |\cdot|_{s}),
  \end{align*}
  and put $\gamma = d(z,K)$ and $\delta = d(y, [1,x])$. \\
  $(1)$ $E$ is of type $\mathrm{\Rnum{15}}_4$ if and only if $E$ satisfies the following three conditions:
  \begin{itemize}
    \item[$(a)$] $t \delta = s \gamma$.
    \item[$(b)$] $([1,x,y],|\cdot|_t)$ is of type $\mathrm{\Rnum{4}}_3$.
    \item[$(c)$] $w \nsim z$ for each $w \in [1,x,y] \setminus K$ satisfying $w \nsim x$.
  \end{itemize}
  $(2)$ Suppose that $E$ is of type $\mathrm{\Rnum{15}}_4$. Then each $2$-dimensional subspace of $E$ is isometrically isomorphic to either
  \begin{align*}
    (K^2,|\cdot|_t \times |\cdot|_s) \ \text{or} \ ([1,x],|\cdot|_t).
  \end{align*}
  Moreover, each $3$-dimensional subspace of $E$, of type $\mathrm{\Rnum{2}}_3$, is isometrically isomorphic to
  \begin{align*}
    ([1,x],|\cdot|_t) \oplus (K,|\cdot|_s).
  \end{align*}
  \noindent $(3)$ Suppose that $E$ is of type $\mathrm{\Rnum{15}}_4$. Then $E$ satisfies $\mathrm{(SE)}$ if and only if $t, s \notin V_K$.
\end{thm}
\begin{proof}
  First, we prove $(1)$ and $(2)$. Suppose that $E$ is of type $\mathrm{\Rnum{15}}_4$. Then by the same proof as that of Theorem \ref{thm13}, $E$ satisfies $(a)$ and $(b)$. Moreover, by Lemma \ref{lem11}, we have $(c)$. \par
  Conversely, suppose that $(a)$, $(b)$ and $(c)$ hold. We shall prove that each $2$-dimensional subspace of $E$ is isometrically isomorphic to either
  \begin{align*}
    (K^2,|\cdot|_t \times |\cdot|_s) \ \text{or} \ ([1,x],|\cdot|_t).
  \end{align*}
  Let $F$ be a $2$-dimensional subspace of $E$. Then $F$ is contained in one of the following $3$-dimensional subspaces of $E$: 
  \begin{itemize}
    \item[$(A)$] $[(1,0),(0,1),(x,0)]$. 
    \item[$(B)$] $[(1,0),(0,1),(y + \lambda x,z)]$ $(\lambda \in K)$.
    \item[$(C)$] $[(0,1), (x + \mu, 0),(y + \lambda, z)]$ $(\lambda, \mu \in K)$.
  \end{itemize}
  In the case $(A)$, we are done from Theorem \ref{.4thm6}. In the case $(B)$, we observe that the following equivalence holds:
  \begin{align*}
    y + \lambda x \nsim x \Leftrightarrow t d(y + \lambda x, K) = t \delta.
  \end{align*}
  Hence, if $y + \lambda x \nsim x$, then by Theorem \ref{.4thm2} and the condition $(c)$, $[(1,0),(0,1),(y + \lambda x,z)]$ is of type $\mathrm{\Rnum{3}}_3$ and $F$ is isometrically isomorphic to $(K^2,|\cdot|_t \times |\cdot|_s)$. If $y + \lambda x \sim x$, then $t d(y + \lambda x, K) > t \delta$, and therefore, we have
  \begin{align*}
    [(1,0),(0,1),(y + \lambda x,z)] \cong ([1,y + \lambda x],|\cdot|_t) \oplus (K,|\cdot|_s) \cong ([1,x],|\cdot|_t) \oplus (K,|\cdot|_s)
  \end{align*}
  by Remark \ref{sysrem1}. Thus, by Theorem \ref{.4thm6}, $F$ is isometrically isomorphic to either
    \begin{align*}
    (K^2,|\cdot|_t \times |\cdot|_s) \ \text{or} \ ([1,x],|\cdot|_t).
  \end{align*}
  Finally, consider the case $(C)$. Since $d(\tfrac{y + \lambda}{x + \mu},[1, \tfrac{1}{x + \mu}]) = \delta / |x + \mu|$, we have
  \begin{align*}
    \frac{y + \lambda}{x + \mu} \nsim \frac{1}{x + \mu} \Leftrightarrow d(\frac{y + \lambda}{x + \mu}, K) = \frac{\delta}{|x + \mu|}.
  \end{align*}
  Moreover, by Theorem \ref{.4thm3} and the condition $(c)$, $\tfrac{y + \lambda}{x + \mu} \nsim \tfrac{1}{x + \mu}$ implies $\tfrac{y + \lambda}{x + \mu} \nsim z$. Furthermore, we have
  \begin{align*}
    [(0,1), (x + \mu, 0),(y + \lambda, z)] \cong ([(1,0),(0,1),(\frac{y + \lambda}{x + \mu}, z)], |\cdot|_{t |x + \mu|} \times |\cdot|_s).
  \end{align*}
  Now, we can apply a proof similar to that in the case $(B)$. \par
  Thus, we have proved each $2$-dimensional subspace of $E$ is isometrically isomorphic to either
  \begin{align*}
    (K^2,|\cdot|_t \times |\cdot|_s) \ \text{or} \ ([1,x],|\cdot|_t).
  \end{align*}
  Let us prove that $E$ is of type $\mathrm{\Rnum{15}}_4$. Indeed, since $([1,x, y],|\cdot|_t)$ is of type $\mathrm{\Rnum{4}}_3$, there exists $\lambda \in K$ such that $y + \lambda x \nsim x$. Then $d(y + \lambda x,K) = \delta$ and therefore a subspace $[(1,0),(0,1),(y + \lambda x,z)]$ is of type $\mathrm{\Rnum{3}}_3$. Obviously, $E$ contains a $3$-dimensional subspace of type $\mathrm{\Rnum{2}}_3$. It remains to prove that $E$ is indecomposable. Suppose that $E$ is decomposable. Then since $\dim_{K^{\lor}} E^{\lor} = 2$ and $E$ contains a $3$-dimensional subspace of type $\mathrm{\Rnum{3}}_3$, by Lemmas \ref{lem5} and \ref{lem6}, $E$ must be of type $\mathrm{\Rnum{6}}_4$. However, each $2$-dimensional subspace without an orthogonal base of $E$ is isometrically isomorphic to $([1,x],|\cdot|_t)$. Therefore, $E$ cannot be of type $\mathrm{\Rnum{6}}_4$ (see Theorem \ref{thm2}), which is a contradiction. Thus, $(1)$ is proved. \par
  Now, we shall prove that each $3$-dimensional subspace, of type $\mathrm{\Rnum{2}}_3$, of $E$ is isometrically isomorphic to
  \begin{align*}
    ([1,x],|\cdot|_t) \oplus (K,|\cdot|_s).
  \end{align*}
  Indeed, let $F$ be a $3$-dimensional subspace, of type $\mathrm{\Rnum{2}}_3$, of $E$. Since each $2$-dimensional subspace without an orthogonal base of $E$ is isometrically isomorphic to $([1,x],|\cdot|_t)$, we have
  \begin{align*}
    F \cong ([1,x],|\cdot|_t) \oplus (K,|\cdot|_{t'}) \ \text{for some $t' \in \mathbb{R}_{> 0}$}.
  \end{align*}
 Since $\dim_{K^{\lor}} E^{\lor} = \dim_{K^{\lor}} F^{\lor}$, $E$ is an immediate extension of $F$. In particular, we have $\|E\| = \|F\|$. Hence, we obtain $s / t' \in V_K$ and $(2)$ is proved. \par
  Finally, we shall prove $(3)$. We will use Lemma \ref{selem1}. Suppose $t, s \notin V_K$. Then $\|E\| \cap V_K = \emptyset$. Let $F$ be a $3$-dimensional subspace of $E$. If $F$ is of type $\mathrm{\Rnum{2}}_3$, then by $(2)$ and Theorem \ref{.4thm4}, $F$ satisfies (SE). If $F$ is of type $\mathrm{\Rnum{3}}_3$, then by Theorem \ref{.4thm2}, $F$ satisfies (SE). Thus, by Lemma \ref{lem12} below, each $3$-dimensional subspace of $E$ satisfies (SE). Now, we will prove that the condition $(2)$ in Lemma \ref{selem1} is satisfied. By Lemma \ref{lem9} and Theorem \ref{thm13}, $E'$ is of type $\mathrm{\Rnum{15}}_4$ and contains a $3$-dimensional subspace, isometrically isomorphic to 
  \begin{align*}
    (K,|\cdot|_{\frac{1}{t \alpha}}) \oplus ([1,z],|\cdot|_{\frac{1}{s \gamma}}) \ \text{where $ \alpha = d(x,K)$}.
  \end{align*}
  Then by $(2)$, each $2$-dimensional subspace without an orthogonal base is isometrically isomorphic to $([1,z],|\cdot|_{\frac{1}{s \gamma}})$. Since $d(z,K)/s\gamma = 1/s \notin V_K$, by Lemma \ref{selem1}, $E$ satisfies (SE). \par
  Conversely, suppose that $E$ satisfies (SE). Then each $3$-dimensional subspace of $E$ satisfies (SE). Let $\lambda \in K$ be such that $y + \lambda x \nsim x$. Then a $3$-dimensional normed space $[(1,0),(0,1),(y + \lambda x,z)]$ of $\mathrm{\Rnum{3}}_3$ satisfies (SE). Hence by Theorem \ref{.4thm2}, we have $t, s \notin V_K$. Thus, the proof is complete.
\end{proof}

\begin{lem} \label{lem12}
  Let $E$ be an indecomposable $4$-dimensional normed space such that $\dim_{K^{\lor}} E^{\lor} = 2$. Then each $3$-dimensional subspace of $E$ is either of type $\mathrm{\Rnum{2}}_3$ or $\mathrm{\Rnum{3}}_3$.
\end{lem}
\begin{proof}
  Let $F$ be a $3$-dimensional subspace of $E$. Obviously, we have $\dim_{K^{\lor}} F^{\lor} \le 2$. Suppose $\dim_{K^{\lor}} F^{\lor} = 1$. Let $u \in F$ be a non-zero element. Since $\dim_{K^{\lor}} E^{\lor} = 2$, there exists $v \in E$ such that $\{u,v\}$ is an orthogonal set. Thus, since $F$ is an immediate extension of $[u]$, we have $E \cong F \oplus [v]$, which is a contradiction.
\end{proof}

By Theorem \ref{thm14}, we can identify $3$-dimensional subspaces, of type $\mathrm{\Rnum{2}}_3$, of a $4$-dimensional normed space of type $\mathrm{\Rnum{15}}_4$. In fact, we can identify $3$-dimensional subspaces of type $\mathrm{\Rnum{3}}_3$.

\begin{lem} \label{lem15}/
  Let $x, y, z \in K^{\lor} \setminus K$, $\lambda \in K$ and $t, s \in \mathbb{R}_{> 0}$. Put $\gamma = d(z, K)$ and $\delta = d(y, [1, x])$. Suppose that the following four conditions hold: \\
  $(1)$ $t \delta = s \gamma$. \\
  $(2)$ The $3$-dimensional normed space $([1, x, y],|\cdot|_t)$ is of type $\mathrm{\Rnum{4}}_3$. \\
  $(3)$ $w \nsim z$ for each $w \in [1, x, y] \setminus K$ satisfying $w \nsim x$. \\
  $(4)$ The $3$-dimensional normed space $F := ([(1,0),(x, \lambda), (y,z)],|\cdot|_t \times |\cdot|_s)$ is of type $\mathrm{\Rnum{3}}_3$. \par
  Then there exists $\mu \in K$ such that $y + \mu x \nsim x$ and
  \begin{align*}
    F \cong ([(1,0),(0,1),(y + \mu x,z)], |\cdot|_t \times |\cdot|_{s}).
  \end{align*}
\end{lem}
\begin{proof}
  By the condition $(4)$, each $2$-dimensional subspace of $F$ has an orthogonal base. Therefore, there exists $\lambda_0 \in K$ such that $t |x + \lambda_0| < s |\lambda|$. In particular, $\lambda \neq 0$. Put $x' = x + \lambda_0$. By the condition $(2)$, there exist $\mu_0, \eta_0 \in K$ such that
  \begin{align*}
    y + \mu_0 x' + \eta_0 \nsim x' \ \text{and} \ \frac{y + \mu_0 x' + \eta_0}{x'} \nsim \frac{1}{x'}.
  \end{align*}
  Thus, without loss of generality, we may assume
  \begin{align*}
    \lambda = 1, t |x| < s, y \nsim x \ \text{and} \ \frac{y}{x} \nsim \frac{1}{x}.
  \end{align*}
  Then we have
  \begin{align*}
    d(y,K) = \delta \ \text{and} \ d(\frac{y}{x}, K) = \frac{\delta}{|x|}.
  \end{align*}
  We set $(E, \|\cdot\|) := ([(1,0), (0,1), (x,0), (y,z)],|\cdot|_{t} \times |\cdot|_{s})$ and consider $F$ as a subspace of $E$. By Corollary \ref{.4cor2}, we have $\dim_{K^{\lor}} (F / [(1,0)])^{\lor} = 1$. Hence, by Lemmas \ref{lemh3} and \ref{lem9}, we obtain
  \begin{align*}
    F / [(1,0)] \cong ([1,z],|\cdot|_s) \ \text{and} \ t \alpha \le \frac{s \gamma}{|z|}, \text{hence} \ |z| \alpha \le \delta
  \end{align*}
  where $\alpha = d(x,K)$. We note that by Theorem \ref{systhm3}, we obtain the explicit isometry from $F / [(1,0)]$ to $([1,z],|\cdot|_s)$. Therefore, since $\dim_{K^{\lor}} F^{\lor} = 2$ and $\{(1,0),(x,1)\}$ is an orthogonal set, there exist $x_0 \in K^{\lor} \setminus K$ and a linear isometry
  \begin{align*}
    F &\to ([(1,0),(0,1),(x_0, z)], |\cdot|_t \times |\cdot|_s) \\
    (1,0) &\mapsto (1,0) \\
    (x,1) &\mapsto (0,1) \\
    (y,z) &\mapsto (x_0, z).
  \end{align*}
  Finally, we consider $F / [(x, 1)]$. Obviously, we have
  \begin{align*}
    F / [(x, 1)] \cong \big([(\tfrac{1}{x}, 0), (1, 1), (\tfrac{y}{x}, z)], |\cdot|_{t |x|} \times |\cdot|_s \big) \big/ [(1, 1)]
  \end{align*}
  We consider $([(\tfrac{1}{x}, 0), (1, 1), (\tfrac{y}{x}, z)], |\cdot|_{t |x|} \times |\cdot|_s)$ as a subspace of a $4$-dimensional normed space
  \begin{align*}
    E_0 := ([(1,0), (0,1), (\tfrac{1}{x}, 0), (\tfrac{y}{x}, z)], |\cdot|_{t |x|} \times |\cdot|_s).
  \end{align*}
  Let $\pi : E_0 \to E_0 / [(1,1)]$ be the canonical quotient map. By Lemma \ref{lemh4}, the map
  \begin{align*}
    [(1,0),(0,1),(\tfrac{1}{x},0)] / [(1,1)] &\to ([1, \tfrac{1}{x}], |\cdot|_{t |x|}) \\
    \pi((1,0)) &\mapsto 1 \\
    \pi((\tfrac{1}{x}, 0)) &\mapsto \tfrac{1}{x}
  \end{align*}
  gives a linear isometry. Moreover, by Theorem \ref{.4thm3} and the conditions $(2)$ and $(3)$, we have $\tfrac{y}{x} \nsim z$. Hence, by Corollary \ref{.4cor3}, the map
  \begin{align*}
    [(1,0),(0,1),(\tfrac{y}{x},z)] / [(1,1)] &\to ([1, \tfrac{y}{x} - z], |\cdot|_{\tfrac{s \gamma}{\beta}}) \\
    \pi((1,0)) &\mapsto 1 \\
    \pi((\tfrac{y}{x}, z)) &\mapsto \tfrac{y}{x} - z
  \end{align*}
  gives a linear isometry where $\beta = d(\tfrac{y}{x} - z, K)$. Now, by the choice of $x, y$, we have
  \begin{align*}
    d(z,K) = \gamma = \frac{t \delta}{s} < \frac{\delta}{|x|} = d(\frac{y}{x}, K).
  \end{align*}
  Hence, $s \gamma /\beta = t |x|$ and there exists $\mu \in K$ with $|\mu| = |z|$ such that $|z - \mu| < d(\tfrac{y}{x},K)$. Therefore, the map
  \begin{align*}
    [(1,0),(0,1),(\tfrac{y}{x},z)] / [(1,1)] &\to ([1, \tfrac{y}{x}], |\cdot|_{t |x|}) \\
    \pi((1,0)) &\mapsto 1 \\
    \pi((\tfrac{y}{x}, z)) &\mapsto \tfrac{y}{x} - \mu
  \end{align*}
  gives a linear isometry. Since $\tfrac{y}{x} \nsim x$, $E_0 / [(1,1)]$ must be of type $\mathrm{\Rnum{4}}_3$. In particular, $\dim_{K^{\lor}} (E_0 / [(1,1)])^{\lor} = 1$. Finally, by Lemma \ref{syslem1}, there exists a linear isometry
  \begin{align*}
    E_0 / [(1,1)] &\to ([1, \tfrac{1}{x}, \tfrac{y}{x}], |\cdot|_{t |x|}) \\ 
    \pi((1,0)) &\mapsto 1 \\
    \pi((\tfrac{1}{x}, 0)) &\mapsto \tfrac{1}{x} \\
    \pi((\tfrac{y}{x}, z)) &\mapsto \tfrac{y}{x} - \mu.
  \end{align*}
  Consequently, we have a linear isometry
  \begin{align*}
    F / [(x,1)] &\to ([1, y - \mu x], |\cdot|_t) \\
    \pi'((1,0)) &\mapsto 1 \\
    \pi'((y, z)) &\mapsto y - \mu x
  \end{align*}
  where $\pi' : F \to F / [(x,1)]$ is the canonical quotient map. Since 
  \begin{align*}
    d(\mu x, K) = |\mu| d(x,K) = |z| \alpha \le \delta = d(y,K),
  \end{align*}
  we have $y - \mu x \nsim x$, which completes the proof. 
\end{proof}

\begin{thm} \label{thm20}
    Let $x, y, z \in K^{\lor} \setminus K$ and $t, s \in \mathbb{R}_{> 0}$. We set 
  \begin{align*}
    (E, \|\cdot\|) := ([(1,0), (0,1), (x,0), (y,z)],|\cdot|_{t} \times |\cdot|_{s}),
  \end{align*}
  and suppose that $E$ is of type $\mathrm{\Rnum{15}}_4$. Then each $3$-dimensional subspace of $E$, of type $\mathrm{\Rnum{3}}_3$, is isometrically isomorphic to a $3$-dimensional normed space of the form
  \begin{align*}
    ([(1,0),(0,1),(y + \lambda x, z)],|\cdot|_t \times |\cdot|_s) \ \text{with $\lambda \in K$ satisfying $y + \lambda x \nsim x$}.
  \end{align*}
\end{thm}
\begin{proof}
  Let $F$ be a $3$-dimensional subspace of $E$, of type $\mathrm{\Rnum{3}}_3$, and $u, v, w \in F$ be its basis. It suffices to consider the following four cases:
  \begin{itemize}
    \item[$(A)$] $u = (x,0)$, $v = (0,1)$, $w = (1,0)$. 
    \item[$(B)$] $u = (y + \lambda x,z)$, $v = (0,1)$, $w = (1,0)$ $(\lambda \in K)$. 
    \item[$(C)$] $u = (y,z + \lambda)$, $v = (x,\mu)$, $w = (1,0)$ $(\lambda, \mu \in K)$. 
    \item[$(D)$] $u = (y + \lambda,z)$, $v = (x + \mu,0)$, $w = (\eta, 1)$ $(\lambda, \mu, \eta \in K)$.
  \end{itemize}
  In the case $(A)$, obviously $F$ is not of type $\mathrm{\Rnum{3}}_3$. In the case $(B)$, by Theorem \ref{.4thm2} and the proof of Theorem \ref{thm14}, we have $y + \lambda x \nsim x$. \par
  Let us consider the case $(C)$. Since $E$ is of type $\mathrm{\Rnum{15}}_4$, the assumption in Lemma \ref{lem15} is satisfied. Hence, there exists $\mu_0 \in K$ such that $y + \mu_0 x \nsim x$ and  
  \begin{align*}
    F \cong ([(1,0),(0,1),(y + \mu_0 x, z)],|\cdot|_t \times |\cdot|_s). 
  \end{align*}
  Finally, in the case $(D)$, $F$ is isometrically isomorphic to 
  \begin{align*}
    ([(1,0), (\frac{\eta}{x + \mu},1), (\frac{y + \lambda}{x + \mu}, z)], |\cdot|_{t |x + \mu|} \times |\cdot|_{s}).
  \end{align*}
  If $\eta = 0$, then by Theorem \ref{.4thm2} and the proof of Theorem \ref{thm14}, we have $\tfrac{y + \lambda}{x + \mu} \nsim x$. If $\eta \neq 0$, since a $4$-dimensional normed space
  \begin{align*}
     ([(1,0), (0,1), (\frac{1}{x + \mu},1), (\frac{y + \lambda}{x + \mu}, z)], |\cdot|_{t |x + \mu|} \times |\cdot|_{s})
  \end{align*}
  is of type $\mathrm{\Rnum{15}}_4$, we can apply Lemma \ref{lem15}. Therefore, there exists $w \in [1, \tfrac{1}{x + \mu}, \tfrac{y + \lambda}{x + \mu}] \setminus K$ such that $w \nsim x$ and
  \begin{align*}
    F \cong ([(1,0), (0,1), (w, z), |\cdot|_{t |x + \mu|} \times |\cdot|_{s}).
  \end{align*}
  In both cases, by Theorem \ref{.4thm3}, there exist $\lambda_1, \mu_1 \in K$ such that $\lambda_1 y + \mu_1 x \nsim x$ and
  \begin{align*}
     F \cong ([(1,0), (0,1), (\lambda_1 y + \mu_1 x, z), |\cdot|_{t |x + \mu|} \times |\cdot|_{s}).
  \end{align*}
  Since $\lambda_1 y + \mu_1 x \nsim x$, we have 
  \begin{align*}
    d (\lambda_1 y + \mu_1 x, K) = d(\lambda_1 y + \mu_1 x,[1,x]) = |\lambda_1| d(y, [1, x]).
  \end{align*}
  Moreover, by Theorems \ref{.4thm3} and \ref{thm14}, we obtain $|\lambda_1| = |x + \mu|^{-1}$. Thus, $F$ is isometrically isomorphic to
  \begin{align*}
    ([(1,0), (0,1), (y + \lambda_1^{-1} \mu_1 x, z), |\cdot|_{t} \times |\cdot|_{s}),
  \end{align*}
  which completes the proof.
\end{proof}

By Theorem \ref{thm13}, the dual space of a normed space of type $\mathrm{\Rnum{15}}_4$ is also of type $\mathrm{\Rnum{15}}_4$. More precisely, we have the following.

\begin{thm}
  Let $x, y, z \in K^{\lor} \setminus K$ and $t, s \in \mathbb{R}_{> 0}$. We set 
  \begin{align*}
    (E, \|\cdot\|) := ([(1,0), (0,1), (x,0), (y,z)],|\cdot|_{t} \times |\cdot|_{s}),
  \end{align*}
  and put $\alpha = d(x, K)$ and $\gamma = d(z,K)$. Suppose $y \nsim x$ and that $E$ is of type $\mathrm{\Rnum{15}}_4$. Let $e_1, e_2, e_3, e_4 \in E'$ be the dual basis of $(y,z), (x,0), (0,1), (1,0)$. Then the map
  \begin{align*}
    E' &\to ([(1,0), (0,1), (z,0), (y,x)],|\cdot|_{\tfrac{1}{s \gamma}} \times |\cdot|_{\tfrac{1}{t \alpha}}) \\
    e_1 &\mapsto (1,0) \\
    e_2 &\mapsto (0,1) \\
    e_3 &\mapsto - (z,0) \\
    e_4 &\mapsto - (y,x)
  \end{align*}
  gives a linear isometry.
\end{thm}
\begin{proof}
  By Theorem \ref{thmh10} and Lemma \ref{lem9}, there exists a linear isometry
  \begin{align*}
    [e_1, e_2, e_3] &\to ([1,z], |\cdot|_{\tfrac{1}{s \gamma}}) \oplus (K, |\cdot|_{\tfrac{1}{t \alpha}}) \\
    e_1 &\mapsto (1,0) \\
    e_2 &\mapsto (0,1) \\
    e_3 &\mapsto - (z,0).
  \end{align*}
  Since $E'$ is of type $\mathrm{\Rnum{15}}_4$, we have $\dim_{K^{\lor}} (E')^{\lor} = 2$. Therefore, we have a linear isometry
  \begin{align*}
    E' &\to ([(1,0), (0,1), (z,0), (z_0,z_1)],|\cdot|_{\tfrac{1}{s \gamma}} \times |\cdot|_{\tfrac{1}{t \alpha}}) \\
    e_1 &\mapsto (1,0) \\
    e_2 &\mapsto (0,1) \\
    e_3 &\mapsto - (z,0) \\
    e_4 &\mapsto (z_0,z_1)
  \end{align*}
  where $z_0, z_1 \in K^{\lor}$. On the other hand, by Corollary \ref{syscor5} and Lemma \ref{lem10}, the map
  \begin{align*}
    ([(1,0), (0,1), (y, x)],|\cdot|_{\tfrac{1}{t \delta}} \times |\cdot|_{\tfrac{1}{t \alpha}}) &\to ([(1,0), (0,1), (z_0,z_1)],|\cdot|_{\tfrac{1}{s \gamma}} \times |\cdot|_{\tfrac{1}{t \alpha}}) \\
    (1,0) &\mapsto (1,0) \\
    (0,1) &\mapsto (0,1) \\
    (y,x) &\mapsto - (z_0, z_1)
  \end{align*}
  where $\delta = d(y,K)$ gives a linear isometry. Therefore, we have $|z_0 + y| \le d(y,K)$ and $|z_1 + x| \le d(x,K)$. Moreover, by Theorem \ref{thm14}, $y \nsim x$ implies $y \nsim z$. Hence, we obtain
  \begin{align*}
    |z_0 + y| \le d(y, K) = d(y, [1, z]),
  \end{align*}
  which completes the proof.
\end{proof}

The following corollary is an extension of \cite[Theorem 3]{ort}. For details, see \cite{ort}.

\begin{cor}
  Let $E$ be a $4$-dimensional normed space of type $\mathrm{\Rnum{15}}_4$. Then there exists a $2$-dimensional subspace $D$ of $E$ satisfying the following properties:
  \begin{itemize}
    \item $D$ is strict in $E$, that is, $d(u,D)$ is attained for each $u \in E$. 
    \item $D$ is a HB-subspace of $E$, that is, for each $f \in D'$, $f$ has an extension $\tilde{f} \in E'$ with $\|f\| = \|\tilde{f}\|$.
    \item $D$ is not orthocomplemented in $E$, that is, there exists no subspace $D_1$ such that $E = D \oplus D_1$.
  \end{itemize}
\end{cor}
\begin{proof}
  Let $x, y, z \in K^{\lor} \setminus K$ and $t, s \in \mathbb{R}_{> 0}$ be such that $x \nsim y$ and
  \begin{align*}
    E \cong ([(1,0), (0,1), (x,0), (y,z)],|\cdot|_{t} \times |\cdot|_{s}).
  \end{align*}
  Put $\gamma = d(z,K)$ and $\delta = d(y, [1,x])$. Set $D := [(1,0), (x,0)]$. Since $E$ is indecomposable, $D$ is not orthocomplemented in $E$. We prove that $D$ is strict in $E$. Indeed, let $u \in E$. Without loss of generality, we may assume that $u$ is of the form $u = \lambda (y, z) + (0,\mu)$ where $\lambda, \mu \in K$. Since $d(u,D) = |\lambda z + \mu|_s$ and $|\lambda z + \mu|_s > |\lambda|_s \gamma$, we see that $d(u,D)$ is attained. Here, we use $t \delta = s \gamma$. \par
  Finally, we prove that $D$ is a HB-subspace of $E$. Let $\pi: E \to E / [(0,1)]$ be the canonical quotient map. Obviously, the restriction of $\pi$ to $D$ is an isometry. Thus, it suffices to prove a $2$-dimensional subspace $\pi(D)$ is a HB-subspace of $E / [(0,1)]$. On the other hand, by Lemma \ref{lem10}, $E / [(0,1)]$ is of type $\mathrm{\Rnum{4}}_3$. Thus, $(E / [(0,1)])'$ is of type $\mathrm{\Rnum{3}}_3$ and each $2$-dimensional subspace of $(E / [(0,1)])'$ has an orthogonal base. Let $f \in (\pi(D))'$ and $g \in (E / [(0,1)])'$ be an extension of $f$. Since $\pi(D)^{\perp}$ is one-dimensional, $d(g,\pi(D)^{\perp})$ is attained. Therefore, $\pi(D)$ is a HB-subspace of $E / [(0,1)]$, which completes the proof.
\end{proof}

A structure theorem for normed spaces of type $\mathrm{\Rnum{15}}_4$ is unknown. We expect that the following conjecture holds.

\begin{conj}
  Let $x, x_1, y, y_1, z, z_1 \in K^{\lor} \setminus K$ and $t, t_1, s, s_1 \in \mathbb{R}_{> 0}$. We set 
  \begin{align*}
    &(E, \|\cdot\|) := ([(1,0), (0,1), (x,0), (y,z)],|\cdot|_{t} \times |\cdot|_{s}), \\
    &(E_1, \|\cdot\|) := ([(1,0), (0,1), (x_1,0), (y_1,z_1)],|\cdot|_{t_1} \times |\cdot|_{s_1}).
  \end{align*}
  Suppose that both $E$ and $E_1$ are of type $\mathrm{\Rnum{15}}_4$. Then $E \cong E_1$ if and only if the following three conditions are satisfied: \\
  $(1)$ $t / t_1, s / s_1 \in V_K$. \\
  $(2)$ $x \sim x_1$, $z \sim z_1$. \\
  $(3)$ $([1, x, y], |\cdot|) \cong ([1, x_1, y_1], |\cdot|)$.
\end{conj}

\subsection{Type $\mathrm{\protect \Rnum{16}}_4, \mathrm{\protect \Rnum{17}}_4$}

Finally, we will study $4$-dimensional normed spaces of type $\mathrm{\Rnum{16}}_4$ or $\mathrm{\Rnum{17}}_4$. Our key lemma is Lemma \ref{lem13} which allows us to classify $4$-dimensional normed spaces. 

\begin{lem} \label{lem14}
  Let $E$ be a $4$-dimensional normed space of type $\mathrm{\Rnum{16}}_4$. Then $E'$ is of type $\mathrm{\Rnum{11}}_4$. 
\end{lem}
\begin{proof}
  By definition, we have $\dim_{K^{\lor}} (E')^{\lor} = 1$. Then by Theorem \ref{thm10}, $E'$ must be of type $\mathrm{\Rnum{11}}_4$.
\end{proof}

\begin{lem} \label{lem13}
  Let $E$ be an indecomposable $4$-dimensional normed space such that $\dim_{K^{\lor}} E^{\lor} = 2$ and all $3$-dimensional subspaces of $E$ are of type $\mathrm{\Rnum{3}}_3$. Then the following are equivalent: \\
  $(1)$ $\dim_{K^{\lor}} (E / F)^{\lor} = 1$ for each $2$-dimensional subspace of $E$. \\
  $(2)$ $\dim_{K^{\lor}} (E / F)^{\lor} = 1$ for some $2$-dimensional subspace of $E$.
\end{lem}
\begin{proof}
  Suppose $(2)$. We shall prove $\dim_{K^{\lor}} (E')^{\lor} = 1$. By assumption, $E'$ contains a $2$-dimensional subspace without an orthogonal base. Therefore, if $\dim_{K^{\lor}} (E')^{\lor} = 2$, then by Lemma \ref{lem12}, $E'$ is either of type $\mathrm{\Rnum{14}}_4$ or $\mathrm{\Rnum{15}}_4$. If $E'$ is of type $\mathrm{\Rnum{14}}_4$, then it follows from Corollary \ref{cor2} that $E''$ is of type $\mathrm{\Rnum{9}}_4$. In particular, we have $\dim_{K^{\lor}} E^{\lor} = 1$, which is a contradiction. If $E'$ is of type $\mathrm{\Rnum{15}}_4$, then by Theorem \ref{thm13}, $E''$ is of type $\mathrm{\Rnum{15}}_4$. Hence, $E$ contains a $3$-dimensional subspace of type $\mathrm{\Rnum{2}}_3$. Thus, we derive a contradiction. Now, suppose  $\dim_{K^{\lor}} (E')^{\lor} = 3$. Then $E'$ is of type $\mathrm{\Rnum{13}}_4$ and $\dim_{K^{\lor}} E^{\lor} = 1$ by Corollary \ref{cor3}, which is a contradiction. Thus, we have $\dim_{K^{\lor}} E^{\lor} = 1$ and obtain $(1)$.
\end{proof}

By Lemmas \ref{lem12} and \ref{lem13}, we obtain the following theorem.

\begin{thm} \label{thm16}
  Let $E$ be an indecomposable $4$-dimensional normed space such that $\dim_{K^{\lor}} E^{\lor} = 2$. Then $E$ is exactly of one type of type $\mathrm{\Rnum{14}}_4 \sim \mathrm{\Rnum{17}}_4$.
\end{thm}

Now, we see that the types $\mathrm{\Rnum{1}}_4$ through $\mathrm{\Rnum{17}}_4$ completely classify $4$-dimensional normed spaces.

\begin{thm} \label{thm19}
  Let $E$ be a $4$-dimensional normed space. Then $E$ is exactly of one type of type $\mathrm{\Rnum{1}}_4 \sim \mathrm{\Rnum{17}}_4$.
\end{thm}
\begin{proof}
  We can use Theorems \ref{thm15}, \ref{thm10} and \ref{thm16}. 
\end{proof}

\begin{thm}
  Let $E$ be a $4$-dimensional normed space of type $\mathrm{\Rnum{17}}_4$. Then $E'$ is of type $\mathrm{\Rnum{17}}_4$.
\end{thm}
\begin{proof}
  By definition, each $2$-dimensional subspace of $E'$ has an orthogonal base. In particular, we have $\dim_{K^{\lor}} E^{\lor} \ge 2$. By the preceding theorem, $E'$ must be exactly of one type of type $\mathrm{\Rnum{13}}_4$, $\mathrm{\Rnum{16}}_4$ or $\mathrm{\Rnum{17}}_4$. Now, by Corollary \ref{cor3} and Lemma \ref{lem14}, $E'$ must be of type $\mathrm{\Rnum{17}}_4$.
\end{proof}

By the same proof as that of the above theorem, we have the following.

\begin{thm}
  Let $E$ be an indecomposable $4$-dimensional normed space with $\dim_{K^{\lor}} E^{\lor} = 2$. Then $E$ is of type $\mathrm{\Rnum{17}}_4$ if and only if $\dim_{K^{\lor}} (E / F)^{\lor} = 2$ for each $2$-dimensional subspace $F$ of $E$.
\end{thm}

Now, we can identify the dual spaces of $4$-dimensional normed spaces according to the types.

\begin{thm} \label{thm17}
  Let $E$ be a $4$-dimensional normed space of type $\mathrm{\Rnum{1}}_4$ (resp. $\mathrm{\Rnum{2}}_4$, $\mathrm{\Rnum{3}}_4$, $\mathrm{\Rnum{4}}_4$, $\mathrm{\Rnum{5}}_4$, $\mathrm{\Rnum{6}}_4$, $\mathrm{\Rnum{7}}_4$, $\mathrm{\Rnum{8}}_4$, $\mathrm{\Rnum{9}}_4$, $\mathrm{\Rnum{10}}_4$, $\mathrm{\Rnum{11}}_4$, $\mathrm{\Rnum{12}}_4$, $\mathrm{\Rnum{13}}_4$, $\mathrm{\Rnum{14}}_4$, $\mathrm{\Rnum{15}}_4$, $\mathrm{\Rnum{16}}_4$, $\mathrm{\Rnum{17}}_4$). Then $E'$ is of type $\mathrm{\Rnum{1}}_4$ (resp. $\mathrm{\Rnum{2}}_4$, $\mathrm{\Rnum{4}}_4$, $\mathrm{\Rnum{3}}_4$, $\mathrm{\Rnum{5}}_4$, $\mathrm{\Rnum{6}}_4$, $\mathrm{\Rnum{7}}_4$, $\mathrm{\Rnum{13}}_4$, $\mathrm{\Rnum{14}}_4$, $\mathrm{\Rnum{10}}_4$, $\mathrm{\Rnum{16}}_4$, $\mathrm{\Rnum{12}}_4$, $\mathrm{\Rnum{8}}_4$, $\mathrm{\Rnum{9}}_4$, $\mathrm{\Rnum{15}}_4$, $\mathrm{\Rnum{11}}_4$, $\mathrm{\Rnum{17}}_4$).
\end{thm}

Let $E$ be a $4$-dimensional normed space with $\dim_{K^{\lor}} E^{\lor} = 2$. Then there exist $x, y, z, w \in K^{\lor} \setminus K$ and $t, s \in \mathbb{R}_{> 0}$ for which
\begin{align*}
  E \cong ([(1,0),(0,1),(x,z),(y,w)],|\cdot|_t \times |\cdot|_s).
\end{align*}
First, we will study whether $E$ contains a $3$-dimensional subspace of type $\mathrm{\Rnum{2}}_3$ or not.

\begin{lem}
  Let $x, y, z, w \in K^{\lor} \setminus K$ and $t, s \in \mathbb{R}_{> 0}$. We set
\begin{align*}
  E := ([(1,0),(0,1),(x,z),(y,w)],|\cdot|_t \times |\cdot|_s).
\end{align*}
Suppose $\dim_K E = 4$. If $E$ contains no $3$-dimensional subspace of type $\mathrm{\Rnum{2}}_3$, then $E$ satisfies the following three conditions:
\begin{itemize}
  \item $\dim_K [1,x,y] = 3$. 
  \item $t \cdot d(x + \lambda y, K) = s \cdot d(z + \lambda w, K)$ for each $\lambda \in K$.
  \item $x + \lambda y \nsim z + \lambda w$ for each $\lambda \in K$.
\end{itemize}
\end{lem}
\begin{proof}
  We can use Theorem \ref{.4thm2}.
\end{proof}

\begin{thm} \label{thm18}
   Let $x, y, z, w \in K^{\lor} \setminus K$ and $t, s \in \mathbb{R}_{> 0}$. We set
\begin{align*}
  (E, \|\cdot\|) := ([(1,0),(0,1),(x,z),(y,w)],|\cdot|_t \times |\cdot|_s).
\end{align*} 
\noindent $(1)$ $E$ is of type $\mathrm{\Rnum{16}}_4$ if and only if $E$ satisfies the following three conditions:
\begin{itemize}
  \item[$(a)$] $([1,x,y],|\cdot|_t)$ is of type $\mathrm{\Rnum{5}}_3$. 
  \item[$(b)$] $x \nsim z$.
  \item[$(c)$] $t \cdot d(x + \lambda y, K) = s \cdot d(z + \lambda w, K)$ for each $\lambda \in K$.
\end{itemize}
$(2)$ Suppose that $E$ is of type $\mathrm{\Rnum{16}}_4$. Then each $3$-dimensional subspace of $E$ is isometrically isomorphic to
\begin{align*}
  ([(1,0),(0,1),(x,z)],|\cdot|_t \times |\cdot|_s).
\end{align*}
\noindent $(3)$ Suppose that $E$ is of type $\mathrm{\Rnum{16}}_4$. Then $E$ satisfies $\mathrm{(SE)}$ if and only if
\begin{align*}
  t, s, t d(x,K) \notin V_K.
\end{align*}
\end{thm}
\begin{proof}
  First, we shall prove $(1)$. Suppose that $E$ is of type $\mathrm{\Rnum{16}}_4$. Then $E$ contains no $3$-dimensional subspace of type $\mathrm{\Rnum{2}}_3$. Therefore, by the preceding lemma, $E$ satisfies the conditions $(b)$ and $(c)$. By the condition $(c)$, it is easily seen that $E / [(0,1)] \cong ([1,x,y],|\cdot|_t)$. On the other hand, by Theorem \ref{thm17}, $E'$ is of type $\mathrm{\Rnum{11}}_4$ and $\dim_{K^{\lor}} (E')^{\lor} = 1$. In particular, we have $\dim_{K^{\lor}} (([1,x,y],|\cdot|_t)')^{\lor} = 1$. Therefore, $([1,x,y],|\cdot|_t)$ must be of type $\mathrm{\Rnum{5}}_3$. Thus, $E$ satisfies the condition $(a)$. \par
  Conversely, suppose the conditions $(a), (b)$ and $(c)$ hold. By the conditions $(a)$ and $(c)$, we have that $E / [(0,1)] \cong ([1,x,y],|\cdot|_t)$ is of type $\mathrm{\Rnum{5}}_3$. Therefore, $E'$ contains a $3$-dimensional subspace of type $\mathrm{\Rnum{5}}_3$. Moreover, by the conditions $(b)$ and $(c)$, we have $x \nsim z$ and $td(x,K) = s d(z,K)$. Thus, a $3$-dimensional subspace $[(1,0),(0,1),(x,z)]$ of $E$ is of type $\mathrm{\Rnum{3}}_3$. Now, let us prove that $E$ is indecomposable. Suppose that $E$ is decomposable. Since $\dim_{K^{\lor}} E^{\lor} = 2$ and $E$ contains a $3$-dimensional subspace of type $\mathrm{\Rnum{3}}_3$, $E$ is of type $\mathrm{\Rnum{6}}_4$. Therefore, $E'$ is of type $\mathrm{\Rnum{6}}_4$ and contains no $3$-dimensional subspace of type $\mathrm{\Rnum{5}}_3$, which is a contradiction. Thus, $E$ is indecomposable. \par
  Finally, we determine the type of $E$. Since $E$ contains a $3$-dimensional subspace of type $\mathrm{\Rnum{3}}_3$, $E$ is exactly of one type of type $\mathrm{\Rnum{15}}_4$, $\mathrm{\Rnum{16}}_4$ or $\mathrm{\Rnum{17}}_4$. On the other hand, since $E'$ contains a $3$-dimensional subspace of type $\mathrm{\Rnum{5}}_3$, it follows from Lemma \ref{lem12} and Theorem \ref{thm17} that $E$ must be $\mathrm{\Rnum{16}}_4$. Thus, $(1)$ is proved. Moreover, $(2)$ immediately follows from Corollary \ref{syscor1} and Theorem \ref{thm17}. \par
  Now, we shall prove $(3)$. Since $\dim_{K^{\lor}} (E')^{\lor} = 1$, we can apply Corollary \ref{syscor3}. Hence, $E$ satisfies (SE) if and only if $\|F'\| \cap V_K = \emptyset$ and $F$ satisfies (SE) where $F = ([(1,0),(0,1),(x,z)],|\cdot|_t \times |\cdot|_s)$. Thus, by Theorem \ref{.4thm2}, $E$ satisfies (SE) if and only if $t, s, t d(x,K) \notin V_K$.
\end{proof}

Finally, we give a necessary and sufficient condition for a $4$-dimensional normed space of type $\mathrm{\Rnum{17}}_4$ to satisfy (SE). Thus, we have given a necessary and sufficient condition for a $4$-dimensional normed space to satisfy (SE) according to an embedding into $((K^{\lor})^n, |\cdot|_{t_1} \times \cdots \times |\cdot|_{t_n})$.

\begin{thm}
  Let $(E, \|\cdot\|)$ be a $4$-dimensional normed space of type $\mathrm{\Rnum{17}}_4$. Then $E$ satisfies $\mathrm{(SE)}$ if and only if $\|E\| \cap V_K = \emptyset$.
\end{thm}
\begin{proof}
  We will use Lemma \ref{selem1}. Since $E'$ is of type $\mathrm{\Rnum{17}}_4$ by Theorem \ref{thm17}, each $2$-dimensional subspace of $E'$ has an orthogonal base. Hence, $E$ satisfies (SE) if and only if each $3$-dimensional subspace of $E$ satisfies (SE). Moreover, since $E$ is of type $\mathrm{\Rnum{17}}_4$, each $3$-dimensional subspace of $E$ is of type $\mathrm{\Rnum{3}}_3$. Thus, if $\|E\| \cap V_K = \emptyset$, by Theorem \ref{.4thm2}, $E$ satisfies (SE). Conversely, suppose that $E$ satisfies (SE). Let $F$ be a $3$-dimensional subspace of $E$. Then $F$ satisfies (SE). Moreover, $F$ is of type $\mathrm{\Rnum{3}}_3$ and $\dim_{K^{\lor}} F^{\lor} = 2$. Hence, $E$ is an immediate extension of $F$ and $\|E\| = \|F\|$. Thus, by Theorem \ref{.4thm2}, we have $\|E\| \cap V_K = \emptyset$.
\end{proof}

Unfortunately, we do not know a necessary and sufficient condition for a subspace of $(K^{\lor})^2$ to be of type $\mathrm{\Rnum{17}}_4$, that is, the following problem is unsolved.

\begin{prob*} \label{prob2}
 Let $x, y, z, w \in K^{\lor} \setminus K$ and $t, s \in \mathbb{R}_{> 0}$. We set
\begin{align*}
  (E, \|\cdot\|) := ([(1,0),(0,1),(x,z),(y,w)],|\cdot|_t \times |\cdot|_s).
\end{align*} 
What is a necessary and sufficient condition for $E$ to be of type $\mathrm{\Rnum{17}}_4$?
\end{prob*}

In the next section, we will give a sufficient condition for $E$ to be of type $\mathrm{\Rnum{17}}_4$.

\subsection{Existence problem} In the same way as Proposition \ref{.4prop1}, we have the following.

\begin{prop} \label{prop2}
  Suppose that $K$ is separable as a metric space. Then for each $\alpha = \mathrm{\Rnum{1}}_4, \cdots, \mathrm{\Rnum{16}}_4$, there exists a $4$-dimensional normed space of type $\alpha$.
\end{prop}
\begin{proof}
  Since $K^{\lor} / K$ is spherically complete, by Theorem \ref{systhm5}, $K^{\lor} / K$ contains a subspace, isometrically isomorphic to $((K^{\lor})^3, |\cdot| \times |\cdot| \times |\cdot|)$ (as a $K$-normed space). Therefore, for each $\alpha = \mathrm{\Rnum{8}}_4, \cdots, \mathrm{\Rnum{12}}_4$, there exists a $4$-dimensional normed space of type $\alpha$. For the others, use Theorem \ref{systhm5} and Proposition \ref{.4prop1} if necessary.
\end{proof}

For the existence of a normed space of type $\mathrm{\Rnum{17}}_4$, the following theorem is needed.

\begin{thm}
  Let $x, y, z, w \in K^{\lor} \setminus K$ and $t, s \in \mathbb{R}_{> 0}$. We set
\begin{align*}
  (E, \|\cdot\|) := ([(1,0),(0,1),(x,z),(y,w)],|\cdot|_t \times |\cdot|_s).
\end{align*} 
Suppose that the following two conditions hold: \\
$(1)$ $t d(x,K) = s d(z,K)$ and $t d(y,K) = s d(w,K)$. \\
$(2)$ $\{\pi(x), \pi(y), \pi(z), \pi(w)\}$ is an orthogonal set where $\pi : K^{\lor} \to K^{\lor} / K$ is the canonical quotient map. \par
Then the normed space $E$ is of type $\mathrm{\Rnum{17}}_4$.
\end{thm}
\begin{proof}
  By assumption $E$ contains two $3$-dimensional subspaces of $\mathrm{\Rnum{3}}_3$
  \begin{align*}
    [(1,0),(0,1),(x,z)] \ \text{and} \ [(1,0),(0,1),(y,w)].
  \end{align*}
  Therefore, since $\dim_{K^{\lor}} E^{\lor} = 2$, $E$ is of type $\mathrm{\Rnum{6}}_4$, $\mathrm{\Rnum{15}}_4$, $\mathrm{\Rnum{16}}_4$ or $\mathrm{\Rnum{17}}_4$. For instance, we will prove that $E$ is not of type $\mathrm{\Rnum{15}}_4$. If $E$ is of type $\mathrm{\Rnum{15}}_4$, then there exist $x', y', z' \in K^{\lor} \setminus K$ and $t', s' \in \mathbb{R}_{> 0}$ for which
  \begin{align*}
    E \cong ([(1,0),(0,1),(x',0),(y',z')],|\cdot|_{t'} \times |\cdot|_{s'}).
  \end{align*}
  Then by Remark \ref{rem3} and Theorem \ref{thm20}, there exist $x_1, y_1, z_1, w_1 \in [1, x', y', z']$ such that
  \begin{align*}
    x \sim x_1, y \sim y_1, z \sim z_1 \ \text{and} \ w \sim w_1.
  \end{align*}
  On the other hand, by assumption, we have $\dim_{K^{\lor}} [\pi(x), \pi(y), \pi(z), \pi(w)]^{\lor} = 4$. Therefore, we obtain $\dim_{K^{\lor}} [\pi(x'), \pi(y'), \pi(z')]^{\lor} \ge 4$, which is a contradiction. Hence, $E$ is not of type $\mathrm{\Rnum{15}}_4$. For the same reason as above, using Lemma \ref{lem6} and Theorem \ref{thm18}, $E$ is neither of type $\mathrm{\Rnum{6}}_4$ nor $\mathrm{\Rnum{16}}_4$. Thus, $E$ must be of type $\mathrm{\Rnum{17}}_4$.
\end{proof}

Now, by Theorem \ref{systhm5}, we have the following.

\begin{thm} \label{thm21}
  Suppose that $K$ is separable as a metric space. Then there exists a $4$-dimensional normed space of type $\mathrm{\Rnum{17}}_4$.
\end{thm}

\section{Appendix} \label{appe}

In this appendix, we summarize the types introduced in this paper. By Proposition \ref{seprop1}, it is natural to ask the question: 
\begin{center}
  Does a normed space $(E, \|\cdot\|)$ with $\|E\| \cap V_K = \emptyset$ satisfy (SE)?
\end{center}
In the following tables, "$\|E\| \cap V_K = \emptyset \Rightarrow \mathrm{(SE)}$" means that given one type, each normed space $E$ of that type with $\|E\| \cap V_K = \emptyset$ satisfies (SE). \par
First, we summarize the data on $\dim_{K^{\lor}} E^{\lor}$, the decomposability, the dual spaces and the condition (SE).

\begin{table}[htbp]
    \centering
    \caption{Types of $3$-dimensional normed spaces E}
    \begin{tabular}{|c||c|c|c|c|}
        \hline
        & $\dim_{K^{\lor}} E^{\lor}$ & Decomposability & Dual space &   $\|E\| \cap V_K = \emptyset \Rightarrow \mathrm{(SE)}$ \\ \hline \hline
        $\mathrm{\Rnum{1}}_3$ & 3 & decomposable & $\mathrm{\Rnum{1}}_3$ & true \\ \hline
        $\mathrm{\Rnum{2}}_3$ & 2 & decomposable & $\mathrm{\Rnum{2}}_3$ & true \\ \hline
        $\mathrm{\Rnum{3}}_3$ & 2 & indecomposable & $\mathrm{\Rnum{4}}_3$ & true \\ \hline
        $\mathrm{\Rnum{4}}_3$ & 1 & indecomposable & $\mathrm{\Rnum{3}}_3$ & true \\ \hline
        $\mathrm{\Rnum{5}}_3$ & 1 & indecomposable & $\mathrm{\Rnum{5}}_3$ & false \\ \hline
    \end{tabular}
\end{table}

\begin{table}[htbp]
    \centering
    \caption{Types of $4$-dimensional normed spaces E}
    \begin{tabular}{|c||c|c|c|c|}
        \hline
        & $\dim_{K^{\lor}} E^{\lor}$ & Decomposability & Dual space &   $\|E\| \cap V_K = \emptyset \Rightarrow \mathrm{(SE)}$ \\ \hline \hline
        $\mathrm{\Rnum{1}}_4$ & 4 & decomposable & $\mathrm{\Rnum{1}}_4$ & true \\ \hline
        $\mathrm{\Rnum{2}}_4$ & 3 & decomposable & $\mathrm{\Rnum{2}}_4$ & true \\ \hline
        $\mathrm{\Rnum{3}}_4$ & 3 & decomposable & $\mathrm{\Rnum{4}}_4$ & true \\ \hline
        $\mathrm{\Rnum{4}}_4$ & 2 & decomposable & $\mathrm{\Rnum{3}}_4$ & true \\ \hline
        $\mathrm{\Rnum{5}}_4$ & 2 & decomposable & $\mathrm{\Rnum{5}}_4$ & false \\ \hline
        $\mathrm{\Rnum{6}}_4$ & 2 & decomposable & $\mathrm{\Rnum{6}}_4$ & true \\ \hline
        $\mathrm{\Rnum{7}}_4$ & 2 & decomposable & $\mathrm{\Rnum{7}}_4$ & true \\ \hline
        $\mathrm{\Rnum{8}}_4$ & 1 & indecomposable & $\mathrm{\Rnum{13}}_4$ & true \\ \hline 
        $\mathrm{\Rnum{9}}_4$ & 1 & indecomposable & $\mathrm{\Rnum{14}}_4$ & false \\ \hline
        $\mathrm{\Rnum{10}}_4$ & 1 & indecomposable & $\mathrm{\Rnum{10}}_4$ & false \\ \hline
        $\mathrm{\Rnum{11}}_4$ & 1 & indecomposable & $\mathrm{\Rnum{16}}_4$ & false \\ \hline
        $\mathrm{\Rnum{12}}_4$ & 1 & indecomposable & $\mathrm{\Rnum{12}}_4$ & false \\ \hline
        $\mathrm{\Rnum{13}}_4$ & 3 & indecomposable & $\mathrm{\Rnum{8}}_4$ & true \\ \hline
        $\mathrm{\Rnum{14}}_4$ & 2 & indecomposable & $\mathrm{\Rnum{9}}_4$ & false \\ \hline
        $\mathrm{\Rnum{15}}_4$ & 2 & indecomposable & $\mathrm{\Rnum{15}}_4$ & true \\ \hline
        $\mathrm{\Rnum{16}}_4$ & 2 & indecomposable & $\mathrm{\Rnum{11}}_4$ & false \\ \hline
        $\mathrm{\Rnum{17}}_4$ & 2 & indecomposable & $\mathrm{\Rnum{17}}_4$ & true \\ \hline
    \end{tabular}
\end{table}

Next, we summarize the characteristic properties of $4$-dimensional normed spaces with respect to their types.

\begin{table}[htbp]
    \centering
    \caption{Decomposable normed spaces}
    \begin{tabular}{|c||l|c|}
        \hline
        & Structure & $3$-dimensional subspaces \\ \hline \hline
        $\mathrm{\Rnum{1}}_4$ & $K \oplus \mathrm{\Rnum{1}}_3$ & $\mathrm{\Rnum{1}}_3$ \\ \hline
        $\mathrm{\Rnum{2}}_4$ & $K \oplus \mathrm{\Rnum{2}}_3$ & $\mathrm{\Rnum{1}}_3$, $\mathrm{\Rnum{2}}_3$ \\ \hline
        $\mathrm{\Rnum{3}}_4$ & $K \oplus \mathrm{\Rnum{3}}_3$ & $\mathrm{\Rnum{1}}_3$, $\mathrm{\Rnum{3}}_3$ \\ \hline
        $\mathrm{\Rnum{4}}_4$ & $K \oplus \mathrm{\Rnum{4}}_3$ & $\mathrm{\Rnum{2}}_3$, $\mathrm{\Rnum{4}}_3$ \\ \hline
        $\mathrm{\Rnum{5}}_4$ & $K \oplus \mathrm{\Rnum{5}}_3$ & $\mathrm{\Rnum{2}}_3$, $\mathrm{\Rnum{5}}_3$ \\ \hline
        $\mathrm{\Rnum{6}}_4$ & $\mathrm{\Rnum{2}}_3 \oplus \mathrm{\Rnum{2}}_3$ & $\mathrm{\Rnum{2}}_3$, $\mathrm{\Rnum{3}}_3$ \\ \hline
        $\mathrm{\Rnum{7}}_4$ & $\mathrm{\Rnum{2}}_3 \oplus \mathrm{\Rnum{2}}_3$ & $\mathrm{\Rnum{2}}_3$ \\ \hline
    \end{tabular}
\end{table}

\begin{table}[htbp]
    \centering
    \caption{$\dim_{K^{\lor}} E^{\lor} = 1$}
    \begin{tabular}{|c||c|}
        \hline
        & $E / [u]$ $(u \neq 0)$  \\ \hline \hline
        $\mathrm{\Rnum{8}}_4$ & $\mathrm{\Rnum{1}}_3$ \\ \hline
        $\mathrm{\Rnum{9}}_4$ & $\mathrm{\Rnum{2}}_3$ \\ \hline
        $\mathrm{\Rnum{10}}_4$ & $\mathrm{\Rnum{3}}_3$ \\ \hline
        $\mathrm{\Rnum{11}}_4$ & $\mathrm{\Rnum{4}}_3$ \\ \hline
        $\mathrm{\Rnum{12}}_4$ & $\mathrm{\Rnum{5}}_3$ \\ \hline
    \end{tabular}
\end{table}

\begin{table}[htbp]
    \centering
    \caption{Indecomposable normed spaces $E$ with $\dim_{K^{\lor}} E^{\lor} = 2$}
    \begin{tabular}{|c||c|c|}
        \hline
        & $3$-dimensional subspaces &  \\ \hline \hline
        $\mathrm{\Rnum{14}}_4$ & $\mathrm{\Rnum{2}}_3$ & \\ \hline
        $\mathrm{\Rnum{15}}_4$ & $\mathrm{\Rnum{2}}_3$, $\mathrm{\Rnum{3}}_3$ & \\ \hline
        $\mathrm{\Rnum{16}}_4$ & $\mathrm{\Rnum{3}}_3$ & $\dim_{K^{\lor}} (E / F)^{\lor} = 1$ for all $F \subseteq E$ with $\dim_K F = 2$ \\ \hline
        $\mathrm{\Rnum{17}}_4$ & $\mathrm{\Rnum{3}}_3$ & $\dim_{K^{\lor}} (E / F)^{\lor} = 2$ for all $F \subseteq E$ with $\dim_K F = 2$ \\ \hline
    \end{tabular}
\end{table}

\section{Acknowledgment}

This work was supported by JST SPRING, Grant Number JPMJSP2114.

\end{document}